\documentclass[12pt]{amsart}
\usepackage{ amsmath, amsthm, amsfonts, amssymb, color}
 \usepackage{mathrsfs}
\usepackage{amsfonts, amsmath}
 \usepackage{amsmath,amstext,amsthm,amssymb,amsxtra}
 \usepackage{txfonts} 
 \usepackage[colorlinks, citecolor=blue,pagebackref,hypertexnames=false]{hyperref}
 \allowdisplaybreaks
 \usepackage{pgf,tikz}
 \usepackage{multirow}
 \usepackage{diagbox} 
\usepackage{ulem}
 \usepackage{tikz}

\usetikzlibrary{decorations.pathreplacing}

\DeclareMathOperator*{\esssup}{ess\ sup}
\begin{document}

 \baselineskip 16.6pt
\hfuzz=6pt

\widowpenalty=10000

\newtheorem{cl}{Claim}
\newtheorem{theorem}{Theorem}[section]
\newtheorem{proposition}[theorem]{Proposition}
\newtheorem{coro}[theorem]{Corollary}
\newtheorem{lemma}[theorem]{Lemma}
\newtheorem{definition}[theorem]{Definition}
\newtheorem{assum}{Assumption}[section]
\newtheorem{example}[theorem]{Example}
\newtheorem{remark}[theorem]{Remark}
\renewcommand{\theequation}
{\thesection.\arabic{equation}}

\def\SL{\sqrt H}

\newcommand{\mar}[1]{{\marginpar{\sffamily{\scriptsize
        #1}}}}

\newcommand{\as}[1]{{\mar{AS:#1}}}

\newcommand\R{\mathbb{R}}
\newcommand\RR{\mathbb{R}}
\newcommand\CC{\mathbb{C}}
\newcommand\NN{\mathbb{N}}
\newcommand\ZZ{\mathbb{Z}}
\newcommand\HH{\mathbb{H}}
\def\RN {\mathbb{R}^n}
\renewcommand\Re{\operatorname{Re}}
\renewcommand\Im{\operatorname{Im}}

\newcommand{\mc}{\mathcal}
\newcommand\D{\mathcal{D}}
\def\hs{\hspace{0.33cm}}
\newcommand{\la}{\alpha}
\def \l {\alpha}
\newcommand{\eps}{\tau}
\newcommand{\pl}{\partial}
\newcommand{\supp}{{\rm supp}{\hspace{.05cm}}}
\newcommand{\x}{\times}
\newcommand{\lag}{\langle}
\newcommand{\rag}{\rangle}

\newcommand\wrt{\,{\rm d}}

\title[]{$L^{p}$ estimates and weighted estimates of fractional maximal rough singular integrals on homogeneous groups}

\author{Yanping Chen}
\author{Zhijie Fan$^*$}\thanks{$*$ Corresponding Author}
\author{Ji Li}

\address{Yanping Chen, School of Mathematics and Physics, University of Science and Technology Beijing, Beijing 100083, China}
\email{yanpingch@ustb.edu.cn}

\address{Zhijie Fan, School of Mathematics and Statistics, Wuhan University, Wuhan 430072, China}
\email{ZhijieFan@whu.edu.cn}

\address{Ji Li, Department of Mathematics, Macquarie University, NSW, 2109, Australia}
\email{ji.li@mq.edu.au}

  \date{\today}


\begin{abstract}
In this paper, we study the $L^{p}$ boundedness and $L^{p}(w)$ boundedness ($1<p<\infty$ and $w$ a Muckenhoupt $A_{p}$ weight) of fractional maximal singular integral operators $T_{\Omega,\alpha}^{\#}$ with homogeneous convolution kernel $\Omega(x)$ on an arbitrary homogeneous group $\mathbb H$ of dimension $\mathbb{Q}$.  We show that if $0<\alpha<\mathbb{Q}$, $\Omega\in L^{1}(\Sigma)$ and satisfies the cancellation condition of order $[\alpha]$, then for any $1<p<\infty$,
\begin{align*}
\|T_{\Omega,\alpha}^{\#}f\|_{L^{p}(\mathbb{H})}\lesssim\|\Omega\|_{L^{1}(\Sigma)}\|f\|_{L_{\alpha}^{p}(\mathbb{H})}.
\end{align*}%
where for the case $\alpha=0$, the $L^p$ boundedness of rough singular integral operator and its maximal operator were studied by Tao (\cite{Tao}) and Sato (\cite{sato}), respectively.

We also obtain a quantitative weighted bound for these operators. To be specific, if $0\leq\alpha<\mathbb{Q}$ and $\Omega$ satisfies the same cancellation condition but a stronger condition that $\Omega\in L^{q}(\Sigma)$ for some $q>\mathbb{Q}/\alpha$, then for any $1<p<\infty$ and $w\in A_{p}$,
\begin{align*}
\|T_{\Omega,\alpha}^{\#}f\|_{L^{p}(w)}\lesssim\|\Omega\|_{L^{q}(\Sigma)}\{w\}_{A_p}(w)_{A_p}\|f\|_{L_{\alpha}^{p}(w)},\ \ 1<p<\infty.
\end{align*}
\end{abstract}

\maketitle


\section{Introduction}
\subsection{Background}
Throughout this paper, we regard $\mathbb{H}=\mathbb{R} ^{n}$ $(n\geq 2)$ as a homogeneous group, which is a nilpotent Liegroup. It has multiplication, inverse, dilation, and norm structures
\begin{align*}
(x,y)\mapsto xy, \ \ x\mapsto x^{-1},\ \ (\lambda,x)\mapsto \lambda\circ x,\ \ x\mapsto \rho(x)
\end{align*}
for $x,y\in\HH$, $\lambda>0$. The multiplication and inverse operations are polynomials and form a group with identity 0, the dilation structure preserves the group operations and is given in coordinates by
\begin{align*}
\lambda\circ (x_{1},\ldots,x_{n}):=(\lambda^{a_{1}}x_{1},\ldots,\lambda^{a_{n}}x_{n})
\end{align*}
for some positive numbers $0<a_{1}\leq a_{2}\leq\ldots\leq a_{n}$. Without loss of generality, we shall assume that $a_{1}=1$ (see \cite{FoSt}). Moreover, $\rho(x):=\max\limits_{1\leq j\leq n}\{|x_{j}|^{1/a_{j}}\}$ is a norm associated to the dilation structure. We call $n$ the Euclidean dimension of $\HH$, and the quantity $\mathbb{Q}:=\sum_{j=1}^{n}a_{j}$ the homogeneous dimension of $\HH$, respectively.

Let $\Sigma:=\{x\in \HH:\rho(x)=1\}$ be the unit sphere on $\HH$ and $\sigma$ be the Radon measure on $\Sigma$ (see for example \cite[Proposition 1.15]{FoSt}) such that for any $f\in\mathbb{H}$,
\begin{align}\label{polar}
\int_{\HH}f(x)dx=\int_{0}^{\infty}\int_{\Sigma}f(r\circ \theta)r^{\mathbb{Q}-1}d\sigma(\theta)dr.
\end{align}
\newpage

Let $\Omega$ be a locally integrable function on $\HH \backslash \{0\}$ and it is homogeneous of degree 0 with respect to group dilation, that is, $\Omega(\lambda\circ x)=\Omega(x)$ for $x\neq 0$ and $\lambda>0$. We say that $\Omega$ satisfies the cancellation condition of order $N$ if for any  polynomial $P_{m}$ on $\HH$ of homogeneous degree $m\leq N$, we have
$$\int_{\Sigma}\Omega(\theta)P_{m}(\theta)d\sigma(\theta)=0.$$

In this paper, we will study the singular integral operators $T_{\Omega,\alpha}$ ($\alpha\geq 0$), and the maximal singular integral operators $T_{\Omega,\alpha}^{\#}$, which are formally defined by
\begin{align*}
T_{\Omega,\alpha}f(x):=\lim_{\varepsilon\rightarrow 0}T_{\Omega,\alpha,\varepsilon}f(x):=\lim_{\varepsilon\rightarrow 0}\int_{\rho(y^{-1}x)>\varepsilon}\frac{\Omega(y^{-1}x)}{\rho(y^{-1}x)^{\mathbb{Q}+\alpha}}f(y)dy,
\end{align*}
\begin{align*}
T_{\Omega,\alpha}^{\#}f(x):=\sup\limits_{\varepsilon>0}|T_{\Omega,\alpha,\varepsilon}f(x)|,
\end{align*}
where $\Omega$ satisfies the cancellation condition of order $[\alpha]$. Here we used the notation $[\alpha]$ to denote the integer part of $\alpha$.

It is well-known that for the case $\alpha=0$ and $\HH$ is an isotropic Euclidean space, Calder\'{o}n and Zygmund \cite{CZ} used the method of rotations to show that if $\Omega\in L{\rm log}^{+}L(\mathbb{S}^{n-1})$, then $T_{\Omega,0}$ is bounded on $L^{p}(\mathbb{R}^{n})$ for any $1<p<\infty$. Later, Ricci and Weiss \cite{RW} relaxed the condition to $\Omega\in H^{1}(\mathbb{S}^{n-1})$, under which this result was extended to the maximal singular integral operators $T_{\Omega,0}^{\#}$ by Fan and Pan \cite{FP}. Next, the authors in \cite{CFY,CZ2} considered the case of $\alpha\geq0$, which includes a larger class of singular integrals which are of interest in harmonic analysis and partial differential equation, such as the composition of partial derivative and Riesz transform. They established the $(L_{\alpha}^{p}(\mathbb{R}^{n}), L^{p}(\mathbb{R}^{n}))$ boundedness when $\Omega\in H^{q}(\mathbb{S}^{n-1})$ with $q=\frac{n-1}{n-1+\alpha}$. There are also many other significant progress on rough singular integral operators in the setting of Euclidean space (see for example \cite{CZ0,CD,Christ1,Christ2,DL,DR,FP, GS,Hof,LMW,Se}). However, due to the lack of fundamental tools, 
the parallel results on homogeneous groups are extremely limited. Among these results, we would like to highlight that by studying the left-invariant differentiation structures of homogeneous groups and applying the iterated $TT^{*}$ method, Tao \cite{Tao} has created a pioneering work which illustrates that if $\Omega\in L\log^{+} L(\Sigma)$, then $T_{\Omega,0}$ is bounded on $L^{p}$ for any $1<p<\infty$. Inspired by his work, Sato \cite{sato} obtained the $L^{p}$ boundedness of $T_{\Omega,0}^{\#}$ under the same condition.
%
Thus, it is natural to ask whether one can obtain $(L_{\alpha}^{p}(\HH), L^{p}(\HH))$ boundedness of the 
operators $T_{\Omega,\alpha}$ and $T^{\#}_{\Omega,\alpha}$ for $\alpha>0$, where $L_{\alpha}^{p}(\HH)$ is the Sobolev space on $\HH$ defined in \eqref{sobolevde} with $w=1$.

Another motivation of this research comes from the investigation of fractional-order partial differential equation. Caffarelli and Silvestre \cite{CSpde} constructed a fractional differentiation from an extension problem to the upper half space for a specific elliptic partial differential equation. As a continuation of their work, Chamorro and  Jarr\'{i}n \cite{CJpde} further investigated this problem on general nilpotent Lie group. In order to investigate fractional-order partial differential equations more deeply on homogeneous groups, the first aim of our paper is to establish the $(L_{\alpha}^{p}(\HH), L^{p}(\HH))$-boundedness of the operators $T^{\#}_{\Omega,\alpha}$ since it connects closely to the fractional Laplacian operator and the composition of fractional derivative with Riesz transform.
%
%

\begin{theorem}\label{main1}
Let $0<\alpha<\mathbb{Q}$. Suppose that $\Omega\in L^{1}(\Sigma)$ and satisfies the cancellation condition of order $[\alpha]$, then for any $1<p<\infty$,
\begin{align*}
\|T_{\Omega,\alpha}^{\#}f\|_{L^{p}(\HH)}\leq C_{\mathbb{Q},\alpha,p}\|\Omega\|_{L^{1}(\Sigma)}\|f\|_{L_{\alpha}^{p}(\HH)}
\end{align*}
for some constant $C_{\mathbb{Q},\alpha,p}$ independent of $\Omega$.
\end{theorem}


The second aim  of our paper is to establish the quantitative weight inequalities for $T_{\Omega,\alpha}^{\#}$.

In the Euclidean setting, after the establishment of the sharp weight inequalities for Ahlfors-Beurling operator by Petermichl and Volberg \cite{PV}, for Hilbert transform and Riesz transform by Petermichl \cite{Petermichl1,Petermichl2} and for general Calder\'{o}n--Zygmund operators by Hyt\"{o}nen  \cite{H} (see also \cite{La,Lerner1,Lerner2}), Hyt\"{o}nen--Roncal--Tapiola \cite{HRT} first quantitatively obtained the weighted bounds for $T_{\Omega,0}$ (see also \cite{CCDO,Ler1}). Later, this result was extended to the maximal singular integrals $T_{\Omega,0}^{\#}$ by Di Plinio, Hyt\"{o}nen and Li \cite{DHL} and Lerner \cite{Lerner4} via sparse domination, which gives
\begin{align}\label{Q W bd for Tsharp}
\|T_{\Omega,0}^{\#}f\|_{L^{p}(w)}\leq C_{n,p}\|\Omega\|_{L^{\infty}(\mathbb{S}^{n-1})}\{w\}_{A_{p}}(w)_{A_{p}}\|f\|_{L^{p}(w)},
\end{align}
where $w$ is an $A_p$ weight, $\{w\}_{A_{p}}$ and $(w)_{A_{p}}$ are the quantitative constants with respect to $w$ and $L^{p}(w)$ is the weighted $L^p$ space (all definitions are provided in Section \ref{weisection}).
However, it is still unclear that whether a quantitative weight bound for $T_{\Omega,0}^{\#}$ can be obtained  on homogeneous groups. Moreover, there is no weighted result for the case $\alpha>0$ even non-quantitative one in the special case of $\mathbb{R}^{n}$. To fill in these gaps, we conclude the following result.

%
%

\begin{theorem}\label{main2}
Let $0\leq \alpha<\mathbb{Q}$ and $q>\mathbb{Q}/\alpha$. Suppose that $\Omega\in L^{q}(\Sigma)$ and satisfies the cancellation condition of order $[\alpha]$, then for any $1<p<\infty$ and $w\in A_{p}$,
\begin{align*}
\|T_{\Omega,\alpha}^{\#}f\|_{L^{p}(w)}\leq C_{\mathbb{Q},\alpha,p,q}\|\Omega\|_{L^{q}(\Sigma)}\{w\}_{A_p}(w)_{A_p}\|f\|_{L_{\alpha}^{p}(w)}
\end{align*}
for some constant $C_{\mathbb{Q},\alpha,p,q}$ independent of $\Omega$ and $w$.
\end{theorem}

\noindent Here  
$L_{\alpha}^{p}(w)$\ ($1<p<\infty$, $\alpha\geq 0$, $w\in A_{p} $)  is the homogeneous weighted Sobolev space defined by
\begin{align}\label{sobolevde}
\|f\|_{L_{\alpha}^{p}(w)}:=\left(\int_{\HH}|(-\Delta_{\HH})^{\alpha/2}f(x)|^{p}w(x)dx\right)^{1/p},
\end{align}
where the definition of $\Delta_{\HH}$ is defined in Section \ref{prehomo} and $(-\Delta_{\HH})^{\alpha/2}$ is defined spectrally.


\subsection{Comparisons with previous results}


Table 1  highlights our contributions in the $L^{p}$ boundedness and the quantitative $L^{p}(w)$ boundedness for the operators $T_{\Omega,\alpha}$ and $T_{\Omega,\alpha}^{\#}$ in $\mathbb R^n$ and $\HH$ via a comparison of known results under different conditions. Among these results, we would like to mention that the $L^{p}$ boundedness of $T_{\Omega,\alpha}$ and $T_{\Omega,\alpha}^{\#}$ ($0<\alpha<\mathbb{Q}$) are new even when $\Omega\in L^{\infty}$. Unfortunately, whether one can weaken the size condition of $\Omega$ to certain Hardy space is still open even when $\alpha=0$. Furthermore, for the weighted setting, the weight bound $\{w\}_{A_p}(w)_{A_p}$ we obtained is consistent with that obtained in \cite{DHL}. It is still open that whether this is sharp, but it is the best known quantitative result for this class of operators.
\begin{table}[h]
\begin{center}
\begin{tabular}{|c|c|c|c|c|c|}
\hline
\multicolumn{2}{|c|}{ \multirow{2}*{\diagbox[width=7.7em]{Operator}{Setting}} }&\multicolumn{2}{c|}{$L^{p}$ boundedness}&\multicolumn{2}{c|}{Quantitative $L^{p}(w)$ boundedness}\\
\cline{3-6}
\multicolumn{2}{|c|}{}&  $\mathbb{R}^{n}$ & $\HH$ & $\mathbb{R}^{n}$ & $\HH$ \\
\hline
\multirow{2}*{$T_{\Omega,\alpha}$}&$\alpha=0$&\cite{RW}:\ $\Omega\in H^{1}$&\cite{Tao}$:\ \Omega\in L\log^{+} L$&\cite{HRT} :\ $\Omega\in L^{\infty}$&\color{brown}{\rm Thm}\ref{main2}:\ $\Omega\in L^{\infty}$\\
\cline{2-6}
&$0<\alpha<\mathbb{Q}$&\cite{CZ2}:\ $\Omega\in H^{\frac{n-1}{n-1+\alpha}}$&$\color{brown}{\rm Thm}\ref{main1}:\ \Omega\in L^{1}$&\color{brown}{\rm Thm}\ref{main2}:\ $\Omega\in L^{q/\alpha}$&\color{brown}{\rm Thm}\ref{main2}:\ $\Omega\in L^{q/\alpha}$\\
\hline
\multirow{2}*{$T_{\Omega,\alpha}^{\#}$}&$\alpha=0$&\cite{FP}:\ $\Omega\in H^{1}$&\cite{sato}$:\ \Omega\in L\log^{+} L$&\cite{DHL}:\ $\Omega\in L^{\infty}$&\color{brown}{\rm Thm}\ref{main2}:\ $\Omega\in L^{\infty}$\\
\cline{2-6}
&$0<\alpha<\mathbb{Q}$&\cite{CZ2}:\ $\Omega\in H^{\frac{n-1}{n-1+\alpha}}$&$\color{brown}{\rm Thm}\ref{main1}:\ \Omega\in L^{1}$&\color{brown}{\rm Thm}\ref{main2}:\ $\Omega\in L^{q/\alpha}$&\color{brown}{\rm Thm}\ref{main2}:\ $\Omega\in L^{q/\alpha}$\\
\hline
\end{tabular}
\medskip
\caption{Highlight of our contributions}
\end{center}
\end{table}
\subsection{Difficulties and strategy of our proof}\label{strategy}
Due to the generality of the underlying space, many techniques in the Euclidean setting cannot be applied directly. In particular, some difficulties occur:

$\bullet$ We cannot apply Fourier transform and Plancherel's theorem as effectively as in \cite{DR};

$\bullet$ The order of convolution cannot be exchanged on non-Abelian homogeneous groups;

$\bullet$ Generally, the topology degree and homogeneous degree in Taylor's inequality are not equal;

$\bullet$ $L^{p}$ boundedness of the directional maximal function on homogeneous groups is not at hand;

$\bullet$ Due to the non-isotropic property of homogeneous groups, the method of rotation cannot be applied directly.

To overcome these difficulties, our main ideas are the following:

(1) Either in the unweighted setting or weighted setting, the first step is to decompose our maximal operator into a non-dyadic maximal one and a dyadic one.

(2) The weighted estimate of the non-dyadic maximal operator can be deduced from that of Hardy-Littlewood maximal function directly, whereas under the weaker assumption $\Omega\in L^1(\Sigma)$, the $L^p$ estimate of this operator is technical. Our strategy is to reduce the problem to a kind of mixed $L^2$ norm estimates of certain maximal singular integral operators with smooth kernels by applying Gagliardo-Nirenberg inequality.

(3) We apply Cotlar-Knapp-Stein Lemma and some tricks of geometric means instead of Fourier transform, Plancherel's theorem and the exchange of convolution order to show an abstract $L^{2}$ decay lemma:
: suppose that $\{\mu_{j}^{\alpha}\}_{j\in\mathbb{Z}}$ is a family of Borel measures on $\HH$ such that its behavior like $\frac{\Omega(x)}{\rho(x)^{\mathbb{Q}+\alpha}}\chi_{\rho(x)\sim 2^{j}}$, then there exists a constant $\tau>0$ such that for any $j\in\mathbb{Z}$, $0<\alpha<\mathbb{Q}$ and $q>1$,
\begin{align*}
\|G_{j}^{\alpha}(t)f\|_{2}\lesssim2^{-\tau |j|}\|\Omega\|_{L^{1}(\Sigma)}\|f\|_{2},\ \
\|G_{j}^{0}(t)f\|_{2}\lesssim2^{-\tau |j|}\|\Omega\|_{L^{q}(\Sigma)}\|f\|_{2},
\end{align*}
where $G_{j}^{\alpha}(t)f$ is defined in \eqref{equa} and \eqref{equa001}.

(4) We use Cotlar's decomposition together with Littlewood-Paley theory to decompose operators with rough kernel into summation of operators with smooth kernel. Then we adapt the above $L^2$ decay Lemma together with Khinchin's inequality to establish the $L^2$ decay estimates of decomposed parts.

For the convenience of readers who are interested in our proof of framework, we provide two figures in the following. Here in Figure 1, the definitions of $M_{\Omega,\alpha}$, $T_{\Omega,\alpha}^{k}$, $V_{k,j,t}^{\alpha}$, ${\rm I_{1}}$, ${\rm I_{2}}$, ${\rm I_{3}}$, $G_{k,s,j}^{\alpha}$, $\tilde{T}_{j}^{\alpha}$ can be found in \eqref{defi1}, \eqref{defi2}, \eqref{vkj}, \eqref{IIIIII}, \eqref{IIIIII}, \eqref{IIIIII}, \eqref{gksj}, \eqref{tj1}, respectively. In Figure 2, the definitions of $\tilde{T}_{\Omega,\alpha}^{k}$, ${\rm I}$, ${\rm II}$, ${\rm III}$, $\tilde{T}_{j}^{\alpha,N}$, $R^{\alpha}$, $A_{k+s}^{\alpha}K_{\alpha}^{0}$, $\Delta[2^{k}]\phi$  can be found in \eqref{tildeT},  \eqref{iiiiii}, \eqref{iiiiii}, \eqref{iiiiii}, \eqref{b222}, \eqref{ra}, \eqref{akk0}, \eqref{deltamap}, respectively.
\newpage

%

\textbf{Frame of proof of Theorem \ref{main1}:}

\begin{figure}[h]
\centering
\begin{tikzpicture}[>=stealth,transform shape,scale=.8]
\begin{scope}
\node (supsupL1) {$\Big\|\sup\limits_{t\in[1,2)}\sup\limits_{k\in\mathbb{Z}}|V_{k,j,t}^{\alpha}(\cdot)|\Big\|_{L^{1}\rightarrow L^{1,\infty}}$};
\node[right] (supsup22) at (supsupL1.east) {$\Big\|\sup\limits_{t\in[1,2)}\sup\limits_{k\in\mathbb{Z}}|V_{k,j,t}^{\alpha}(\cdot)|\Big\|_{2\rightarrow 2}$};
\node[right] (TL1) at (supsup22.east) {$\|\tilde{T}_{j}^{\alpha}\|_{L^{1}\rightarrow L^{1,\infty}}$};
\node[right] (T22) at (TL1.east) {$\|\tilde{T}_{j}^{\alpha}\|_{2\rightarrow 2}$};
\node[right] (supL1) at (T22.east) {$\Big\|\sup\limits_{k\in\mathbb{Z}}|G_{k,s,j}^{\alpha}(\cdot)|\Big\|_{L^{1}\rightarrow L^{1,\infty}}$};
\node[right] (sup22) at (supL1.east) {$\Big\|\sup\limits_{k\in\mathbb{Z}}|G_{k,s,j}^{\alpha}(\cdot)|\Big\|_{2\rightarrow 2}$};
\draw[decorate,decoration={brace}] (TL1.north) -- coordinate (TL122) (T22.north);
\draw[->,double] (TL122)+(0,.1) -- node[right] {Interpolation} ++(0,1) coordinate (Tjpp south);
\node[above] (Tjpp) at (Tjpp south) {$\|\tilde{T}_{j}^{\alpha}\|_{p\rightarrow p}$};
\draw[->,double] (Tjpp.north) -- node[right] {Decomposition} ++(0,1) coordinate (Tpp south);
\node[above] (Tpp) at (Tpp south) {$\big\|T_{\Omega,\alpha}(-\Delta_{\HH})^{-\alpha/2}(\cdot)\big\|_{p\rightarrow p}$};
\draw[->,double] (Tpp.north) -- ++(0,1) coordinate (I1 south);
\node[above] (I1) at (I1 south) {$\rm I_{1}$};
\draw[decorate,decoration={brace}] (supL1.north) -- coordinate (supL122) (sup22.north);
\node (I3) at (I1-|supL122) {$\rm I_{3}$};
\draw[<-,double] (I3.south) -- ++(0,-1) coordinate (suppp north);
\node[below] (suppp) at (suppp north) {$\Big\|\sup\limits_{k\in\mathbb{Z}}|G_{k,s,j}^{\alpha}(\cdot)|\Big\|_{p\rightarrow p}$};
\draw[->,double] (supL122)+(0,.1) -- node[right] {Interpolation} (suppp);
\draw[decorate,decoration={brace}] (I1.north) -- coordinate (I2 north) (I3.north);
\node[below] (I2) at (I2 north) {$\rm I_{2}$};
\draw[<-,double] (I2.south) -- ++(0,-1) node[below] {$\|M\|_{p\rightarrow p}$};
\node[above] (supTpp) at (I2 north) {$\Big\|\sup\limits_{k\in\mathbb{Z}}|T_{\Omega,\alpha}^{k}(-\Delta_{\HH})^{-\alpha/2}(\cdot)|\Big\|_{p\rightarrow p}$};
\draw[decorate,decoration={brace}] (supsupL1.north) -- coordinate (supsupL122) (supsup22.north);
\node[below] (Mpp) at (supTpp.north-|supsupL122) {$\|M_{\Omega,\alpha}(-\Delta_{\HH})^{-\alpha/2}\|_{p\rightarrow p}$};
\node (supsuppp) at (Mpp|-Tpp) {$\Big\|\sup\limits_{t\in[1,2)}\sup\limits_{k\in\mathbb{Z}}|V_{k,j,t}^{\alpha}(\cdot)|\Big\|_{p\rightarrow p}$};
\draw[->,double] (supsuppp) -- node[right] {Decomposition} (Mpp);
\draw[->,double] (supsupL122)+(0,.1) -- node[right] {Interpolation} (supsuppp);
\draw[decorate,decoration={brace}] (Mpp.north) -- coordinate (MsupTpp) (supTpp.north);
\draw[->,double] (MsupTpp)+(0,.1) --node[right] {Decomposition} ++(0,1) node[above] {$\|T_{\Omega,\alpha}^{\#}\|_{L_{\alpha}^{p}\rightarrow L^{p}}$};
\end{scope}
\node[below=1cm] (VHc) at (supsup22.south) {Verify H\"{o}rmander condition};
\node[below=1cm] (LKi) at (supL1.south) {$L^{2}$ decay estimate $\oplus$ Khinchin's inequality};
\draw[->,double] (VHc) -- (supsupL1.south);
\draw[->,double] (VHc) -- (TL1.south);
\draw[->,double] (VHc) -- (supL1.south);
\draw[->,double] (LKi) -- (supsup22.south);
\draw[->,double] (LKi) -- (T22.south);
\draw[->,double] (LKi) -- (sup22.south);
\end{tikzpicture}
\caption{}
\end{figure}

\textbf{Frame of proof of Theorem \ref{main2}:}

\begin{figure}[h]
\centering
\begin{tikzpicture}[>=stealth,transform shape,scale=.8]
\node (Sd) {Sparse domination};
\draw[->,double] (Sd.north) -- ++(0,1) node[above] (Lbwo) {$L^{p}(w)$ boundedness without decay};
\node[right=5cm] (Lbwd) at (Lbwo.east) {$L^{2}$ boundedness with decay};
\node (LeKi) at (Sd-|Lbwd) {$L^{2}$ decay estimate $\oplus$ Khinchin's inequality};
\draw[<-,double] (Lbwd.south) -- (LeKi);

\draw[decorate,decoration={brace}] (Lbwo.north) -- coordinate (Lbd) (Lbwd.north);
\node[above=2cm] (MLp) at (Lbd) {$\|M\|_{L^{p}(w)\rightarrow L^{p}(w)}$};
\draw[->,double] (MLp.north) -- ++(0,1) node[above] (I2) {${\rm II}$};

\node (supTjLp) at (Lbwo|-MLp) {$\Big\|\sup\limits_{k\in\mathbb{Z}}|\tilde{T}_{j}^{\alpha,N}(\cdot)\ast\Delta[2^{k}]\phi|\makebox[0pt][l]{$\Big\|_{L^{p}(w)\rightarrow L^{p}(w)}$}$};
\draw[->,double] (Lbd.north)+(0,.1) --node[below left=-3mm,align=left] {Interpolation\\with change of measure} (supTjLp.south);
\node (I1) at (supTjLp|-I2) {${\rm I}$};
\draw[->,double] (supTjLp) --node[right] {Decomposition} (I1.south);

\node (supRLp) at (Lbwd|-MLp) {$\Big\|\sup\limits_{k\in\mathbb{Z}}|(\cdot)\ast R^{\alpha}\ast A_{k+s}^{\alpha}K_{\alpha}^{0}\ast (\delta_{0}-\Delta[2^{k}]\phi)|\makebox[0pt][l]{$\Big\|_{L^{p}(w)\rightarrow L^{p}(w)}$}$};
\draw[->,double] (Lbd.north)+(0,.1) --node[below right=-3mm,align=right] {Interpolation\\with change of measure} (supRLp.south);
\node (I3) at (supRLp|-I2) {${\rm III}$};
\draw[->,double] (supRLp) --node[right] {Decomposition} (I3.south);

\draw[decorate,decoration={brace}] (I1.north) --node[above] (suptTLp) {$\Big\|\sup\limits_{k\in\mathbb{Z}}|\tilde{T}_{\Omega,\alpha}^{k}(-\Delta_{\HH})^{-\alpha/2}(\cdot)|\makebox[0pt][l]{$\Big\|_{L^{p}(w)\rightarrow L^{p}(w)}$}$} (I3.north);
\draw[<->,double] (suptTLp.north) -- ++(0,1) node[above] (supTLp) {$\Big\|\sup\limits_{k\in\mathbb{Z}}|T_{\Omega,\alpha}^{k}(-\Delta_{\HH})^{-\alpha/2}(\cdot)|\makebox[0pt][l]{$\Big\|_{L^{p}(w)\rightarrow L^{p}(w)}$}$};

\draw[decorate,decoration={brace,mirror}] (supTLp.north) -- coordinate (Msup) ++(-7,0) node[below] (MdLp) {$\|M_{\Omega,\alpha}(-\Delta_{\HH})^{-\alpha/2}\|\makebox[0pt][l]{$_{L^{p}(w)\rightarrow L^{p}(w)}$}$};
\draw[<-,double] (MdLp.south) -- ++(0,-1) node[below] {$\|M\|\makebox[0pt][l]{$_{L^{p}(w)\rightarrow L^{p}(w)}$}$};

\draw[->,double] (Msup.north)+(0,.1) --node[right] {Decomposition} ++(0,1) node[above] {$\|T_{\Omega,\alpha}^{\#}\|\makebox[0pt][l]{$_{L_{\alpha}^{p}(w)\rightarrow L^{p}(w)}$}$};
\end{tikzpicture}
\caption{}
\end{figure}

\subsection{Notation and structure of the paper}
For $1\leq p \leq+\infty$, we denote the norm of a function $f\in L^{p}(\HH)$ by $\|f\|_{p}$. If $T$ is a bounded linear operator
on $L^{p}(\HH)$, $1\leq p\leq+\infty$, we write $\|T\|_{p\rightarrow p}$ for the operator norm of $T$. The
indicator function of a subset $E\subseteq X$ is denoted by $\chi_{E}$. We use $A\lesssim B$ to denote the statement that $A\leq CB$ for some constant $C>0$, and $A\sim B$ to denote the statement that $A\lesssim B$ and $B\lesssim A$.

This paper is organized as follows. In Section \ref{preliminariessec} we provide the preliminaries, including some auxiliary lemmas on homogeneous groups, the definitions of $A_p$ weights and Calder\'{o}n-Zygmund operators with Dini-continuous kernel on  homogeneous group $\HH$, two types of decompositions and their corresponding $L^{2}$ decay estimates. In Section \ref{unweightsection} and \ref{weightsection}, we give the proofs of Theorem \ref{main1} and \ref{main2}, respectively.

\bigskip
\section{Preliminaries and fundamental tools on homogeneous Lie groups}\label{preliminariessec}
\setcounter{equation}{0}

\subsection{Auxiliary lemmas on homogeneous group $\HH$}\label{prehomo}
In this subsection, we recall some basic definitions on homogeneous groups and then show some auxiliary lemmas. To begin with, we recall the multiplication structure on $\HH$ (see \cite[Theorem 1.3.15]{Liebook}): if $x=(x_{1},\cdots,x_{n})$, $y=(y_{1},\cdots,y_{n})\in\HH$, then for any $2\leq j\leq n$,
\begin{align}\label{multilaw}
(xy)_{1}=x_{1}+y_{1},\ \ (xy)_{j}=x_{j}+y_{j}+Q_{j}(x,y),
\end{align}
where $Q_{j}(x,y)$ is a sum of mixed monomials in $x$, $y$, only depends on $x_{i}$, $y_{i}$, $i=1,2,\cdots,j-1$, and satisfies 
$Q_{j}(\lambda\circ x,\lambda\circ y)=\lambda^{a_{j}}Q_{j}(x,y)$ for all $\lambda>0$. As a corollary (see \cite[Corollary 1.3.16]{Liebook}), for any $x=(x_{1},\cdots,x_{n})\in\HH$ and $2\leq j\leq n$, one has
\begin{align}\label{xjjj}
(x^{-1})_{1}=-x_{1},\ \ (x^{-1})_{j}=-x_{j}+q_{j}(x),
\end{align}
where $q_{j}(x)$ only depends on $x_{1}$, $x_{2}$, $\cdots$, $x_{j-1}$, and is a polynomial function in $x$ of homogeneous degree $a_{j}$, that is, $q_{j}(\lambda\circ x)=\lambda^{a_{j}}q_{j}(x)$ for all $\lambda>0$. The following simple observation is useful to exploit the $[\alpha]$-order cancellation property of $\Omega$ in what follows.
\begin{lemma}\label{inversethm}
{\rm (1)} Let $N$ be a positive integer and $P(x)$ be a polynomial of homogeneous degree $N$, then $P(x^{-1})$ is a polynomial in $x$ of homogeneous degree $N$.

{\rm (2)} Let $N$ be a positive integer and $P(x)$ be a polynomial of homogeneous degree $N$, then for any $y\in\HH$, $P(xy)$ and $P(yx)$ are summations of polynomials in $x$, with coefficients depending on $y$, of homogeneous degree less than or equal to $N$.
\end{lemma}
\begin{proof}
(1) According to the hypothesis, $P(x)$ is of the form: $P(x)=\sum_{a_{1}b_{1}+\cdots+a_{n}b_{n}=N}c_{b}x_{1}^{b_{1}}\cdots x_{n}^{b_{n}}$. Therefore, it follows from \eqref{xjjj} that for any $\lambda>0$,
\begin{align*}
P(\lambda\circ x^{-1})&=\sum_{a_{1}b_{1}+\cdots+a_{n}b_{n}=N}c_{b}(-\lambda^{a_{1}}x_{1})^{b_{1}}(-\lambda^{a_{2}}x_{2}+q_{2}(\lambda\circ x))^{b_{2}}\cdots(-\lambda^{a_{n}}x_{n}+q_{n}(\lambda\circ x))^{b_{n}}\\
&=\sum_{a_{1}b_{1}+\cdots+a_{n}b_{n}=N}c_{b}(-\lambda^{a_{1}}x_{1})^{b_{1}}(-\lambda^{a_{2}}x_{2}+\lambda^{a_{2}}q_{2}( x))^{b_{2}}\cdots(-\lambda^{a_{n}}x_{n}+\lambda^{a_{n}}q_{n}(x))^{b_{n}}\\
&=\lambda^{N} P(x^{-1}).
\end{align*}
This implies the first statement of Lemma \ref{inversethm}.

(2) By \eqref{multilaw}, $P(xy)$ is of the form:
\begin{align*}
&P(xy)=\sum_{a_{1}b_{1}+\cdots+a_{n}b_{n}=N}c_{b}(xy)_{1}^{b_{1}}(xy)_{2}^{b_{2}}\cdots (xy)_{n}^{b_{n}}\\
&=\sum_{a_{1}b_{1}+\cdots+a_{n}b_{n}=N}c_{b}(x_{1}+y_{1})_{1}^{b_{1}}(x_{2}+y_{2}+Q_{2}(x_{1},y_{1}))^{b_{2}}\cdots (x_{n}+y_{n}+Q_{n}(x_{1},\cdots,x_{n-1},y_{1},\cdots,y_{n-1}))^{b_{n}},
\end{align*}
where $Q_{j}(x,y)$ is a sum of mixed monomials in $x$, $y$, only depends on $x_{i}$, $y_{i}$, $i=1,2,\cdots,j-1$, and satisfies $Q_{j}(\lambda\circ x,\lambda\circ y)=\lambda^{a_{j}}Q_{j}(x,y)$ for all $\lambda>0$. Hence, by expanding each term in the brackets and handling $P(yx)$ similarly, we can easily conclude the second statement of Lemma \ref{inversethm}.
\end{proof}
Now we denote the set of left-invariant vector fields on $\HH$ by $\mathfrak{g}$, which is called the Lie algebra of $\HH$. One identifies $\mathfrak{g}$ and $\mathbb{H}$ via the exponential map
\begin{align*}
\exp: \mathfrak{g} \longrightarrow \mathbb{H},
\end{align*}
which is a globally defined diffeomorphism. Let $X_{j}\ ({\rm resp.}\ Y_{j})$ be the left-invariant (resp. right-invariant) vector field that agrees with $\partial/\partial x_{j}\ ({\rm resp.}\ \partial/\partial y_{j})$ at the origin. Equivalently (\cite{FoSt}),
\begin{align*}
X_{j}f(x)=\frac{d}{dt}f(x \exp(tX_{j}))\big|_{t=0},\ \ Y_{j}f(x)=\frac{d}{dt}f(\exp(tX_{j}) x)\big|_{t=0}.
\end{align*}
By Proposition 1.2.16 in \cite{Liebook}, the family $\{X_{j}\}_{1\leq j\leq n}\ ({\rm resp.}\ \{Y_{j}\}_{1\leq j\leq n})$ forms a Jacobian basis of $\mathfrak{g}$.  We adopt the following multiindex notation for higher order derivatives. For $I=(i_{1},i_{2},\ldots,i_{n})\in\mathbb{N}^{n}$, we denote
$
X^{I}:=X_{1}^{i_{1}}X_{2}^{i_{2}}\cdots X_{n}^{i_{n}}.
$
Moreover, let $|I|$ be the order of the differential operator $X^{I}$, while $d(I)$ is the corresponding homogeneous degree, that is,
\begin{align*}
|I|:=i_{1}+i_{2}+\ldots+i_{n},\ \ d(I):=a_{1}i_{1}+a_{2}i_{2}+\ldots+a_{n}i_{n}.
\end{align*}
Moreover, let $\Delta_{\HH}:=\sum_{j=1}^{n}X_{j}^{2}$ be the sub-Laplacian operator on $\HH$.

\begin{lemma}\label{citetaylor}
Assume that $f\in C_0^{N+1}(\HH)$ for some $N\in\mathbb{N}$, satisfying supp$f\subset B(0,1)$. For any $x\in\HH$, set $f_k(x)=f(2^{-k}\circ x)$. Moreover, assume that $\{m_k\}_{k\in\mathbb{Z}}$ be a family of Borel measure satisfying $N$-order cancellation and supp$m_k\subset B(0,2^k)$. Then there are constants $C_N,\kappa_N>0$ such that
\begin{align*}
|(f_k\ast m_j)(x)|\leq C_N 2^{-k(N+1)}\sup\limits_{\substack{\rho(z)\leq \kappa_{N}2^{j-k}\\d(I)=N+1}}|(X^If)((2^{-k}\circ x)z)|\int_{\HH}\rho(y)^{N+1}d|m_j|(y),
\end{align*}
and
\begin{align*}
|(m_j\ast f_k)(x)|\leq C_N 2^{-k(N+1)}\sup\limits_{\substack{\rho(z)\leq \kappa_{N}2^{j-k}\\d(I)=N+1}}|(Y^If)(z(2^{-k}\circ x))|\int_{\HH}\rho(y)^{N+1}d|m_j|(y).
\end{align*}
\end{lemma}
\begin{proof}
It suffices to show the first inequality since the second one is similar. To this end, let $L_{x}^{f}$ be the left Taylor polynomial of $f$ at $x$ of homogeneous degree $N$. Then by the $N$-order cancellation condition of $m_k$ and Taylor's inequality on homogeneous groups $\HH$ (see [16, Corollary 1.44]),
\begin{align*}
|(f_k\ast m_j)(x)|&=\left|\int_{\HH}\Big(f(2^{-k}\circ(xy^{-1}))-L_{2^{-k}\circ x}^f(2^{-k}\circ y^{-1})\Big)dm_j(y)\right|\\
&\leq C_N\int_{\HH}\sup\limits_{\substack{\rho(z)\leq \kappa_{N}\rho(2^{-k}\circ y)\\d(I)=N+1}}|(X^If)((2^{-k}\circ x)z)|\rho(2^{-k}\circ y)^{N+1}d|m_j|(y)\\
&\leq C_N 2^{-k(N+1)}\sup\limits_{\substack{\rho(z)\leq \kappa_{N}2^{j-k}\\d(I)=N+1}}|(X^If)((2^{-k}\circ x)z)|\int_{\HH}\rho(y)^{N+1}d|m_j|(y)
\end{align*}
for some constants $C_N,\kappa_N>0$. This ends the proof of Lemma \ref{citetaylor}.
\end{proof}

We have the following mean value theorem of integral type.
\begin{lemma}\label{appen}
Let $f\in C^{1}(\HH)$ and $j=1,2,\cdots,\nu$ be the indices such that $a_{j}=1$, then
\begin{align*}
\int_{\HH}|f(yx)-f(x)|dx\lesssim\rho(y)\sum_{j=1}^{\nu}\int_{\HH}|(Y_{j}f)(x)|dx.
\end{align*}
\end{lemma}
\begin{proof}
Let $V_{1}$ be the linear span of $X_{1},\cdots,X_{\nu}$. Since $\HH$ is a stratified group, by \cite[Lemma 1.40]{FoSt}, there exists a constant  $N\in\mathbb{N}$ such that each element $y\in\HH$ can be represented as $y=y_{1}\cdots y_{N}$ with $y_{k}\in\exp(V_{1})$ and $\rho(y_{k})\lesssim\rho(y)$ for all $k$. Moreover, since $V_{1}$ is the linear span of $X_{1},\cdots,X_{\nu}$, each $y_{k}$ can be further written as $y_{k}=\exp(t_{k,1}X_{1}+\cdots+t_{k,\nu}X_{\nu})$. Equivalently,
\begin{align}\label{expmap1}
\exp^{-1}y_{k}=t_{k,1}X_{1}+\cdots+t_{k,\nu}X_{\nu}.
\end{align}
On the other hand, by \cite[Theorem 1.3.28]{Liebook}, The exponential map $\exp$ and its inverse map $\exp^{-1}$ are globally defined diffeomorphisms with polynomial component functions, which implies that if we write $y_{k}=(y_{k,1},\cdots,y_{k,n})$, then (see \cite[(1,75b)]{Liebook})
\begin{align}\label{expmap2}
\exp^{-1}y_{k}=y_{k,1}X_{1}+(y_{k,2}+C_{2}(y_{k,1}))X_{2}+\cdots+(y_{k,\nu}+C_{\nu}(y_{k,1},\cdots,y_{k,\nu-1}))X_{\nu},
\end{align}
where $C_{j}$'s are polynomial functions of homogeneous degree 1, completely determined by the multiplication law on $\HH$. Combining the equalities \eqref{expmap1} and \eqref{expmap2}, we conclude that
$$ t_{k,j}=\left\{\begin{array}{ll}y_{k,1}, &j=1,\\y_{k,j}+C_{j}(y_{k,1},\cdots,y_{k,j-1}), &j=2,\cdots,\nu.\end{array}\right.$$
Hence, $|t_{k,j}|\lesssim \rho(y_{k})\lesssim \rho(y)$. Therefore,
\begin{align*}
&\int_{\HH}|f(yx)-f(x)|dx\\
&\leq \sum_{k=1}^{N}\int_{\HH}|f(y_{k}y_{k+1}\cdots y_{N}x)-f(y_{k+1}\cdots y_{N}x)|dx\\
&=\sum_{k=1}^{N}\int_{\HH}\bigg|\int_{0}^{1}\big((t_{k,1}Y_{1}+\cdots+t_{k,\nu}Y_{\nu})f\big)\big(\exp(s(t_{k,1}X_{1}+\cdots+t_{k,\nu}X_{\nu}))y_{k+1}\cdots y_{N}x\big)ds\bigg|dx\\
&\leq \sum_{k=1}^{N}\sup\limits_{\rho(z_{k})\leq \rho(y_{k})}\int_{\HH}|\big((t_{k,1}Y_{1}+\cdots+t_{k,\nu}Y_{\nu})f\big)(z_{k}y_{k+1}\cdots y_{N}x)|dx\\
&\leq \sum_{k=1}^{N}\sum_{j=1}^{\nu}|t_{k,j}|\sup\limits_{\rho(z_{k})\leq \rho(y)}\int_{\HH}|(Y_{j}f)(z_{k}y_{k+1}\cdots y_{N}x)|dx\lesssim \rho(y)\sum_{j=1}^{\nu}\int_{\HH}|(Y_{j}f)(x)|dx,
\end{align*}
where we note that the integrand in the right-hand side of the first inequality above becomes $f(y_Nx)-f(x)$ when $k=N$ and similar notation happens in what follows.
This ends the proof of Lemma \ref{appen}.
\end{proof}
\subsection{$A_p$ weights on  homogeneous group $\HH$}\label{weisection}

We next recall the definition and some properties of $A_p$ weight on $\HH$. To begin with, we define a left-invariant quasi-distance $d$ on $\HH$ by $d(x,y)=\rho(x^{-1}y)$, which means that
there exists a constant $A_{0}\geq 1$ such that for any $x,y,z\in\HH$,
\begin{align*}
d(x,y)\leq A_{0}[d(x,z)+d(z,y)].
\end{align*}
Next, let $B(x,r):=\{y\in\HH :d(x,y)<r\}$ be the open ball with center $x\in\HH$ and radius $r>0$.

For $1<p<\infty$, we say that $w\in A_p$ if there exists a constant $C>0$ such that
\begin{align}\label{[Ap]}
[w]_{A_p}:=\sup_B\bigg(\frac1{|B|}\int_Bw(x)dx\bigg)\bigg(\frac1{|B|}\int_Bw(x)^{1-p'}dx\bigg)^{p-1}\le C.
\end{align}
Moreover, let $A_\infty := \cup_{1\leq p<\infty} A_p$ and we have
\begin{align*}
[w]_{A_{\infty}}:=\sup\limits_{B}\left(\frac1{|B|}\int_{B}wdx\right){\rm exp}\left(\frac1{|B|}\int_{B}{\rm log}\left(\frac{1}{w}\right)dx\right)<\infty,
\end{align*}
where the supremum is taken over all balls~$B\subset \HH$ with respect to the quasi-distance function $d$. We recall the following variants of the weight characteristic (see for example \cite{HRT}):
\begin{align*}
\{w\}_{A_{p}}:=[w]_{A_{p}}^{1/p}{\rm max}\{[w]_{A_{\infty}}^{1/p^{\prime}},[w^{1-p^{\prime}}]_{A_{\infty}}^{1/p}\},
\quad (w)_{A_{p}}:={\rm max}\{[w]_{A_{\infty}},[w^{1-p^{\prime}}]_{A_{\infty}}\}.
\end{align*}

\subsection{Calder\'{o}n-Zygmund operators with Dini-continuous kernel}
Let $T$ be a bounded linear operator on $L^{2}(\HH)$ represented as
\begin{align*}
Tf(x)=\int_{\HH}K(x,y)f(y)dy,\ \ \forall x\notin\supp f.
\end{align*}
A function $\omega:[0,1]\rightarrow [0,\infty)$ is a modulus of continuity if it satisfies the following three properties: (1) $\omega(0)=0$; (2) $\omega(s)$ is an increasing function; (3) For any $s_{1},s_{2}>0$, $\omega(s_{1}+s_{2})\leq\omega(s_{1})+\omega(s_{2})$.
\begin{definition}
We say that the operator $T$ is an $\omega$-Calder\'{o}n-Zygmund operator if the kernel $K$ satisfies the following two conditions:

(1) (size condition): There exists a constant $C_{T}>0$ such that
\begin{align*}
|K(x,y)|\leq \frac{C_{T}}{d(x,y)^{\mathbb{Q}}}.
\end{align*}

(2) (smoothness condition): Whenever $d(x,y)\geq 2A_{0}d(x,x^{\prime})>0$, we have
\begin{align*}
|K(x,y)-K(x^{\prime},y)|+|K(y,x)-K(y,x^{\prime})|\leq \omega\left(\frac{d(x,x^{\prime})}{d(x,y)}\right)\frac{1}{d(x,y)^{\mathbb{Q}}}.
\end{align*}
\end{definition}

Furthermore, $K$ is said to be a \textit{Dini-continuous kernel} if $\omega$ satisfies the \textit{Dini condition}:
\begin{align*}
\|\omega\|_{{\rm Dini}}:=\int_{0}^{1}\omega(s)\frac{ds}{s}<\infty.
\end{align*}
\subsection{Two types of decompositions}
To begin with, we recall that
for appropriate functions $f$ and $g$ defined on $\HH$, the convolution $f\ast g$ is defined by
\begin{align*}
f\ast g(x):=\int_{\HH}f(y)g(y^{-1}x)dy.
\end{align*}

For simplicity, denote
$$K_{\alpha}(x):=\frac{\Omega(\rho(x)^{-1}\circ x)}{\rho(x)^{\mathbb{Q+\alpha}}}.$$
Let $K_{\alpha}^{0}$ be the restriction of $K_{\alpha}$ to the annulus $\Sigma_{0}:=\{x\in\HH: 1\leq \rho(x)\leq 2\}$. Then we decompose the kernel $K_{\alpha}$ into smooth dyadic parts in the following way:
\begin{align*}
K_{\alpha}(x)=\frac{1}{\ln 2}\int_{0}^{\infty}\Delta_{\alpha}[t]K_{\alpha}^{0}(x)\frac{dt}{t},
\end{align*}
where for each $t$, we define the $\alpha$-scaling map by
\begin{align}\label{deltamap}
\Delta_{\alpha}[t]f(x):=t^{-\mathbb{Q}-\alpha}f(t^{-1}\circ x).
\end{align}
For simplicity, we set $\Delta[t]f(x):=\Delta_{0}[t]f(x)$.
Hence, we have the following decomposition
\begin{align}\label{akk0}
K_{\alpha}(x)=\sum\limits_{j\in\mathbb{Z}}2^{-j}\int_{0}^{\infty}\varphi(2^{-j}t)\Delta_{\alpha}[t]K_{\alpha}^{0}(x)dt=:\sum\limits_{j\in\mathbb{Z}}A_{j}^{\alpha}K_{\alpha}^{0}(x),
\end{align}
where $\varphi$ is a smooth cut-off function localized in $\{\frac{1}{2}\leq t\leq 2\}$ such that
$\sum\limits_{j\in\mathbb{Z}}2^{-j}t\varphi(2^{-j}t)=\frac{1}{\ln 2}$.
Hence,
\begin{align}\label{hjhbkk}
T_{\Omega,\alpha}f=\sum\limits_{j\in\mathbb{Z}}f\ast A_{j}^{\alpha}K_{\alpha}^{0}=:\sum\limits_{j\in\mathbb{Z}}T_{j}^{\alpha}f.
\end{align}
Since $\supp K_{\alpha}^{0}\subset \{x\in\HH: 1\leq \rho(x)\leq 2\}$, we see that for any $q\geq1$,
\begin{align}\label{Aj}
\|A_{j}^{\alpha}K_{\alpha}^{0}\|_{1}\lesssim2^{-\alpha j}\|K_{\alpha}^{0}\|_{1}\lesssim2^{-\alpha j}\|\Omega\|_{L^{q}(\Sigma)}.
\end{align}

In addition, we introduce the following non-smooth dyadic decomposition of $K_{\alpha}$, which plays a key role in obtaining the  $L^{p}$ estimate of $T_{\Omega,\alpha}^{\#}$ for the case $\alpha>0$:
\begin{align*}
K_{\alpha}(x)=\sum\limits_{j\in\mathbb{Z}}\frac{\Omega(x)}{\rho(x)^{\mathbb{Q}+\alpha}}\chi_{2^{j}<\rho(x)\leq 2^{j+1}}=:\sum\limits_{j\in\mathbb{Z}}B_{j}^{\alpha}\Omega(x).
\end{align*}
Therefore,
\begin{align*}
T_{\Omega,\alpha}f=\sum_{j\in\mathbb{Z}}f\ast B_{j}^{\alpha}\Omega=:\sum\limits_{j\in\mathbb{Z}}V_{j}^{\alpha}f.
\end{align*}
Moreover, it is direct that for any $q\geq 1$,
\begin{align}\label{Bj}
\|B_{j}^{\alpha}\Omega\|_{1}\lesssim 2^{-\alpha j}\|\Omega\|_{L^{q}(\Sigma)}.
\end{align}
\begin{remark}\label{share}
It is obvious that these two types of kernel truncations share many similar properties, including:  $A_{j}^{\alpha}K_{\alpha}^{0}$ and $B_{j}^{\alpha}\Omega$ are supported in the annulus $\{x\in\HH:\rho(x)\sim 2^{j}\}$ and satisfy the cancellation condition of order $[\alpha]$. The key difference lies in the $L^{2}$ decay estimate for $\alpha=0$ (see Lemma \ref{ckt}): the integral in the definition of $A_{j}^{\alpha}K_{\alpha}^{0}$ plays a crucial role in implementing the iterated $TT^{*}$ argument from \cite{Tao}, which relies on the integration by part with respect to $t$. Moreover,  the truncation $B_{j}^{\alpha}\Omega$ is feasible to obtain the unweighted $L^{p}$ estimate of $T_{\Omega,\alpha}^{\#}$ in the case $\alpha>0$, while the feasibility of using $A_{j}^{\alpha}K_{\alpha}^{0}$ is unclear.
\end{remark}

We now introduce a form of Littlewood-Paley theory without any explicit use of the Fourier transform. To this end, let $\phi\in C_{c}^{\infty}(\mathbb{H})$ be a smooth cut-off function satisfying

(1)$\supp\phi\subset\big\{x\in\HH:\frac{1}{200}\leq\rho(x)\leq\frac{1}{100}\big\}$; (2) $\int_{\HH}\phi(x)dx=1$; (3)$\phi\geq 0$; (4) $\phi=\tilde{\phi}$,\\
where $\tilde{F}$ denotes the function $\tilde{F}(x):=F(x^{-1})$. In addition, for each integer $j$, denote
\begin{align*}
\Psi_{j}:=\Delta[2^{j-1}]\phi-\Delta[2^{j}]\phi.
\end{align*}
With this choice of $\phi$, it follows that $\Psi_{j}$ is supported on the ball of radius $C2^{j}$, has mean zero, and $\tilde{\Psi}_{j}=\Psi_{j}$. Next we define the partial sum operators $S_{j}$ by
\begin{align*}
S_{j}f:=f\ast \Delta[2^{j-1}]\phi.
\end{align*}
Their differences are given by
\begin{align*}
S_{j}f-S_{j+1}f=f\ast \Psi_{j}.
\end{align*}
\subsection{$L^{2}$ decay estimates}
Let $p(t,x,y)$  ($t>0$, $x,y\in\HH$) be the heat kernel associated to the sub-Laplacian operator $-\Delta_{\HH}$ (that is, the integral kernel of $e^{-t\Delta_{\HH}}$). For convenience, we set $p(t,x)=p(t,x,o)$ (see \cite[Chapter 1, section G]{FoSt}) and $p(x)=p(1,x)$. Moreover, for $0< \alpha<Q$, let $R^{\alpha}$ be the kernel of Riesz potential operator $(-\Delta_{\HH})^{-\alpha/2}$ of order $\alpha$, that is,
\begin{align}\label{ra}
(-\Delta_{\HH})^{-\alpha/2}f=f\ast R^{\alpha}.
\end{align}
Note that for any $t>0$, $x,y\in\HH$, we have the following representation formula:
\begin{align}\label{repre}
R^{\alpha}(x)=\frac{1}{\Gamma(\alpha/2)}\int_{0}^{\infty}t^{\frac{\alpha}{2}-1}p(t,x)dt,
\end{align}
where $p(t,x)$ satisfies the following Gaussian estimate (see \cite[Theorem 4.2]{VSC}): for any multi-index $I=(i_{1},i_{2},\ldots,i_{n})\in\mathbb{N}^{n}$,
\begin{align*}
|Y^{I}p(t,x)|\lesssim t^{-\frac{\mathbb{Q}+|I|}{2}}\exp\left(-c\frac{\rho^{2}(x)}{t}\right).
\end{align*}
It follows from this formula that for $0<\alpha<\mathbb{Q}$,
\begin{align}\label{Rie}
|R^{\alpha}(x)|\lesssim\rho(x)^{-\mathbb{Q}+\alpha}.
\end{align}
For convenience, we extend the definition of $R^{\alpha}$ to $\alpha=0$ by setting $R^{0}=\delta_{0}$, where $\delta_{0}$ is a dirac measure.

 Moreover, let $0< \alpha<\mathbb{Q}$.
 We suppose that $\{\mu_{j}^{\alpha}\}_{j\in\mathbb{Z}}$ is a family of Borel measures on $\HH$ and for any $j\in\mathbb{N}$,

(1) $\supp\mu_{j}^{\alpha}\subset \{x\in\HH:\rho(x)\sim 2^{j}\}$ and $\supp\widetilde{\mu_{j}^{\alpha}}\subset \{x\in\HH:\rho(x)\sim 2^{j}\}$;

(2) $\mu_{j}^{\alpha}$ and $\widetilde{\mu_{j}^{\alpha}}$ satisfy the cancellation condition of order $[\alpha]$, i.e. for any polynomial $P$ on $\HH$ of homogeneous degree $\leq [\alpha]$, $$\int_{\HH}P(x)d\mu_{j}^{\alpha}(x)=\int_{\HH}P(x)d\widetilde{\mu_{j}^{\alpha}}(x)=0;$$

(3) $\|\mu_{j}^{\alpha}\|_{1}\lesssim 2^{-\alpha j}\|\Omega\|_{L^{1}(\Sigma)}$ and $\|\widetilde{\mu_{j}^{\alpha}}\|_{1}\lesssim 2^{-\alpha j}\|\Omega\|_{L^{1}(\Sigma)}$, where $\widetilde{\mu_{j}^{\alpha}}$ denotes the reflection of $\mu_{j}^{\alpha}$. i.e. $\widetilde{\mu_{j}^{\alpha}}(E):=\mu_{j}^{\alpha}(\{x^{-1}:x\in E\})$, for any $\mu_{j}^{\alpha}$-measurable set $E\subset\HH$.

Then we have the following key $L^{2}$ decay estimate.

\begin{proposition}\label{keypro}
Let $0< \alpha<\mathbb{Q}$. Then there exist constants $C_{\mathbb{Q},\alpha}>0$ and $\tau>0$ such that for any $j,k\in\mathbb{Z}$,
\begin{align}\label{huhihi}
\|f\ast \Psi_{k}\ast R^{\alpha}\ast \mu_{j}^{\alpha}\|_{2}\leq C_{\mathbb{Q},\alpha}2^{-\tau|j-k|}\|\Omega\|_{L^{1}(\Sigma)}\|f\|_{2},
\end{align}
and
\begin{align}\label{huhihi2}
\|f\ast R^{\alpha}\ast \mu_{j}^{\alpha}\ast \Psi_{k}\|_{2}\leq C_{\mathbb{Q},\alpha}2^{-\tau|j-k|}\|\Omega\|_{L^{1}(\Sigma)}\|f\|_{2}.
\end{align}
\end{proposition}
\begin{proof}
We first give the proof of the first statement. Let $\eta$ be a standard smooth cut-off function supported on $\frac{1}{2}\leq \rho(x)\leq 2$ and satisfying $\sum_{\ell\in\mathbb{Z}}\eta(2^{-\ell}\circ x)=1$.  Then we can decompose the kernel $R^{\alpha}$ of the Riesz potential $(-\Delta_{\HH})^{-\alpha/2}$ as follows.
\begin{align}\label{rieszde}
R^{\alpha}(x)=\sum_{\ell\in\mathbb{Z}}R^{\alpha}(x)\eta(2^{-\ell}\circ x)=:\sum_{\ell\in\mathbb{Z}}R_{\ell}^{\alpha}(x).
\end{align}
It can be verified from \eqref{Rie} that
$
\|R_{\ell}^{\alpha}\|_{1}\lesssim 2^{\alpha\ell}.
$
Moreover, it follows from the representation formula \eqref{repre} that $R_{\ell}^{\alpha}(x)=2^{-(\mathbb{Q}-\alpha)\ell}\zeta(2^{-\ell}\circ x)$, where
\begin{align}\label{dezeta}
\zeta(x):=\frac{1}{\Gamma(\alpha/2)}\int_{0}^{\infty}t^{\frac{\alpha}{2}-1}p(t,x)\eta(x)dt
\end{align}
is a smooth function supported in the region $\{\frac{1}{2}\leq \rho(x)\leq 2\}$.

To continue, we divide our proof into six cases.

\textbf{Case 1:}  $k\geq j$.

\textbf{Case 1.1:} If $\ell\leq j$, then we apply Lemma \ref{citetaylor} to see that
\begin{align*}
\|\Psi_{k}\ast R_{\ell}^{\alpha}\ast \mu_{j}^{\alpha}\|_{1}
&\lesssim 2^{-k([\alpha]+1)}\int_{\HH}|R_{\ell}^{\alpha}\ast \mu_{j}^{\alpha}(y)|\rho(y)^{[\alpha]+1}dy\lesssim 2^{-([\alpha]+1)(k-j)}\|R_{\ell}^{\alpha}\ast \mu_{j}^{\alpha}\|_{1}.
\end{align*}
This, in combination with Young's inequality, yields
\begin{align}\label{345}
\|f\ast \Psi_{k}\ast R_{\ell}^{\alpha}\ast \mu_{j}^{\alpha}\|_{2}
\lesssim 2^{-([\alpha]+1)(k-j)}\|R_{\ell}^{\alpha}\|_{1}\|\mu_{j}^{\alpha}\|_{1}\|f\|_{2}
\lesssim 2^{-([\alpha]+1)(k-j)}2^{-\alpha (j-\ell)}\|\Omega\|_{L^{1}(\Sigma)}\|f\|_{2}.
\end{align}

\textbf{Case 1.2:} If $\ell\geq j$, then we apply Lemma \ref{citetaylor} to see that
\begin{align}\label{678}
\|R_{\ell}^{\alpha}\ast \mu_{j}^{\alpha}\|_{1}
&\lesssim 2^{-([\alpha]+1-\alpha)\ell}\int_{\HH}\rho(y)^{[\alpha]+1}d|\mu_{j}^{\alpha}|(y)
\lesssim 2^{-([\alpha]+1-\alpha)(\ell-j)}\|\Omega\|_{L^{1}(\Sigma)}.
\end{align}
This, in combination with Young's inequality, yields
\begin{align}\label{zbbm}
\|f\ast \Psi_{k}\ast R_{\ell}^{\alpha}\ast \mu_{j}^{\alpha}\|_{2}\leq \|R_{\ell}^{\alpha}\ast \mu_{j}^{\alpha}\|_{1}\|f\|_{2}\lesssim 2^{-([\alpha]+1-\alpha)(\ell-j)}\|\Omega\|_{L^{1}(\Sigma)}\|f\|_{2}.
\end{align}

\textbf{Case 1.2.1:} If $\ell\geq k$, then the estimate \eqref{zbbm} is enough.


\textbf{Case 1.2.2:} If $j\leq\ell\leq k$, then we apply Lemma \ref{citetaylor} to see that
\begin{align*}
\|\Psi_{k}\ast R_{\ell}^{\alpha}\ast \mu_{j}^{\alpha}\|_{1}
&\lesssim 2^{-([\alpha]+1)k}\int_{\HH}|R_{\ell}^{\alpha}\ast \mu_{j}^{\alpha}(y)|\rho( y)^{[\alpha]+1}dy\lesssim 2^{-([\alpha]+1)(k-\ell)}\|R_{\ell}^{\alpha}\ast \mu_{j}^{\alpha}\|_{1}.
\end{align*}
This, in combination with the estimate \eqref{678}, implies
\begin{align}\label{456}
\|f\ast \Psi_{k}\ast R_{\ell}^{\alpha}\ast \mu_{j}^{\alpha}\|_{2}\lesssim 2^{-([\alpha]+1)(k-\ell)}\|R_{\ell}^{\alpha}\ast \mu_{j}^{\alpha}\|_{1}\lesssim 2^{-([\alpha]+1)(k-j)}2^{\alpha(\ell-j)}\|\Omega\|_{L^{1}(\Sigma)}\|f\|_{2}.
\end{align}

Back to the proof of case 1, we combine the estimates \eqref{345}, \eqref{zbbm} and \eqref{456} to obtain that
\begin{align*}
&\|f\ast \Psi_{k}\ast R^{\alpha}\ast \mu_{j}^{\alpha}\|_{2}\\&\lesssim \sum_{\ell\leq j}2^{-([\alpha]+1)(k-j)}2^{-\alpha (j-\ell)}\|\Omega\|_{L^{1}(\Sigma)}\|f\|_{2}+\sum_{j\leq\ell\leq k}2^{-([\alpha]+1)(k-j)}2^{\alpha(\ell-j)}\|\Omega\|_{L^{1}(\Sigma)}\|f\|_{2}\\&+\sum_{\ell\geq k}2^{-([\alpha]+1-\alpha)(\ell-j)}\|\Omega\|_{L^{1}(\Sigma)}\|f\|_{2}\lesssim 2^{-\tau(k-j)}\|\Omega\|_{L^{1}(\Sigma)}\|f\|_{2}
\end{align*}
for some constant $\tau>0$.

\textbf{Case 2:} If $k\leq j$, then by Young's inequality,
\begin{align}\label{hbnm}
\|f\ast \Psi_{k}\ast R_{\ell}^{\alpha}\ast \mu_{j}^{\alpha}\|_{2}\lesssim \|\Psi_{k}\|_{1} \|R_{\ell}^{\alpha}\|_{1}\|\mu_{j}^{\alpha}\|_{1}\|f\|_{2}\lesssim 2^{\alpha(\ell-j)}\|\Omega\|_{L^{1}(\Sigma)}\|f\|_{2}.
\end{align}

\textbf{Case 2.1:} If $\ell\leq k$, then the estimate \eqref{hbnm} is enough.

\textbf{Case 2.2:} If $\ell\geq k$, then by the cancellation property of $\Psi_{k}$ and the mean value theorem on homogeneous groups (see \cite{FoSt}), we see that $\|\Psi_{k}\ast R_{\ell}^{\alpha}\|_{1}\lesssim 2^{\alpha\ell}2^{-(\ell-k)}$. Therefore,
\begin{align}\label{zbbn}
\|f\ast \Psi_{k}\ast R_{\ell}^{\alpha}\ast \mu_{j}^{\alpha}\|_{2}\leq \|\Psi_{k}\ast R_{\ell}^{\alpha}\|_{1}\|\mu_{j}^{\alpha}\|_{1}\|f\|_{2}\lesssim 2^{\alpha(\ell-j)}2^{-(\ell-k)}\|\Omega\|_{L^{1}(\Sigma)}\|f\|_{2}.
\end{align}

\textbf{Case 2.2.1:} If $k\leq\ell\leq j$, then taking geometric means of \eqref{hbnm} and \eqref{zbbn}, we see that for any $\gamma\in [0,1]$,
\begin{align*}
\|f\ast \Psi_{k}\ast R_{\ell}^{\alpha}\ast \mu_{j}^{\alpha}\|_{2}
&\lesssim 2^{\alpha\gamma(\ell-j)}2^{\alpha(1-\gamma)(\ell-j)}2^{-(1-\gamma)(\ell-k)}\|\Omega\|_{L^{1}(\Sigma)}\|f\|_{2}\\
&=2^{(\alpha-1+\gamma)(\ell-j)}2^{-(1-\gamma)(j-k)}\|\Omega\|_{L^{1}(\Sigma)}\|f\|_{2}.
\end{align*}
Choosing $\max\{1-\alpha,0\}<\gamma<1$, we obtain
\begin{align}\label{789}
\|f\ast \Psi_{k}\ast R_{\ell}^{\alpha}\ast \mu_{j}^{\alpha}\|_{2}\lesssim2^{\tau(\ell-j)}2^{-\tau(j-k)}\|\Omega\|_{L^{1}(\Sigma)}\|f\|_{2}
\end{align}
for some constant $\tau>0$.


\textbf{Case 2.2.2:} If $\ell\geq j$, then \eqref{678} holds.
This, in combination with Young's inequality, yields
\begin{align}\label{222}
\|f\ast \Psi_{k}\ast R_{\ell}^{\alpha}\ast \mu_{j}^{\alpha}\|_{2}\leq \|R_{\ell}^{\alpha}\ast \mu_{j}^{\alpha}\|_{1}\|f\|_{2}\lesssim 2^{-([\alpha]+1-\alpha)(\ell-j)}\|\Omega\|_{L^{1}(\Sigma)}\|f\|_{2}.
\end{align}
Taking geometric means of the estimates \eqref{zbbn} and \eqref{222}, we obtain that for any $\gamma\in [0,1]$,
\begin{align*}
\|f\ast \Psi_{k}\ast R_{\ell}^{\alpha}\ast \mu_{j}^{\alpha}\|_{2}
&\lesssim 2^{-\gamma([\alpha]+1-\alpha)(\ell-j)}2^{\alpha(1-\gamma)(\ell-j)}2^{-(1-\gamma)(\ell-k)}\|\Omega\|_{L^{1}(\Sigma)}\|f\|_{2}\\
&=2^{(-\gamma[\alpha]+\alpha-1)(\ell-j)}2^{-(1-\gamma)(j-k)}\|\Omega\|_{L^{1}(\Sigma)}\|f\|_{2}.
\end{align*}
Choosing $\gamma=\frac{1}{2}$ when $\alpha\leq 1$ and $\gamma\in(\frac{\alpha-1}{[\alpha]},1)$ when $\alpha>1$ such that $-\gamma[\alpha]+\alpha-1<0$ for $0<\alpha<\mathbb{Q}$, we obtain that
\begin{align}\label{890}
\|f\ast \Psi_{k}\ast R_{\ell}^{\alpha}\ast \mu_{j}^{\alpha}\|_{2}\lesssim2^{-\tau(\ell-j)}2^{-\tau(j-k)}\|\Omega\|_{L^{1}(\Sigma)}\|f\|_{2}
\end{align}
for some constant $\tau>0$.

Back to the proof of case 2, we combine the estimates \eqref{hbnm}, \eqref{789} and \eqref{890} to conclude that
\begin{align*}
&\|f\ast \Psi_{k}\ast R^{\alpha}\ast \mu_{j}^{\alpha}\|_{2}\\
&\lesssim \sum_{\ell\leq k}2^{\alpha(\ell-j)}\|\Omega\|_{L^{1}(\Sigma)}\|f\|_{2}+\sum_{k\leq\ell\leq j}2^{\tau(\ell-j)}2^{-\tau(j-k)}\|\Omega\|_{L^{1}(\Sigma)}\|f\|_{2}+\sum_{\ell\geq j}2^{-\tau(\ell-j)}2^{-\tau(j-k)}\|\Omega\|_{L^{1}(\Sigma)}\|f\|_{2}\\
&\lesssim2^{-\tau(j-k)}\|\Omega\|_{L^{1}(\Sigma)}\|f\|_{2}.
\end{align*}
This finishes the proof of the first statement.

Now we turn to the second statement. Before presenting its proof, we first point out that in the setting of Euclidean space, the second statement is clearly equivalent to the first one. However, generally speaking, the convolution operation cannot be exchanged in homogeneous groups, the proof of the second statement is slightly different from the first one and cannot be directly deduced from the first one.

We divide the proof of the second statement into four cases.

\textbf{Case 1:} If $k\geq j$, then we apply Lemma \ref{citetaylor} to see that
\begin{align}\label{case1}
\|\mu_{j}^{\alpha}\ast \Psi_{k}\|_{1}
&\lesssim 2^{-([\alpha]+1)k}\int_{\HH}\rho(y)^{[\alpha]+1}d|\mu_{j}^{\alpha}|(y)\lesssim 2^{-([\alpha]+1)(k-j)}2^{-\alpha j}\|\Omega\|_{L^{1}(\Sigma)}.
\end{align}

\textbf{Case 1.1:} If $\ell\leq k$, then we apply Young's inequality and the estimate \eqref{case1} to obtain
\begin{align}\label{cca1}
\|f\ast R_{\ell}^{\alpha}\ast \mu_{j}^{\alpha}\ast \Psi_{k}\|_{2}\lesssim 2^{\alpha\ell} \|\mu_{j}^{\alpha}\ast \Psi_{k}\|_{1}\|f\|_{2}\lesssim 2^{-\alpha(k-\ell)}2^{-([\alpha]+1-\alpha)(k-j)}\|\Omega\|_{L^{1}(\Sigma)}\|f\|_{2}.
\end{align}

\textbf{Case 1.2:} If $\ell\geq k$, then we apply Lemma \ref{citetaylor} to see that
\begin{align}\label{simi}
\|R_{\ell}^{\alpha}\ast \mu_{j}^{\alpha}\ast \Psi_{k}\|_{1}
&\lesssim 2^{-([\alpha]+1-\alpha)\ell}\int_{\HH}|\mu_{j}^{\alpha}\ast \Psi_{k}(y)|\rho( y)^{[\alpha]+1}dy\lesssim 2^{-([\alpha]+1-\alpha)\ell}2^{([\alpha]+1)k}\|\mu_{j}^{\alpha}\ast \Psi_{k}\|_{1}.
\end{align}
This, in combination with the estimate \eqref{case1} and Young's inequality, yields
\begin{align}\label{cca2}
\|f\ast R_{\ell}^{\alpha}\ast \mu_{j}^{\alpha}\ast \Psi_{k}\|_{2}\lesssim 2^{-([\alpha]+1-\alpha)(\ell-k)}2^{-([\alpha]+1-\alpha)(k-j)}\|\Omega\|_{L^{1}(\Sigma)}\|f\|_{2}.
\end{align}

Back to the proof of case 1, we combine the estimates \eqref{cca1} and \eqref{cca2} to get that
\begin{align*}
\|f\ast R^{\alpha}\ast \mu_{j}^{\alpha}\ast \Psi_{k}\|_{2}
&\lesssim \sum_{\ell\leq k} 2^{-\alpha(k-\ell)}2^{-([\alpha]+1-\alpha)(k-j)}\|\Omega\|_{L^{1}(\Sigma)}\|f\|_{2}\\
&\qquad+\sum_{\ell\geq k}2^{-([\alpha]+1-\alpha)(\ell-k)}2^{-([\alpha]+1-\alpha)(k-j)}\|\Omega\|_{L^{1}(\Sigma)}\|f\|_{2}\\
&\lesssim2^{-([\alpha]+1-\alpha)(k-j)}\|\Omega\|_{L^{1}(\Sigma)}\|f\|_{2}.
\end{align*}

\textbf{Case 2:} If $k\leq j$, then it follows from \eqref{repre} that for any $i=1,2,\ldots,n$, we have $|Y_{i} R_{\ell}^{\alpha}(x)|\lesssim 2^{-(\mathbb{Q}-\alpha+1)\ell}$ and therefore,
\begin{align*}
|Y_{i} R_{\ell}^{\alpha}\ast \mu_{j}^{\alpha}(x)|\lesssim 2^{-(\mathbb{Q}-\alpha+1)\ell}2^{-\alpha j}\|\Omega\|_{L^{1}(\Sigma)},\ {\rm uniformly\ in}\ x\in\mathbb{H}.
\end{align*}

\textbf{Case 2.1:} If $\ell\leq j$, then we apply Lemma \ref{citetaylor} to see that
\begin{align*}
\|R_{\ell}^{\alpha}\ast \mu_{j}^{\alpha}\ast \Psi_{k}\|_{1}
&\lesssim 2^{-(\mathbb{Q}-\alpha+1)\ell}2^{(\mathbb{Q}-\alpha)j}\|\Omega\|_{L^{1}(\Sigma)}\int_{\HH}\rho(y)|\Psi_{k}(y)|dy\lesssim 2^{-(\mathbb{Q}-\alpha+1)\ell}2^{(\mathbb{Q}-\alpha)j}2^{k}\|\Omega\|_{L^{1}(\Sigma)}.
\end{align*}
This, in combination with Young's inequality, implies
\begin{align}\label{790}
\|f\ast R_{\ell}^{\alpha}\ast \mu_{j}^{\alpha}\ast \Psi_{k}\|_{2}\lesssim 2^{-(\mathbb{Q}-\alpha+1)\ell}2^{(\mathbb{Q}-\alpha)j}2^{k}\|\Omega\|_{L^{1}(\Sigma)}\|f\|_{2}.
\end{align}
By taking geometric means of the estimates \eqref{hbnm} and \eqref{790}, we obtain that for any $\gamma_{1}\in [0,1]$,
\begin{align}\label{l0}
\|f\ast R_{\ell}^{\alpha}\ast \mu_{j}^{\alpha}\ast \Psi_{k}\|_{2}\lesssim 2^{\gamma_{1}\alpha(\ell-j)} 2^{-(1-\gamma_{1})(\mathbb{Q}-\alpha+1)\ell}2^{(1-\gamma_{1})(\mathbb{Q}-\alpha)j}2^{(1-\gamma_{1})k}\|\Omega\|_{L^{1}(\Sigma)}\|f\|_{2}.
\end{align}

\textbf{Case 2.2:} If $\ell\geq j$, then we apply Lemma \ref{citetaylor} to see that
\begin{align*}
\|R_{\ell}^{\alpha}\ast \mu_{j}^{\alpha}\ast \Psi_{k}\|_{1}
&\lesssim 2^{-(1-\alpha)\ell}2^{-\alpha j}\|\Omega\|_{L^{1}(\Sigma)}\int_{\rho(x)\lesssim 2^{k}}\rho(y)|\Psi_{k}(y)|dy\lesssim 2^{-(1-\alpha)\ell}2^{-\alpha j}2^{k}\|\Omega\|_{L^{1}(\Sigma)}.
\end{align*}

This, in combination with Young's inequality, implies
\begin{align}\label{730}
\|f\ast R_{\ell}^{\alpha}\ast \mu_{j}^{\alpha}\ast \Psi_{k}\|_{2}\lesssim 2^{-(1-\alpha)\ell}2^{-\alpha j}2^{k}\|\Omega\|_{L^{1}(\Sigma)}\|f\|_{2}.
\end{align}

Moreover, from the proof of \eqref{simi} and the fact $\supp \mu_{j}^{\alpha}\ast \Psi_{k}\subset\{x\in\HH: \rho(x)\lesssim 2^{j}\}$ we see that
\begin{align*}
\|R_{\ell}^{\alpha}\ast \mu_{j}^{\alpha}\ast \Psi_{k}\|_{1}
&\lesssim 2^{\alpha\ell}\int_{\HH}|\mu_{j}^{\alpha}\ast \Psi_{k}(y)|\rho(2^{-\ell}\circ y)^{[\alpha]+1}dy\lesssim 2^{-([\alpha]+1-\alpha)\ell}2^{([\alpha]+1)j}\|\mu_{j}^{\alpha}\ast \Psi_{k}\|_{1}.
\end{align*}
This, in combination with Young's inequality $\|\mu_{j}^{\alpha}\ast \Psi_{k}\|_{1}\lesssim 2^{-\alpha j}\|\Omega\|_{L^{1}(\Sigma)}$, yields
\begin{align}\label{278}
\|f\ast R_{\ell}^{\alpha}\ast \mu_{j}^{\alpha}\ast \Psi_{k}\|_{2}\lesssim 2^{([\alpha]+1-\alpha)(j-\ell)}\|\Omega\|_{L^{1}(\Sigma)}\|f\|_{2}.
\end{align}
Taking geometric means of the estimates \eqref{730} and \eqref{278}, we obtain that for any $\gamma_{2}\in [0,1]$,
\begin{align}\label{j0}
\|f\ast R_{\ell}^{\alpha}\ast \mu_{j}^{\alpha}\ast \Psi_{k}\|_{2}\lesssim 2^{\gamma_{2}([\alpha]+1-\alpha)(j-\ell)}2^{-(1-\gamma_{2})(-\alpha+1)\ell}2^{-(1-\gamma_{2})\alpha j}2^{(1-\gamma_{2}) k}\|\Omega\|_{L^{1}(\Sigma)}\|f\|_{2}.
\end{align}
Back to the proof of case 2, choosing $\gamma_{1}\in(\frac{\mathbb{Q}+1-\alpha}{\mathbb{Q}+1},1)$ and $\gamma_{2}=\frac{1}{2}$ when $\alpha<1$, and  $\frac{\alpha-1}{[\alpha]}<\gamma_{2}<1$ when $\alpha\geq 1$, respectively, such that $\gamma_{1}\alpha-(1-\gamma_{1})(\mathbb{Q}-\alpha+1)>0$ and $\gamma_{2}([\alpha]+1-\alpha)+(1-\gamma_{2})(-\alpha+1)>0$, we combine the estimates \eqref{l0} and \eqref{j0} to obtain that
\begin{align*}
\|f\ast R^{\alpha}\ast \mu_{j}^{\alpha}\ast \Psi_{k}\|_{2}
&\lesssim \sum_{\ell\leq j}2^{\gamma_{1}\alpha(\ell-j)} 2^{-(1-\gamma_{1})(\mathbb{Q}-\alpha+1)\ell}2^{(1-\gamma_{1})(\mathbb{Q}-\alpha)j}2^{(1-\gamma_{1})k}\|\Omega\|_{L^{1}(\Sigma)}\|f\|_{2}\\
&\qquad+\sum_{\ell\geq j}2^{\gamma_{2}([\alpha]+1-\alpha)(j-\ell)}2^{-(1-\gamma_{2})(-\alpha+1)\ell}2^{-(1-\gamma_{2})\alpha j}2^{(1-\gamma_{2}) k}\|\Omega\|_{L^{1}(\Sigma)}\|f\|_{2}\\
&\lesssim 2^{-\tau(j-k)}\|\Omega\|_{L^{1}(\Sigma)}\|f\|_{2}
\end{align*}
for some constant $\tau>0$.
This finishes the proof of Proposition \ref{keypro}.
\end{proof}

Let $0< \alpha<\mathbb{Q}$ and $\{\mu_{j}^{\alpha}\}_{j\in\mathbb{Z}}$ be a family of Borel measures defined as above. In addition, we let $\{r_{k,j}(t)\}_{k,j\in\mathbb{Z}}$ be a family of functions such that for any $k,j\in\mathbb{Z}$ and $t\in[0,1]$, $r_{k,j}(t)=1$ or $r_{k,j}(t)=-1$. We consider the operators $G_{j}^{\alpha}(t)$, associated with $\{\mu_{k}^{\alpha}\}_{k\in\mathbb{Z}}$ and $\{r_{k,j}(t)\}_{k\in\mathbb{Z}}$, defined by
\begin{align}\label{equa}
G_{j}^{\alpha}(t)f:=\sum_{k\in\mathbb{Z}}r_{k,j}(t)f\ast\Psi_{k-j}\ast R^{\alpha}\ast \mu_{k}^{\alpha}=:\sum_{k\in\mathbb{Z}}G_{j,k}^{\alpha}(t)f.
\end{align}
For simplicity, we extend this definition to the critical case $\alpha=0$ by setting $d\mu_{j}^{0}=A_{j}^{0}K_{0}^{0}dx$ and
\begin{align}\label{equa001}
G_{j}^{0}(t)f:=\sum_{k\in\mathbb{Z}}r_{k,j}(t)f\ast\Psi_{k-j}\ast A_{k}^{0}K_{0}^{0}=:\sum_{k\in\mathbb{Z}}G_{j,k}^{0}(t)f.
\end{align}
We have the following $L^{2}$ decay lemma.
\begin{lemma}\label{ckt}
Let $0< \alpha<\mathbb{Q}$ and $q>1$, then there exist constants $C_{\mathbb{Q},\alpha}$, $C_{\mathbb{Q},q}>0$ and $\tau>0$ (independent of the value of $r_{k,j}(t)$) such that for any $j\in\mathbb{Z}$,
\begin{align}\label{alpha0}
\|G_{j}^{\alpha}(t)f\|_{2}\leq C_{\mathbb{Q},\alpha}2^{-\tau |j|}\|\Omega\|_{L^{1}(\Sigma)}\|f\|_{2},
\end{align}
and
\begin{align}\label{alphaaa}
\|G_{j}^{0}(t)f\|_{2}\leq C_{\mathbb{Q},q}2^{-\tau |j|}\|\Omega\|_{L^{q}(\Sigma)}\|f\|_{2}.
\end{align}
\end{lemma}
\begin{proof}
We first give the proof of estimate \eqref{alpha0}. By Cotlar-Knapp-Stein Lemma (see \cite{s93}), it suffices to show that:
\begin{align}\label{CKS}
\|(G_{j,k}^{\alpha}(t))^{*}G_{j,k^{\prime}}^{\alpha}(t)\|_{2\rightarrow 2}+\|G_{j,k^{\prime}}^{\alpha}(t)(G_{j,k}^{\alpha}(t))^{*}\|_{2\rightarrow 2}\leq C_{\mathbb{Q}}2^{-2\tau |j|}2^{-\tau|k-k^{\prime}|}\|\Omega\|_{L^{1}(\Sigma)}^{2}.
\end{align}
We only estimate the first term, since the estimate of the second one is similar. Note that
\begin{align}\label{dede}
(G_{j,k}^{\alpha}(t))^{*}G_{j,k^{\prime}}^{\alpha}(t)f&=r_{k^{\prime},j}(t)r_{k,j}(t)f\ast \Psi_{k^{\prime}-j}\ast R^{\alpha}\ast \mu_{k^{\prime}}^{\alpha}\ast \widetilde{\mu_{k}^{\alpha}}\ast R^{\alpha}\ast \Psi_{k-j}\nonumber\\
&=\sum_{\ell\in\mathbb{Z}}\sum_{\ell^{\prime}\in\mathbb{Z}}r_{k^{\prime},j}(t)r_{k,j}(t)f\ast \Psi_{k^{\prime}-j}\ast R^{\alpha}\ast \mu_{k^{\prime}}^{\alpha}\ast\Psi_{\ell}\ast\Psi_{\ell^{\prime}}\ast \widetilde{\mu_{k}^{\alpha}}\ast R^{\alpha}\ast \Psi_{k-j}.
\end{align}

On the one hand, it follows from Young's inequality, Proposition \ref{keypro} and its dual version that
\begin{align}\label{f1}
&\|(f\ast \Psi_{k^{\prime}-j})\ast (R^{\alpha}\ast \mu_{k^{\prime}}^{\alpha}\ast\Psi_{\ell})\ast(\Psi_{\ell^{\prime}}\ast \widetilde{\mu_{k}^{\alpha}}\ast R^{\alpha})\ast \Psi_{k-j}\|_{2}\nonumber\\
&\lesssim \|(f\ast \Psi_{k^{\prime}-j})\ast (R^{\alpha}\ast \mu_{k^{\prime}}^{\alpha}\ast\Psi_{\ell})\ast(\Psi_{\ell^{\prime}}\ast \widetilde{\mu_{k}^{\alpha}}\ast R^{\alpha})\|_{2}\nonumber\\
&\lesssim 2^{-\tau|k-\ell^{\prime}|}\|\Omega\|_{L^{1}(\Sigma)}\|(f\ast \Psi_{k^{\prime}-j})\ast (R^{\alpha}\ast \mu_{k^{\prime}}^{\alpha}\ast\Psi_{\ell})\|_{2}\nonumber\\
&\lesssim 2^{-\tau|k-\ell^{\prime}|}2^{-\tau|\ell-k^{\prime}|}\|\Omega\|_{L^{1}(\Sigma)}^{2}\|f\ast \Psi_{k^{\prime}-j}\|_{2}\nonumber\\
&\lesssim 2^{-\tau|k-\ell^{\prime}|}2^{-\tau|\ell-k^{\prime}|}\|\Omega\|_{L^{1}(\Sigma)}^{2}\|f\|_{2}.
\end{align}

On the other hand, using the cancellation and smoothness properties of $\Psi_{\ell}$ and $\Psi_{\ell^\prime}$ to get that
\begin{align}\label{cacan}
\|\Psi_{\ell}\ast \Psi_{\ell^{\prime}}\|_{1}\lesssim 2^{-|\ell-\ell^{\prime}|}.
\end{align}
This, in combination with Young's inequality, Proposition \ref{keypro} and its dual version, indicates that
\begin{align}\label{f2}
&\|f\ast (\Psi_{k^{\prime}-j}\ast R^{\alpha}\ast \mu_{k^{\prime}}^{\alpha})\ast(\Psi_{\ell}\ast\Psi_{\ell^{\prime}})\ast (\widetilde{\mu_{k}^{\alpha}}\ast R^{\alpha}\ast \Psi_{k-j})\|_{2}\nonumber\\
&\lesssim 2^{-\tau|j|}\|\Omega\|_{L^{1}(\Sigma)}\|f\ast (\Psi_{k^{\prime}-j}\ast R^{\alpha}\ast \mu_{k^{\prime}}^{\alpha})\ast(\Psi_{\ell}\ast\Psi_{\ell^{\prime}})\|_{2}\nonumber\\
&\lesssim 2^{-\tau|j|}2^{-|\ell-\ell^{\prime}|}\|\Omega\|_{L^{1}(\Sigma)}\|f\ast \Psi_{k^{\prime}-j}\ast R^{\alpha}\ast \mu_{k^{\prime}}^{\alpha}\|_{2}\nonumber\\
&\lesssim 2^{-2\tau|j|}2^{-|\ell-\ell^{\prime}|}\|\Omega\|_{L^{1}(\Sigma)}^{2}\|f\|_{2}.
\end{align}

Taking geometric mean of \eqref{f1} with \eqref{f2}, we see that
\begin{align*}
\|f\ast (\Psi_{k^{\prime}-j}\ast R^{\alpha}\ast \mu_{k^{\prime}}^{\alpha})\ast(\Psi_{\ell}\ast\Psi_{\ell^{\prime}})\ast (\widetilde{\mu_{k}^{\alpha}}\ast R^{\alpha}\ast \Psi_{k-j})\|_{2}\lesssim 2^{-\tau|j|}2^{-\tau|k-k^{\prime}|}2^{-\tau|k-\ell^{\prime}|}2^{-\tau|\ell-k^{\prime}|}\|\Omega\|_{L^{1}(\Sigma)}^{2}\|f\|_{2}.
\end{align*}
Combining this inequality with the equality \eqref{dede}, we obtain the estimate \eqref{CKS} and therefore the first statement of Lemma \ref{ckt}.

For the second statement, we recall that Tao \cite[Proposition 4.1]{Tao} applied iterated $(TT^{*})^{N}$ method to obtain the following inequality with $q=\infty$ and then Sato \cite[Lemma 1]{sato} extended it to general $q>1$:
there exist constants $C_{\mathbb{Q},q}>0$ and $\tau>0$ such that for any $j,k\in\mathbb{Z}$,
\begin{align}\label{kkkey}
\|f\ast A_{j}^{0}K_{0}^{0}\ast \Psi_{k}\|_{2}\leq C_{\mathbb{Q},q}2^{-\tau|j-k|}\|f\|_{2}\|\Omega\|_{L^{q}(\Sigma)}.
\end{align}
Thus, the second statement can be shown by repeating the above argument and with Proposition \ref{keypro} replaced by \eqref{kkkey}. This finishes the proof of Lemma \ref{ckt}.
\end{proof}
\begin{remark}
In what follows, $\mu_{j}^{\alpha}$ will be endowed with several exact values. On the one hand, we will choose $d\mu_{j}^{\alpha}:=A_{j}^{\alpha}K_{\alpha}^{0}dx$, $ B_{j}^{\alpha}\Omega dx$ and $B_{j,t}^{\alpha}\Omega dx$, where $B_{j,t}^{\alpha}\Omega$ is defined in Section \ref{nondyadicsection}. On the other hand, we also choose $d\mu_{j}^{\alpha}:=2^{-j(\mathbb{Q}-1+\alpha)}t^{-\mathbb{Q}-\alpha}\Omega d\sigma_{t2^{j}}$, where $\sigma_{r}$ is the unique Radon measure on $r\Sigma:=\{x\in\HH:\rho(x)=r\}$ satisfying for any $f\in L^{1}(\Sigma)$,
\begin{align*}
\int_{\Sigma}f(\theta)d\sigma(\theta)=\int_{r\Sigma}f(r^{-1}\circ \theta)d\sigma_{r}(\theta)r^{-(\mathbb{Q}-1)}.
\end{align*}
In both cases, it can be seen from the proof of Proposition \ref{keypro} and Lemma \ref{ckt} that the bounds on the right hand side of \eqref{huhihi}, \eqref{huhihi2} and \eqref{alpha0} are independent of $t\in[1,2]$. We would like to mention that in the Euclidean setting, for the particular choice  $d\mu_{j}^{0}=B_{j}^{0}\Omega dx$, the corresponding $L^{2}$ decay estimate was obtained in \cite{DR} (see also \cite{HRT}) via a standard argument of Fourier transform and Plancherel's theorem. Actually, in this setting, one can also apply this argument to obtain the corresponding $L^{2}$ decay estimates for the other choices of $\mu_{j}^{\alpha}$ when $0\leq\alpha<\mathbb{Q}$.
\end{remark}

\section{Unweighted $L^{p}$ estimate}\label{unweightsection}
\setcounter{equation}{0}

Throughout this section, unless we mention the contrary, we suppose that $\Omega$ satisfies the following assumption.

\textbf{Assumption: } Let $0<\alpha<\mathbb{Q}$. Suppose that $\Omega\in L^{1}(\Sigma)$ and satisfies the cancellation condition of order $[\alpha]$.
\subsection{Boundedness of $T_{\Omega,\alpha}$}\label{sssss2}
In this section, we show the $(L^{p}(\HH),L_{\alpha}^{p}(\HH))$ estimate of $T_{\Omega,\alpha}$.
\begin{proposition}\label{prop1}
Let $0<\alpha<\mathbb{Q}$. Suppose that $\Omega\in L^{1}(\Sigma)$ and satisfies the cancellation condition of order $[\alpha]$, then for any $1<p<\infty$,
\begin{align*}
\|T_{\Omega,\alpha}f\|_{L^{p}(\HH)}\leq C_{\mathbb{Q},\alpha,p}\|\Omega\|_{L^{1}(\Sigma)}\|f\|_{L_{\alpha}^{p}(\HH)}
\end{align*}
for some constant $C_{\mathbb{Q},\alpha,p}$ independent of $\Omega$.
\end{proposition}

To show Proposition \ref{prop1}, we first note that $S_{j}f\rightarrow f$ as $j\rightarrow -\infty$, which implies that:
\begin{align}\label{sldksl}
V_{k}^{\alpha}(-\Delta_{\HH})^{-\alpha/2}=V_{k}^{\alpha}(-\Delta_{\HH})^{-\alpha/2}S_{k}+\sum_{j=1}^{\infty}V_{k}^{\alpha}(-\Delta_{\HH})^{-\alpha/2}(S_{k-j}-S_{k-(j-1)}).
\end{align}
In this way, $T_{\Omega,\alpha}(-\Delta_{\HH})^{-\alpha/2}=\sum_{j=0}^{\infty}\tilde{T}_{j}^{\alpha}$, where
$
\tilde{T}_{0}^{\alpha}:=\sum_{k\in\mathbb{Z}}V_{k}^{\alpha}(-\Delta_{\HH})^{-\alpha/2}S_{k},
$
, and for $j\geq 1$,
\begin{align}\label{tj1}
\tilde{T}_{j}^{\alpha}:=\sum_{k\in\mathbb{Z}}V_{k}^{\alpha}(-\Delta_{\HH})^{-\alpha/2}(S_{k-j}-S_{k-(j-1)}).
\end{align}
\begin{lemma}\label{newww}
There exist  constants $C_{\mathbb{Q},\alpha}>0$ and $\tau>0$, such that for any $j\geq0$,
\begin{align}\label{djdlp}
\|\tilde{T}_{j}^{\alpha}f\|_{2}\leq C_{\mathbb{Q},\alpha}2^{-\tau j}\|\Omega\|_{L^{1}(\Sigma)}\|f\|_{2}.
\end{align}
\end{lemma}
\begin{proof}

Note that $\tilde{T}_{0}^{\alpha}=\sum_{j=-\infty}^{0}G_{j}^{\alpha}$, with $\mu_{j}^{\alpha}$ and $r_{k,j}(t)$ chosen to be $B_{j}^{\alpha}\Omega$ and the constant function $1$, respectively.  This, together with Lemma \ref{ckt}, ends the proof of Lemma \ref{newww}.
\end{proof}

Next we study the kernel of the operator $\tilde{T}_{j}^{\alpha}$. To this end, we denote
$$K_{j}^{\alpha}:=\sum_{k\in\mathbb{Z}}\Delta[2^{k-j}]\phi\ast R^{\alpha}\ast B_{k}^{\alpha}\Omega.$$
\begin{lemma}\label{hormander}
The kernel $K_{j}^{\alpha}$ satisfies the following H\"{o}rmander condition: there exists a constant $C_{\mathbb{Q},\alpha}>0$ such that for any $j\geq 0$,
\begin{align*}
\int_{\rho(x)\geq 2A_{0}\rho(y)}|K_{j}^{\alpha}(y^{-1}x)-K_{j}^{\alpha}(x)|dx\leq C_{\mathbb{Q},\alpha}(1+j)\|\Omega\|_{L^{1}(\Sigma)}.
\end{align*}
\end{lemma}

\begin{proof}
We will postpone the proof of this lemma to Lemma \ref{hormander222}, in which we will show a stronger statement.
\end{proof}

\begin{lemma}\label{weak11}
For any $1<p<\infty$, there exist  constants $C_{\mathbb{Q},\alpha}$, $C_{\mathbb{Q},\alpha,p}>0$ and $\tau>0$ such that for any $j\geq 0$,
\begin{align}\label{weaktype}
\|\tilde{T}_{j}^{\alpha}f\|_{L^{1,\infty}}\leq C_{\mathbb{Q},\alpha}(1+j)\|\Omega\|_{L^{1}(\Sigma)}\|f\|_{1},
\end{align}
and
\begin{align*}
\|\tilde{T}_{j}^{\alpha}f\|_{p}\leq C_{\mathbb{Q},\alpha,p}2^{-\tau j}\|\Omega\|_{L^{1}(\Sigma)}\|f\|_{p}.
\end{align*}
\end{lemma}
\begin{proof}
By Lemmas \ref{newww} and \ref{hormander}, the first statement is the consequence of a standard argument of Calder\'{o}n-Zygmund decomposition (see for example \cite{JDbook,s93}). Then, the second one can be obtained by interpolating \eqref{djdlp} with \eqref{weaktype}, and by a standard duality argument.
\end{proof}

{\it Proof of Proposition \ref{prop1}}.

By the equality \eqref{sldksl} and Lemma \ref{weak11},
\begin{align*}
\|T_{\Omega,\alpha}(-\Delta_{\HH})^{-\alpha/2}f\|_{p}\leq \sum_{j=0}^{\infty}\|\tilde{T}_{j}^{\alpha}f\|_{p}\lesssim \sum_{j=0}^{\infty}2^{-\tau j}\|\Omega\|_{L^{1}(\Sigma)}\|f\|_{p}\lesssim \|\Omega\|_{L^{1}(\Sigma)}\|f\|_{p}.
\end{align*}
Equivalently,
\begin{align*}
\|T_{\Omega,\alpha}f\|_{L^{p}}\lesssim\|\Omega\|_{L^{1}(\Sigma)}\|f\|_{L_{\alpha}^{p}}.
\end{align*}
This completes the proof of Proposition \ref{prop1}.
\hfill $\square$
\begin{remark}\label{remmm}
With a slight modification, the above arguments can be also applied to give an alternative proof of the $L^{p}$ boundedness of $T_{\Omega,0}$, which was shown in \cite{Tao} and \cite{sato} by two different approaches, provided $\Omega\in  L\log^{+} L(\Sigma)$ and $\int_{\Sigma}\Omega(\theta) d\sigma(\theta)=0$. For the convenience of the readers, we sketch the proof below.
\end{remark}
{\it Proof of Remark \ref{remmm}}. Let $E_{0}:=\{x\in\Sigma:|\Omega(x)|\leq 2\}$ and $E_{m}:=\{x\in\Sigma:2^{2^{m-1}}<|\Omega(x)|\leq2^{2^{m}}\}$ for $m\geq 1$. In addition, for any $m\geq 0$, set
\begin{align*}
\Omega_{m}(x):=\Omega(x)\chi_{E_{m}}(x)-\frac{1}{\sigma(\Sigma)}\int_{E_{m}}\Omega(\theta)d\sigma(\theta).
\end{align*}
Then $\int_{\Sigma}\Omega_{m}(x)d\sigma(x)=0$ and $\Omega(x)=\sum_{m=0}^{\infty}\Omega_{m}(x)$. Let $T_{m,j}^{0}$ be the operator defined in \eqref{hjhbkk} with  $\alpha=0$ and $\Omega$ replaced by $\Omega_{m}$. In addition, for any $j\geq 1$, let
\begin{align*}
\tilde{T}_{m,j}^{+}:=\sum_{k\in\mathbb{Z}}T_{m,k}^{0}(S_{k+2^{j-1}}-S_{k+2^{j}})\ {\rm and}\ \tilde{T}_{m,j}^{-}:=\sum_{k\in\mathbb{Z}}T_{m,k}^{0}(S_{k-2^{j}}-S_{k-2^{j-1}}).
\end{align*}
Then similar to the proof of Lemma \ref{weak11}, it can be verified that for any $j\geq 1$,
\begin{align}\label{newnew}
\|\tilde{T}_{m,j}^{\pm}f\|_{L^{1,\infty}}\lesssim 2^{j}\|\Omega_{m}\|_{L^{1}(\Sigma)}\|f\|_{1}.
\end{align}
This, along with \eqref{alphaaa} (where $\Omega$ is replaced by $\Omega_{m}$) and interpolation, shows that for any $1<p\leq 2$,
\begin{align}\label{uiouoi}
\|\tilde{T}_{m,j}^{\pm}f\|_{p}\lesssim 2^{-\tau 2^{j}}\|\Omega_{m}\|_{L^{\infty}(\Sigma)}\|f\|_{p}
\end{align}
for some constant $\tau>0$.

Moreover, by Young's inequality, for any  $j,k\in\mathbb{Z}$, we have
\begin{align*}
\|T_{m,k}^{0}(S_{k-j}-S_{k-(j-1)})f\|_{2}\lesssim \|\Omega_{m}\|_{L^{1}(\Sigma)}\|f\|_{2}.
\end{align*}
Next, similar to the proof of Lemma \ref{ckt}, it can be verified by Cotlar-Knapp-Stein Lemma  that
\begin{align}
\|\tilde{T}_{m,j}^{\pm}f\|_{2}\lesssim 2^{j}\|\Omega_{m}\|_{L^{1}(\Sigma)}\|f\|_{2},
\end{align}
which interpolating with \eqref{newnew}, in which $\Omega$ is replaced by $\Omega_{m}$, yields that for any $1<p\leq 2$,
\begin{align}\label{yuiiyiu}
\|\tilde{T}_{m,j}^{\pm}f\|_{p}\lesssim 2^{j}\|\Omega_{m}\|_{L^{1}(\Sigma)}\|f\|_{p}.
\end{align}

Combining \eqref{uiouoi} and \eqref{yuiiyiu}, we obtain that
\begin{align*}
\|T_{\Omega,0}f\|_{p}&\leq\sum_{m=0}^{\infty}\sum_{j\geq 1}\|\tilde{T}_{m,j}^{+}f\|_{p}+\|\tilde{T}_{m,j}^{-}f\|_{p}\\
&\lesssim \sum_{m=0}^{\infty}\sum_{2^{j}\leq \lambda 2^{m}}2^{j}\|\Omega_{m}\|_{L^{1}(\Sigma)}\|f\|_{p}+\sum_{m=0}^{\infty}\sum_{2^{j}>\lambda 2^m}2^{-\tau 2^{j}}\|\Omega_{m}\|_{L^{\infty}(\Sigma)}\|f\|_{p}\\
&\lesssim \sum_{m=0}^{\infty}\lambda 2^m \|\Omega_{m}\|_{L^{1}(\Sigma)}\|f\|_{p}+\sum_{m=0}^{\infty}2^{(-\tau\lambda+1)2^m}\|f\|_{p}\\
&\lesssim\left(\int_{\Sigma}|\Omega(\theta)|\log(2+|\Omega(\theta)|)d\sigma(\theta)+1\right) \|f\|_{p},
\end{align*}
where we choose $\lambda>1/\tau$. This, along with a duality argument, shows the proof of Remark \ref{remmm}.
\hfill $\square$
\subsection{Fundamental reduction of Theorem \ref{main1}}\label{Basic reduction}
To begin with, denote
\begin{align}\label{defi1}
M_{\Omega,\alpha}f(x):=\sup\limits_{k>0,\varepsilon\in [2^{k},2^{k+1})}\left|\int_{\varepsilon<\rho(y^{-1}x)\leq2^{k+1}}\frac{\Omega(y^{-1}x)}{\rho(y^{-1}x)^{\mathbb{Q}+\alpha}}f(y)dy\right|,
\end{align}
and
\begin{align}\label{defi2}
T_{\Omega,\alpha}^{k}f(x):=\int_{\rho(y^{-1}x)> 2^{k+1}}\frac{\Omega(y^{-1}x)}{\rho(y^{-1}x)^{\mathbb{Q}+\alpha}}f(y)dy.
\end{align}
Since there exists a unique $k\in\mathbb{Z}$ such that $2^{k}\leq\varepsilon< 2^{k+1}$, we have
\begin{align}\label{maximalcontrol}
T_{\Omega,\alpha}^{\#}f(x)&\leq\sup\limits_{\varepsilon>0}\left|\int_{\varepsilon<\rho(y^{-1}x)\leq 2^{[\log \varepsilon]+1}}\frac{\Omega(y^{-1}x)}{\rho(y^{-1}x)^{\mathbb{Q}+\alpha}}f(y)dy\right|+\sup\limits_{\varepsilon>0}\left|\int_{\rho(y^{-1}x)> 2^{[\log \varepsilon]+1}}\frac{\Omega(y^{-1}x)}{\rho(y^{-1}x)^{\mathbb{Q}+\alpha}}f(y)dy\right|\nonumber\\
&=M_{\Omega,\alpha}f(x)+\sup\limits_{k\in\mathbb{Z}}|T_{\Omega,\alpha}^{k}f(x)|.
\end{align}
Hence, Theorem \ref{main1} can be reduced to showing the following two propositions.
\begin{proposition}\label{nondyadicprop}
For any $1<p<\infty$, there exists a constant $C_{\mathbb{Q},\alpha,p}>0$ such that for any $j\geq 0$,
\begin{align}\label{nondyadic}
\|M_{\Omega,\alpha}(-\Delta_{\HH})^{-\alpha/2}f\|_{p}\leq C_{\mathbb{Q},\alpha,p}\|\Omega\|_{L^{1}(\Sigma)}\|f\|_{p}.
\end{align}
\end{proposition}
\begin{proposition}\label{dyadicprop}
For any $1<p<\infty$, there exists a constant $C_{\mathbb{Q},\alpha,p}>0$ such that for any $j\geq 0$,
\begin{align}\label{dyadic}
\Big\|\sup\limits_{k\in\mathbb{Z}}|T_{\Omega,\alpha}^{k}(-\Delta_{\HH})^{-\alpha/2}f|\Big\|_{p}\leq C_{\mathbb{Q},\alpha,p}\|\Omega\|_{L^{1}(\Sigma)}\|f\|_{p}.
\end{align}
\end{proposition}
We will give the proof of Propositions \ref{nondyadicprop} and \ref{dyadicprop} in Sections \ref{nondyadicsection} and \ref{dyadicsection}, respectively.
\subsection{$L^{p}$ estimate of non-dyadic maximal function}\label{nondyadicsection}
In this subsection, we give the proof of the Proposition \ref{nondyadicprop}. To begin with, for any $t\in[1,2)$ and $j\in\mathbb{N}$, denote
\begin{align*}
B_{j,t}^{\alpha}\Omega(x):=\frac{\Omega(x)}{\rho(x)^{\mathbb{Q}+\alpha}}\chi_{t2^{j}< \rho(x)\leq 2^{j+1}},
\end{align*}
and $V_{j,t}^{\alpha}f:=f\ast B_{j,t}^{\alpha}\Omega$.
Noting that $S_{j}f\rightarrow f$ as $j\rightarrow -\infty$, we have the following identity:
\begin{align}\label{vkj}
V_{k,t}^{\alpha}(-\Delta_{\HH})^{-\alpha/2}=V_{k,t}^{\alpha}(-\Delta_{\HH})^{-\alpha/2}S_{k}+\sum_{j=1}^{\infty}V_{k,t}^{\alpha}(-\Delta_{\HH})^{-\alpha/2}(S_{k-j}-S_{k-(j-1)})=:\sum\limits_{j=0}^{\infty}V_{k,j,t}^{\alpha},
\end{align}
where $V_{k,0,t}^{\alpha}:=V_{k,t}^{\alpha}(-\Delta_{\HH})^{-\alpha/2}S_{k}$ and $V_{k,j,t}^{\alpha}:=V_{k,t}^{\alpha}(-\Delta_{\HH})^{-\alpha/2}(S_{k-j}-S_{k-(j-1)})$ for $j\geq 1$. Denote
\begin{align}\label{denotee}
K_{k,j,t}^{\alpha}:=\Delta[2^{k-j}]\phi\ast R^{\alpha}\ast B_{k,t}^{\alpha}\Omega.
\end{align}
\begin{lemma}\label{hormander222}
The kernel $K_{k,j,t}^{\alpha}$ satisfies the following uniform H\"{o}rmander condition: there exists a constant $C_{\mathbb{Q},\alpha}>0$ such that for any $j\geq 0$,
\begin{align*}
\int_{\rho(x)\geq 2A_{0}\rho(y)}\sup\limits_{t\in[1,2)}\sum_{k\in\mathbb{Z}}|K_{k,j,t}^{\alpha}(y^{-1}x)-K_{k,j,t}^{\alpha}(x)|dx\leq C_{\mathbb{Q},\alpha}(1+j)\|\Omega\|_{L^{1}(\Sigma)}.
\end{align*}
\end{lemma}
\begin{proof}
By the fundamental theorem of calculus, we note that
\begin{align*}
K_{k,j,t}^{\alpha}(x)=-\int_{t}^{2}\frac{d}{ds}K_{k,j,s}^{\alpha}(x)ds.
\end{align*}
Hence,

\begin{align*}
&\int_{\rho(x)\geq 2A_{0}\rho(y)}\sup\limits_{t\in[1,2)}\sum_{k\in\mathbb{Z}}|K_{k,j,t}^{\alpha}(y^{-1}x)-K_{k,j,t}^{\alpha}(x)|dx\\
&\leq \int_{1}^{2}\int_{\rho(x)\geq 2A_{0}\rho(y)}\sum_{k\in\mathbb{Z}}\left|\frac{d}{ds}\Delta[2^{k-j}]\phi\ast R^{\alpha}\ast B_{k,s}^{\alpha}\Omega(y^{-1}x)-\frac{d}{ds}\Delta[2^{k-j}]\phi\ast R^{\alpha}\ast B_{k,s}^{\alpha}\Omega(x)\right|dxds.
\end{align*}
A direct calculation yields
\begin{align}\label{recalll}
\frac{d}{ds}R^{\alpha}\ast B_{k,s}^{\alpha}\Omega(x)
&=\frac{d}{ds}\left(\int_{s2^{k}}^{ 2^{k+1}}\int_{\Sigma}R^{\alpha}(x(r\circ\theta)^{-1})\Omega(\theta)d\sigma(\theta) \frac{dr}{r^{\alpha+1}}\right)\nonumber\\
&=-\frac{2^{-\alpha k}}{s^{\alpha+1}}\int_{\Sigma}R^{\alpha}(x((s2^{k})\circ\theta)^{-1})\Omega(\theta)d\sigma(\theta)\Big(1-\eta_{0}\Big(\frac{x}{A_{0}^{2}\kappa_{[\alpha]}2^{k+4}}\Big)\Big)\\
&-\frac{2^{-\alpha k}}{s^{\alpha+1}}\int_{\Sigma}R^{\alpha}(x((s2^{k})\circ\theta)^{-1})\Omega(\theta)d\sigma(\theta)\eta_{0}\Big(\frac{x}{A_{0}^{2}\kappa_{[\alpha]}2^{k+4}}\Big)=:I^{1}_{\alpha,k,s}(x)+I^{2}_{\alpha,k,s}(x)\nonumber.
\end{align}
Therefore, it remains to show that there exists a constant $C_{\mathbb{Q},\alpha}>0$ such that for any $j\geq 0$,
\begin{align}\label{firstpart11}
\int_{\rho(x)\geq 2A_{0}\rho(y)}\sum_{k\in\mathbb{Z}}\bigg|\Delta[2^{k-j}]\phi\ast I^{1}_{\alpha,k,s}(y^{-1}x)-\Delta[2^{k-j}]\phi\ast I^{1}_{\alpha,k,s}(x)\bigg|dx\leq C_{\mathbb{Q},\alpha}\|\Omega\|_{L^{1}(\Sigma)},
\end{align}
and that
\begin{align}\label{secondpart22}
\int_{\rho(x)\geq 2A_{0}\rho(y)}\sum_{k\in\mathbb{Z}}\bigg|\Delta[2^{k-j}]\phi\ast I^{2}_{\alpha,k,s}(y^{-1}x)-\Delta[2^{k-j}]\phi\ast I^{2}_{\alpha,k,s}(x)\bigg|dx\leq C_{\mathbb{Q},\alpha}(1+j)\|\Omega\|_{L^{1}(\Sigma)}.
\end{align}

\textbf{Estimate of \eqref{firstpart11}:}

If $\rho(x)\geq A_{0}^{2}\kappa_{[\alpha]}2^{k+4}$, then
$\rho((2^{-\ell}\circ x)z)\geq \frac{1}{A_{0}}\rho(2^{-\ell}\circ x)-\rho(z)\gtrsim 2^{-\ell}\rho(x)$
whenever $\rho(z)\leq 2^{k-\ell+2}\kappa_{[\alpha]}$. By the $[\alpha]$-order cancellation condition of $\Omega$ and Lemma \ref{citetaylor}, for any sufficient large constant $N$,
\begin{align}\label{vuiui100}
|I^{1}_{\alpha,k,s}(x)|
&\lesssim\frac{2^{-\alpha k}}{s^{\alpha+1}}\sum_{\ell\in\mathbb{Z}}2^{-(\mathbb{Q}-\alpha)\ell}\int_{\Sigma}\rho\big((2^{(k-\ell)}s)\circ \theta\big)^{([\alpha]+1)}\sup\limits_{\substack{\rho(z)\leq \kappa_{[\alpha]}\rho((2^{(k-\ell)}s)\circ \theta)\\d(I)=[\alpha]+1}}|(X^{I}\zeta)\big((2^{-\ell}\circ x)z\big)||\Omega(\theta)|d\sigma(\theta)\nonumber\\
&\lesssim\frac{2^{-\alpha k}}{s^{\alpha+1}}\sum_{\ell\in\mathbb{Z}}2^{-(\mathbb{Q}-\alpha)\ell}2^{([\alpha]+1)(k-\ell)}(1+2^{-\ell}\rho(x))^{-N}\|\Omega\|_{L^{1}(\Sigma)}\lesssim \frac{2^{([\alpha]+1-\alpha)k}}{\rho(x)^{\mathbb{Q}+1+[\alpha]-\alpha}}\|\Omega\|_{L^{1}(\Sigma)},
\end{align}
where $\zeta$ is the smooth function defined in \eqref{dezeta}. Since $\supp\phi\subset\{x\in\HH:\rho(x)\leq\frac{1}{100}\}$, we have
\begin{align}\label{eq100}
|\Delta[2^{k-j}]\phi\ast I^{1}_{\alpha,k,s}(x)|
&\lesssim  \|\Omega\|_{L^{1}(\Sigma)}\int_{\rho(y)\leq 2^{k}}\frac{2^{([\alpha]+1-\alpha)k}}{d(x,y)^{\mathbb{Q}+1+[\alpha]-\alpha}}\chi_{d(x,y)>A_{0}^{2}\kappa_{[\alpha]}2^{k+4}}|\Delta[2^{k-j}]\phi(y)|dy\nonumber\\
&\lesssim \|\Omega\|_{L^{1}(\Sigma)}\frac{2^{([\alpha]+1-\alpha)k}}{\rho(x)^{\mathbb{Q}+1+[\alpha]-\alpha}}\chi_{\rho(x)\geq A_{0} \kappa_{[\alpha]} 2^{k+3}}(x).
\end{align}
This also implies that whenever $\rho(x)\geq 2A_{0}\rho(y)$,
\begin{align}\label{eq200}
|\Delta[2^{k-j}]\phi\ast I^{1}_{\alpha,k,s}(y^{-1}x)|
&\lesssim \|\Omega\|_{L^{1}(\Sigma)}\frac{2^{([\alpha]+1-\alpha)k}}{\rho(y^{-1}x)^{\mathbb{Q}+1+[\alpha]-\alpha}}\chi_{\rho(y^{-1}x)\geq A_{0} \kappa_{[\alpha]} 2^{k+3}}(y^{-1}x)\nonumber\\
&\lesssim \|\Omega\|_{L^{1}(\Sigma)}\frac{2^{([\alpha]+1-\alpha)k}}{\rho(x)^{\mathbb{Q}+1+[\alpha]-\alpha}}\chi_{\rho(x)\geq  \kappa_{[\alpha]} 2^{k+2}}(x).
\end{align}
Combining the estimates \eqref{eq100} and \eqref{eq200}, we obtain that
\begin{align}\label{eer100}
\sum_{k\in\mathbb{Z}}\bigg|\Delta[2^{k-j}]\phi\ast I^{1}_{\alpha,k,s}(y^{-1}x)-\Delta[2^{k-j}]\phi\ast I^{1}_{\alpha,k,s}(x)\bigg|&\lesssim \|\Omega\|_{L^{1}(\Sigma)}\sum_{ 2^{k+2}\kappa_{[\alpha]}\leq \rho(x)}\frac{2^{([\alpha]+1-\alpha)k}}{\rho(x)^{\mathbb{Q}+1+[\alpha]-\alpha}}\nonumber\\
&\lesssim\frac{\|\Omega\|_{L^{1}(\Sigma)}}{\rho(x)^{\mathbb{Q}}}.
\end{align}

On the other hand, similar to the proof of \eqref{vuiui100}, for any $i=1,2,\ldots,n$,
\begin{align}\label{vuiui200}
&|X_{i} I^{1}_{\alpha,k,s}(x)|
\lesssim \|\Omega\|_{L^{1}(\Sigma)}\frac{2^{([\alpha]+1-\alpha)k}}{\rho(x)^{\mathbb{Q}+2+[\alpha]-\alpha}}.
\end{align}

Hence, for any $i=1,2,\ldots,n$,
\begin{align}\label{ffrom000}
|X_{i}\Delta[2^{k-j}]\phi\ast I^{1}_{\alpha,k,s}(x)|\lesssim \|\Omega\|_{L^{1}(\Sigma)} \frac{2^{([\alpha]+1-\alpha)k}}{\rho(x)^{\mathbb{Q}+2+[\alpha]-\alpha}}\chi_{\rho(x)\geq A_{0} \kappa_{[\alpha]} 2^{k+3}}(x).
\end{align}
Therefore,
\begin{align}\label{alsk000}
\sum\limits_{k\in\mathbb{Z}}\bigg|X_{i}\Delta[2^{k-j}]\phi\ast I^{1}_{\alpha,k,s}(x)\bigg|&\lesssim  \|\Omega\|_{L^{1}(\Sigma)}\sum\limits_{k\in\mathbb{Z}}\int_{\rho(y)\leq 2^{k}}\frac{2^{([\alpha]+1-\alpha)k}}{d(x,y)^{\mathbb{Q}+2+[\alpha]-\alpha}}\chi_{d(x,y)>A_{0}\kappa_{[\alpha]}2^{k+3}}|\Delta[2^{k-j}]\phi(y)|dy\nonumber\\
&\lesssim \|\Omega\|_{L^{1}(\Sigma)}\sum\limits_{k\in\mathbb{Z}}\frac{2^{([\alpha]+1-\alpha)k}}{\rho(x)^{\mathbb{Q}+2+[\alpha]-\alpha}}\chi_{\rho(x)\geq \kappa_{[\alpha]} 2^{k+2}}(x)\lesssim \|\Omega\|_{L^{1}(\Sigma)}\rho(x)^{-\mathbb{Q}-1}.
\end{align}
This, together with the mean value theorem on homogeneous groups (see \cite{FoSt}), shows
\begin{align}\label{eer2000}
\sum\limits_{k\in\mathbb{Z}}\bigg|\Delta[2^{k-j}]\phi\ast I^{1}_{\alpha,k,s}(y^{-1}x)-\Delta[2^{k-j}]\phi\ast I^{1}_{\alpha,k,s}(x)\bigg|\lesssim \|\Omega\|_{L^{1}(\Sigma)}\frac{\rho(y)}{\rho(x)^{\mathbb{Q}+1}}.
\end{align}
Combining the estimates \eqref{eer100} and \eqref{eer2000}, we conclude that
\begin{align}\label{eeeer000}
\sum\limits_{k\in\mathbb{Z}}\bigg|\Delta[2^{k-j}]\phi\ast I^{1}_{\alpha,k,s}(y^{-1}x)-\Delta[2^{k-j}]\phi\ast I^{1}_{\alpha,k,s}(x)\bigg|\lesssim \frac{\|\Omega\|_{L^{1}(\Sigma)}}{\rho(x)^{\mathbb{Q}}}\omega\Big(\frac{\rho(y)}{\rho(x)}\Big),
\end{align}
where $\omega(t)\leq \min\{1,t\}$ uniformly in $j\geq 0$. Thus,
\begin{align}\label{ererer1000}
&\int_{\rho(x)\geq 2A_{0}\rho(y)}\sum\limits_{k\in\mathbb{Z}}\bigg|\Delta[2^{k-j}]\phi\ast I^{1}_{\alpha,k,s}(y^{-1}x)-\Delta[2^{k-j}]\phi\ast I^{1}_{\alpha,k,s}(x)\bigg|dx\nonumber\\
&\lesssim \|\Omega\|_{L^{1}(\Sigma)}\int_{\rho(x)\geq 2A_{0}\rho(y)}\frac{1}{\rho(x)^{\mathbb{Q}}}\omega\Big(\frac{\rho(y)}{\rho(x)}\Big)dx
\lesssim \|\Omega\|_{L^{1}(\Sigma)}\sum_{k=0}^{\infty}\omega(2^{-k})\lesssim \|\Omega\|_{L^{1}(\Sigma)}.
\end{align}

\textbf{Estimate of \eqref{secondpart22}:}

To begin with,
\begin{align}\label{rhs}
&\sum_{k\in\mathbb{Z}}\int_{\rho(x)\geq 2A_{0}\rho(y)}\bigg|\Delta[2^{k-j}]\phi\ast I^{2}_{\alpha,k,s}(y^{-1}x)-\Delta[2^{k-j}]\phi\ast I^{2}_{\alpha,k,s}(x)\bigg|dx\nonumber\\
&\leq \sum_{k\in\mathbb{Z}}\int_{\HH}\int_{\rho(x)\geq 2A_{0}\rho(y)}|\Delta[2^{k-j}]\phi(y^{-1}xz^{-1})-\Delta[2^{k-j}]\phi(xz^{-1})|dx|I^{2}_{\alpha,k,s}(z)|dz.
\end{align}
It can be verified directly that if $\rho(y)\geq 2^{k+7}A_{0}^{3}\kappa_{[\alpha]}$ and $\rho(x)\geq 2A_{0}\rho(y)$, then
$$\rho(2^{-(k-j)}\circ y^{-1}xz^{-1})\geq 1\ {\rm and}\ \rho(2^{-(k-j)}\circ xz^{-1})\geq 1.$$
This, in combination with a simple change of variable and Lemma \ref{appen}, yields
\begin{align}\label{1234}
&\sum_{k\in\mathbb{Z}}\int_{\rho(x)\geq 2A_{0}\rho(y)}\bigg|\Delta[2^{k-j}]\phi\ast I^{2}_{\alpha,k,s}(y^{-1}x)-\Delta[2^{k-j}]\phi\ast I^{2}_{\alpha,k,s}(x)\bigg|dx\nonumber\\
&\leq \sum_{2^{k+7}A_{0}^{3}\kappa_{[\alpha]}\geq\rho(y)}\int_{\HH}\int_{\rho(x)\geq 2A_{0}\rho(y)}|\Delta[2^{k-j}]\phi(y^{-1}xz^{-1})-\Delta[2^{k-j}]\phi(xz^{-1})|dx|I^{2}_{\alpha,k,s}(z)|dz\nonumber\\
&\leq \sum_{2^{k+7}A_{0}^{3}\kappa_{[\alpha]}\geq\rho(y)}\int_{\HH}\int_{\HH}|\phi\big((2^{-(k-j)}\circ y^{-1}) x (2^{-(k-j)} \circ z^{-1})\big)-\phi\big(x(2^{-(k-j)}\circ z^{-1})\big)|dx|I^{2}_{\alpha,k,s}(z)|dz\nonumber\\
&\lesssim \sum_{2^{k+7}A_{0}^{3}\kappa_{[\alpha]}\geq\rho(y)}\min\Big\{1,\frac{\rho(y)}{2^{k-j}}\Big\}\|I^{2}_{\alpha,k,s}\|_{1}.
\end{align}
To continue, we estimate the $L^{1}$-norm of $I^{2}_{\alpha,k,s}$. Note that
\begin{align*}
\int_{\HH}|I^{2}_{\alpha,k,s}(z)|dz
&\lesssim 2^{-\alpha k}\int_{\rho(z)\leq A_{0}^{2}\kappa_{[\alpha]}2^{k+5}}\int_{\Sigma}\frac{|\Omega(\theta)|}{d(z,(s2^{k})\circ\theta)^{\mathbb{Q}-\alpha}}d\sigma(\theta) dz\\
&= \int_{\Sigma}\int_{\rho(z)\leq A_{0}^{2}\kappa_{[\alpha]}2^{5}}\frac{|\Omega(\theta)|}{d(z,s\circ\theta)^{\mathbb{Q}-\alpha}} dzd\sigma(\theta)\lesssim\|\Omega\|_{L^{1}(\Sigma)}.
\end{align*}
This, in combination with \eqref{1234}, implies
\begin{align}\label{omit2}
&\sum_{k\in\mathbb{Z}}\int_{\rho(x)\geq 2A_{0}\rho(y)}\bigg|\Delta[2^{k-j}]\phi\ast I^{2}_{\alpha,k,s}(y^{-1}x)-\Delta[2^{k-j}]\phi\ast I^{2}_{\alpha,k,s}(x)\bigg|dx\nonumber\\
&\lesssim \|\Omega\|_{L^{1}(\Sigma)} \sum_{2^{k+7}A_{0}^{3}\kappa_{[\alpha]}\geq\rho(y)}\min\Big\{1,\frac{\rho(y)}{2^{k-j}}\Big\}\lesssim (1+j)\|\Omega\|_{L^{1}(\Sigma)}.
\end{align}
This verifies \eqref{secondpart22} and therefore, the proof of Lemma \ref{hormander222} is complete.
\end{proof}
\begin{lemma}\label{strong22}
There exist  constants $C_{\mathbb{Q},\alpha}>0$ and $\tau>0$ such that for any $j\geq 0$,
\begin{align}\label{forsss}
\Big\|\sup\limits_{t\in[1,2)}\sup\limits_{k\in\mathbb{Z}}|V_{k,j,t}^{\alpha}f|\Big\|_{2}\leq C_{\mathbb{Q},\alpha}2^{-\tau j}\|\Omega\|_{L^{1}(\Sigma)}\|f\|_{2}.
\end{align}
\end{lemma}
\begin{proof}

Using fundamental theorem of calculus to the function $|V_{k,j,t}^{\alpha}f(x)|^2$, we get the following type of Gagliardo-Nirenberg inequality:
\begin{align*}
\sup\limits_{t\in[1,2)}\sup\limits_{k\in\mathbb{Z}}|V_{k,j,t}^{\alpha}f(x)|\lesssim \sum_{k\in\mathbb{Z}}\bigg(\int_{1}^{2}|V_{k,j,s}^{\alpha}f(x)|^{2}ds\bigg)^{1/4}\bigg(\int_{1}^{2}\Big|\frac{d}{ds}V_{k,j,s}^{\alpha}f(x)\Big|^{2}ds\bigg)^{1/4}.
\end{align*}
By Cauchy-Schwartz's inequality,
\begin{align*}
\Big\|\sup\limits_{t\in[1,2)}\sup\limits_{k\in\mathbb{Z}}|V_{k,j,t}^{\alpha}f|\Big\|_{2}\leq \bigg\|\bigg(\sum_{k\in\mathbb{Z}}\int_{1}^{2}|V_{k,j,s}^{\alpha}f|^{2}ds\bigg)^{1/2}\bigg\|_{2}^{1/2}
\bigg\|\bigg(\sum_{k\in\mathbb{Z}}\int_{1}^{2}\bigg|\frac{d}{ds}V_{k,j,s}^{\alpha}f\bigg|^{2}ds\bigg)^{1/2}\bigg\|_{2}^{1/2}.
\end{align*}
Thus, Lemma \ref{strong22} can be reduced to showing that
\begin{align}\label{red1}
\bigg\|\bigg(\sum_{k\in\mathbb{Z}}\int_{1}^{2}|H_{k,j,s}^{\alpha}f|^{2}ds\bigg)^{1/2}\bigg\|_{2}\lesssim 2^{-\tau j}\|\Omega\|_{L^{1}(\Sigma)}\|f\|_{2}
\end{align}
for some $\tau>0$, where $H_{k,j,s}^{\alpha}$ stands for the operator  $V_{k,j,s}^{\alpha}$ or $\frac{d}{ds}V_{k,j,s}^{\alpha}$.

\textbf{Estimate of \eqref{red1}:}

To this end, let $r_{k}:[0,1]\rightarrow \mathbb{R}$ be a collection of independent Rademacher random variables. From the calculation of \eqref{recalll} we see that
\begin{align*}
\frac{d}{ds}V_{k,j,s}^{\alpha}f=f\ast\Psi_{k-j}\ast R^{\alpha}\ast (-2^{-(\mathbb{Q}-1+\alpha)k}s^{-\mathbb{Q}-\alpha}\Omega d\sigma_{s2^{k}}).
\end{align*}
Then by Khinchin's inequality and Lemma \ref{ckt},
\begin{align*}
\bigg\|\bigg(\sum_{k\in\mathbb{Z}}\int_{1}^{2}|H_{k,j,s}^{\alpha}f|^{2}ds\bigg)^{1/2}\bigg\|_{2}
&\lesssim\sup\limits_{s\in[1,2]}\bigg(\int_{\HH}\bigg\|\sum_{k\in\mathbb{Z}}r_{k}(t)H_{k,j,s}^{\alpha}f\bigg\|_{L^{2}([0,1])}^{2}dx\bigg)^{1/2}\\
&\lesssim\sup\limits_{s\in[1,2]}\sup\limits_{t\in[0,1]}\bigg\|\sum_{k\in\mathbb{Z}}r_{k}(t)H_{k,j,s}^{\alpha}f\bigg\|_{2}\lesssim2^{-\tau j}\|\Omega\|_{L^{1}(\Sigma)}\|f\|_{2}.
\end{align*}


%
This ends the proof of \eqref{red1} and then Lemma \ref{strong22}.
\end{proof}

\begin{lemma}\label{weak22}
For any $1<p<\infty$, there exist  constants $C_{\mathbb{Q},\alpha}$, $C_{\mathbb{Q},\alpha,p}>0$ and $\tau>0$ such that for any $j\geq 0$,
\begin{align}\label{weaktype2222}
\Big\|\sup\limits_{t\in[1,2)}\sup\limits_{k\in\mathbb{Z}}|V_{k,j,t}^{\alpha}f|\Big\|_{L^{1,\infty}}\leq C_{\mathbb{Q},\alpha}(1+j)\|\Omega\|_{L^{1}(\Sigma)}\|f\|_{1},
\end{align}
and
\begin{align}\label{strongpppp}
\Big\|\sup\limits_{t\in[1,2)}\sup\limits_{k\in\mathbb{Z}}|V_{k,j,t}^{\alpha}f|\Big\|_{L^{p}}\leq C_{\mathbb{Q},\alpha,p}2^{-\tau j}\|\Omega\|_{L^{1}(\Sigma)}\|f\|_{p}.
\end{align}
\end{lemma}
\begin{proof}
By Lemmas \ref{hormander222} and \ref{strong22}, these inequalities can be obtained by a standard argument of Calder\'{o}n-Zygmund decomposition (see for example \cite{JDbook,s93}), together with the interpolation theorem and a standard vector-valued duality argument (see for example \cite[Theorem 5.17]{JDbook}).
\end{proof}

{\it Proof of Proposition \ref{nondyadicprop}}.

Note that $M_{\Omega,\alpha}(-\Delta_{\HH})^{-\alpha/2}f=\sup\limits_{t\in[1,2)}\sup\limits_{k\in\mathbb{Z}}|f\ast R^{\alpha}\ast B_{k,t}^{\alpha}\Omega|$. This, in combination with the equality \eqref{vkj} and inequality \eqref{strongpppp}, yields
\begin{align*}
\|M_{\Omega,\alpha}(-\Delta_{\HH})^{-\alpha/2}f\|_{p}\leq \sum\limits_{j=0}^{\infty}\Big\|\sup\limits_{t\in[1,2)}\sup\limits_{k\in\mathbb{Z}}|V_{k,j,t}^{\alpha}f|\Big\|_{p}\lesssim \|\Omega\|_{L^{1}(\Sigma)}\|f\|_{p}.
\end{align*}
This ends the proof of Proposition \ref{nondyadicprop}.
\hfill $\square$
\subsection{$L^{p}$ estimate of dyadic maximal function}\label{dyadicsection}
In this subsection, we give the proof of the Proposition \ref{dyadicprop}. To this end, let $\phi$ be a smooth cut-off function defined in Section \ref{preliminariessec}. Then
\begin{align*}
T_{\Omega,\alpha}^{k}f= T_{\Omega,\alpha}f\ast\Delta[2^{k}]\phi-\sum_{s=-\infty}^{0}f\ast B_{k+s}^{\alpha}\Omega\ast\Delta[2^{k}]\phi+\sum_{s=1}^{\infty}f\ast B_{k+s}^{\alpha}\Omega\ast(\delta_{0}-\Delta[2^{k}]\phi).
\end{align*}
From this equality we see that
\begin{align}\label{IIIIII}
&\Big\|\sup\limits_{k\in\mathbb{Z}}|T_{\Omega,\alpha}^{k}(-\Delta_{\HH})^{-\alpha/2}f|\Big\|_{p}\nonumber\\
&\leq \Big\|\sup\limits_{k\in\mathbb{Z}}|T_{\Omega,\alpha}(-\Delta_{\HH})^{-\alpha/2}f\ast\Delta[2^{k}]\phi|\Big\|_{p}
+\Big\|\sup\limits_{k\in\mathbb{Z}}\Big|\sum_{s=-\infty}^{0}f\ast R^{\alpha}\ast B_{k+s}^{\alpha}\Omega\ast\Delta[2^{k}]\phi\Big|\Big\|_{p}\nonumber\\
&+\Big\|\sup\limits_{k\in\mathbb{Z}}\Big|\sum_{s=1}^{\infty}f\ast R^{\alpha}\ast B_{k+s}^{\alpha}\Omega\ast(\delta_{0}-\Delta[2^{k}]\phi)\Big|\Big\|_{p}=:{\rm I_{1}+I_{2}+I_{3}}.
\end{align}

\textbf{Estimate of ${\rm I_{1}}$.} By Proposition \ref{prop1} and the $L^{p}$ boundedness of the Hardy-Littlewood maximal function, we see that for any $1<p<\infty$,
\begin{align*}
\Big\|\sup\limits_{k\in\mathbb{Z}}|T_{\Omega,\alpha}(-\Delta_{\HH})^{-\alpha/2}f\ast\Delta[2^{k}]\phi|\Big\|_{p}
\lesssim\|MT_{\Omega,\alpha}(-\Delta_{\HH})^{-\alpha/2}f\|_{p}\lesssim\|T_{\Omega,\alpha}(-\Delta_{\HH})^{-\alpha/2}f\|_{p}\lesssim\|\Omega\|_{L^{1}(\Sigma)}\|f\|_{p}.
\end{align*}

\textbf{Estimate of ${\rm I_{2}}$.} Observe that
\begin{align*}
\sum_{s=-\infty}^{0}B_{k+s}^{\alpha}\Omega=\frac{\Omega(x)}{\rho(x)^{\mathbb{Q}+\alpha}}\chi_{\rho(x)\leq 2^{k+1}}.
\end{align*}

Then using the cancellation condition of $\Omega$ and Lemma \ref{citetaylor} to get that
\begin{align}\label{fjfjfjfj}
\bigg|\sum_{s=-\infty}^{0}B_{k+s}^{\alpha}\Omega\ast\Delta[2^{k}]\phi (x)\bigg|
&\lesssim 2^{-(\mathbb{Q}+[\alpha]+1)k}\int_{\rho(y)\leq 2^{k+2}}\frac{|\Omega(y)|}{\rho(y)^{\mathbb{Q}+\alpha-[\alpha]-1}}dy\chi_{\rho(x)\leq A_{0}\kappa_{[\alpha]}2^{k+2}}\nonumber\\
&\lesssim \frac{\|\Omega\|_{L^{1}(\Sigma)}}{2^{(\mathbb{Q}+\alpha)k}}\chi_{\rho(x)\leq A_{0}\kappa_{[\alpha]}2^{k+2}},
\end{align}
where  the last inequality follows by using polar coordinates \eqref{polar}.

\textbf{Case 1:} If $\rho(x)\leq A_{0}^{2}\kappa_{[\alpha]}2^{k+4}$, then
\begin{align}\label{cmn145}
\bigg|R^{\alpha}\ast \sum_{s=-\infty}^{0}B_{k+s}^{\alpha}\Omega\ast\Delta[2^{k}]\phi (x)\bigg|&=\left|\int_{\HH}R^{\alpha}(xy^{-1}) \sum_{s=-\infty}^{0}B_{k+s}^{\alpha}\Omega\ast\Delta[2^{k}]\phi(y)dy\right|\nonumber\\
&\lesssim\int_{d(x,y)\leq A_{0}^{3}\kappa_{[\alpha]}2^{k+5}}\frac{\|\Omega\|_{L^{1}(\Sigma)}}{2^{(\mathbb{Q}+\alpha)k}}\frac{1}{d(x,y)^{\mathbb{Q}-\alpha}}dy\lesssim 2^{-\mathbb{Q}k}\|\Omega\|_{L^{1}(\Sigma)}.
\end{align}

\textbf{Case 2:} If $\rho(x)\geq A_{0}^{2}\kappa_{[\alpha]}2^{k+4}$, then
$\rho((2^{-\ell}\circ x)z)\geq \frac{1}{A_{0}}\rho(2^{-\ell}\circ x)-\rho(z)\gtrsim 2^{-\ell}\rho(x)$
whenever $\rho(z)\leq 2^{k-\ell+2}\kappa_{[\alpha]}$. Similar to the proof of \eqref{vuiui100}, we use the $[\alpha]$-order cancellation condition of $\Omega$ and Lemma \ref{citetaylor} to get that for any sufficient large constant $N$,
\begin{align}\label{cmn245}
\bigg|R^{\alpha}\ast \sum_{s=-\infty}^{0}B_{k+s}^{\alpha}\Omega\ast\Delta[2^{k}]\phi (x)\bigg|
&\lesssim \sum_{\ell\in\mathbb{Z}}2^{-(\mathbb{Q}-\alpha)\ell}2^{([\alpha]+1)(k-\ell)}(1+2^{-\ell}\rho(x))^{-N}\Big\|\sum_{s=-\infty}^{0}B_{k+s}^{\alpha}\Omega\ast\Delta[2^{k}]\phi\Big\|_{1}\nonumber\\
&\lesssim
\|\Omega\|_{L^{1}(\Sigma)}\frac{2^{([\alpha]+1-\alpha)k}}{\rho(x)^{\mathbb{Q}+1+[\alpha]-\alpha}}.
\end{align}
Combining the estimates \eqref{cmn145} and \eqref{cmn245} together, we conclude that
\begin{align}\label{estimateI222}
\sup\limits_{k\in\mathbb{Z}}\bigg|\sum_{s=-\infty}^{0}f\ast R^{\alpha}\ast B_{k+s}^{\alpha}\Omega\ast\Delta[2^{k}]\phi(x)\bigg|\lesssim \|\Omega\|_{L^{1}(\Sigma)}Mf(x),
\end{align}
where $Mf$ is the Hardy-Littlewood maximal function.

This, together with the $L^{p}$ boundedness of the Hardy-Littlewood maximal function, yields
\begin{align*}
{\rm I_{2}}\lesssim \|\Omega\|_{L^{1}(\Sigma)}\|f\|_{p}.
\end{align*}

\textbf{Estimate of ${\rm I_{3}}$.}

To begin with, we write
\begin{align*}
\sum_{s=1}^{\infty}f\ast R^{\alpha}\ast B_{k+s}^{\alpha}\Omega\ast (\delta_{0}-\Delta[2^{k}]\phi)=:\sum_{s=1}^{\infty}G_{k,s}^{\alpha}f.
\end{align*}
Noting that $S_{j}f\rightarrow f$ as $j\rightarrow -\infty$, we have the following identity:
\begin{align}\label{gksj}
G_{k,s}^{\alpha}=G_{k,s}^{\alpha}S_{k+s}+\sum_{j=1}^{\infty}G_{k,s}^{\alpha}(S_{k+s-j}-S_{k+s-(j-1)})=:\sum_{j=0}^{\infty}G_{k,s,j}^{\alpha},
\end{align}
where $G_{k,s,0}^{\alpha}:=G_{k,s}^{\alpha}S_{k+s}$ and $G_{k,s,j}^{\alpha}:=G_{k,s}^{\alpha}(S_{k+s-j}-S_{k+s-(j-1)})$ for $j\geq 1$. Then
\begin{align*}
{\rm I_{3}}=\bigg\|\sup\limits_{k\in\mathbb{Z}}\bigg|\sum_{s=1}^{\infty}f\ast R^{\alpha}\ast B_{k+s}^{\alpha}\Omega\ast (\delta_{0}-\Delta[2^{k}]\phi)\bigg|\bigg\|_{p}\leq\sum_{j=0}^{\infty}\sum_{s=1}^{\infty}\Big\|\sup\limits_{k\in\mathbb{Z}}|G_{k,s,j}^{\alpha}f|\Big\|_{p}.
\end{align*}
To continue, we establish the following inequalities: for any $j\geq 0$,
\begin{align}\label{ggl2}
\Big\|\sup\limits_{k\in\mathbb{Z}}|G_{k,s,j}^{\alpha}f|\Big\|_{2}\lesssim 2^{-\tau(j+s)}\|\Omega\|_{L^{1}(\Sigma)}\|f\|_{2}
\end{align}
for some $\tau>0$,
and
\begin{align}\label{gglp}
\Big\|\sup\limits_{k\in\mathbb{Z}}|G_{k,s,j}^{\alpha}f|\Big\|_{L^{1,\infty}}\lesssim (1+j+s)\|\Omega\|_{L^{1}(\Sigma)}\|f\|_{1}.
\end{align}
Applying the Marcinkiewicz interpolation theorem between \eqref{ggl2} and \eqref{gglp} and a standard vector-valued duality argument (see for example \cite[Theorem 5.17]{JDbook}), we see that for any $1<p<\infty$,
\begin{align*}
\Big\|\sup\limits_{k\in\mathbb{Z}}|G_{k,s,j}^{\alpha}f|\Big\|_{p}\lesssim 2^{-\tau(j+s)}\|\Omega\|_{L^{1}(\Sigma)}\|f\|_{p}.
\end{align*}
Therefore,
\begin{align*}
{\rm I_{3}}\leq\sum_{j=0}^{\infty}\sum_{s=1}^{\infty}2^{-\tau(j+s)}\|\Omega\|_{L^{1}(\Sigma)}\|f\|_{p}\lesssim\|\Omega\|_{L^{1}(\Sigma)}\|f\|_{p}.
\end{align*}

We now return to verify the inequalities \eqref{ggl2} and \eqref{gglp}.

For the estimate \eqref{ggl2}, it follows from Khinchin's inequality that for any $s\geq1$ and $i\in\mathbb{Z}$,
\begin{align*}
&\bigg\|\sup\limits_{k\in\mathbb{Z}}|f\ast\Psi_{k+s-i}\ast R^{\alpha}\ast B_{k+s}^{\alpha}\Omega\ast (\delta_{0}-\Delta[2^{k}]\phi)|\bigg\|_{2}\\
&\leq \bigg\|\bigg(\sum_{k\in\mathbb{Z}}|f\ast\Psi_{k+s-i}\ast R^{\alpha}\ast B_{k+s}^{\alpha}\Omega\ast (\delta_{0}-\Delta[2^{k}]\phi)|^{2}\bigg)^{1/2}\bigg\|_{2}\\
&\lesssim \bigg\|\Big\|\sum_{k\in\mathbb{Z}}r_{k}(t)f\ast\Psi_{k+s-i}\ast R^{\alpha}\ast B_{k+s}^{\alpha}\Omega\ast (\delta_{0}-\Delta[2^{k}]\phi)\Big\|_{L^{2}([0,1])}\bigg\|_{2}\\
&\lesssim \sum_{\ell=0}^{\infty}\sup\limits_{t\in[0,1]}\bigg\|\sum_{k\in\mathbb{Z}}r_{k}(t)f\ast\Psi_{k+s-i}\ast R^{\alpha}\ast B_{k+s}^{\alpha}\Omega\ast \Psi_{k-\ell}\bigg\|_{2}=:\sum_{\ell=0}^{\infty}\sup\limits_{t\in[0,1]}\bigg\|\sum_{k\in\mathbb{Z}}G_{k,s,i,\ell}^{\alpha}(t)f\bigg\|_{2}.
\end{align*}
By Cotlar-Knapp-Stein Lemma (see \cite{s93}), it suffices to show that:
\begin{align}\label{hhhhou}
&\|(G_{k,s,i,\ell}^{\alpha}(t))^{*}G_{k^{\prime},s,i,\ell}^{\alpha}(t)\|_{2\rightarrow 2}+\|G_{k^{\prime},s,i,\ell}^{\alpha}(t)(G_{k,s,i,\ell}^{\alpha}(t))^{*}\|_{2\rightarrow 2}\lesssim 2^{-2\tau(\ell+s+|i|)}2^{-\tau|k-k^{\prime}|}\|\Omega\|_{L^{1}(\Sigma)}^{2}.
\end{align}
We only estimate the first term, since the second one can be estimated similarly.

On the one hand, by Proposition \ref{keypro} with $\mu_{j}^{\alpha}$ chosen to be $B_{j}^{\alpha}\Omega$, Young's inequality and the estimate \eqref{cacan},
\begin{align}\label{zsdfxc}
\|(G_{k,s,i,\ell}^{\alpha}(t))^{*}G_{k^{\prime},s,i,\ell}^{\alpha}(t)f\|_{2}
&=\|(f\ast\Psi_{k^{\prime}+s-i}\ast R^{\alpha}\ast B_{k^{\prime}+s}^{\alpha}\Omega)\ast (\Psi_{k^{\prime}-\ell}\ast\Psi_{k-\ell})\ast (B_{k+s}^{\alpha}\tilde{\Omega}\ast R^{\alpha}\ast \Psi_{k+s-i}) \|_{2}\nonumber\\
&\lesssim 2^{-2\tau|i|}2^{-|k-k^{\prime}|}\|\Omega\|_{L^{1}(\Sigma)}^{2}\|f\|_{2}.
\end{align}

On the other hand, we also have
\begin{align}\label{zxsdfacv}
\|(G_{k,s,i,\ell}^{\alpha}(t))^{*}G_{k^{\prime},s,i,\ell}^{\alpha}(t)f\|_{2}
&=\|(f\ast\Psi_{k^{\prime}+s-i})\ast (R^{\alpha}\ast B_{k^{\prime}+s}^{\alpha}\Omega\ast \Psi_{k^{\prime}-\ell})\ast(\Psi_{k-\ell}\ast B_{k+s}^{\alpha}\tilde{\Omega}\ast R^{\alpha})\ast \Psi_{k+s-i}\|_{2}\nonumber\\
&\lesssim 2^{-2\tau(s+\ell)}\|\Omega\|_{L^{1}(\Sigma)}^{2}\|f\|_{2}.
\end{align}

Taking geometric means of \eqref{zsdfxc} and \eqref{zxsdfacv}, we obtain \eqref{hhhhou} and therefore,
\begin{align}\label{makey222}
\bigg\|\sup\limits_{k\in\mathbb{Z}}|f\ast\Psi_{k+s-i}\ast R^{\alpha}\ast B_{k+s}^{\alpha}\Omega\ast (\delta_{0}-\Delta[2^{k}]\phi)|\bigg\|_{2}
&\lesssim\sum_{\ell=0}^{\infty}\sup\limits_{t\in[0,1]}\Big\|\sum_{k\in\mathbb{Z}}G_{k,s,i,\ell}^{\alpha}(t)f\Big\|_{2}\nonumber\\
&\lesssim\sum_{\ell=0}^{\infty} 2^{-\tau(\ell+s+|i|)}\|\Omega\|_{L^{1}(\Sigma)}\|f\|_{2}\nonumber\\
&\lesssim2^{-\tau(s+|i|)}\|\Omega\|_{L^{1}(\Sigma)}\|f\|_{2}.
\end{align}
By replacing $i$ with $j$, it is direct that \eqref{makey222} implies  \eqref{ggl2} for $j\geq 1$. While for $j=0$, the left-hand side of  \eqref{ggl2} is dominated by the summation of the left-hand side \eqref{makey222} over $i$ from $-\infty$ to $0$, which gives the required estimate.

Next we verify \eqref{gglp}, which, by a standard argument of Calder\'{o}n-Zygmund decomposition (see for example \cite{JDbook,s93}), can be reduced to showing the following uniform H\"{o}rmander condition: there exists a constant $C_{\mathbb{Q},\alpha}>0$ such that for any $j\geq 0$ and $s\geq1$,
\begin{align}\label{frone}
\int_{\rho(x)\geq 2A_{0}\rho(y)}\sup\limits_{k\in\mathbb{Z}}|K_{k+s,j,1}^{\alpha}(y^{-1}x)-K_{k+s,j,1}^{\alpha}(x)|dx\leq C_{\mathbb{Q},\alpha}(1+j)\|\Omega\|_{L^{1}(\Sigma)},
\end{align}
and
\begin{align}\label{seone}
\int_{\rho(x)\geq 2A_{0}\rho(y)}\sup\limits_{k\in\mathbb{Z}}|K_{k+s,j,1}^{\alpha}\ast\Delta[2^{k}]\phi(y^{-1}x)-K_{k+s,j,1}^{\alpha}\ast\Delta[2^{k}]\phi(x)|dx\leq C_{\mathbb{Q},\alpha}(1+j)\|\Omega\|_{L^{1}(\Sigma)},
\end{align}
where $K_{k+s,j,1}^{\alpha}$ is defined in \eqref{denotee}. Note that the estimate \eqref{frone} is a direct consequence of Lemma \ref{hormander222}. Next, we borrow the argument in the proof of Lemma \ref{hormander} to show \eqref{seone}. To this end, we first decompose $R^{\alpha}\ast B_{k+s}^{\alpha}\Omega\ast \Delta[2^{k}]\phi$ as follows.
\begin{align*}
& R^{\alpha}\ast B_{k+s}^{\alpha}\Omega\ast\Delta[2^{k}]\phi\\
&=R^{\alpha}\ast B_{k+s}^{\alpha}\Omega\ast\Delta[2^{k}]\phi\Big(1-\eta_{0}\Big(\frac{x}{A_{0}^{3}\kappa_{[\alpha]}2^{k+s+5}}\Big)\Big)+ R^{\alpha}\ast B_{k+s}^{\alpha}\Omega\ast\Delta[2^{k}]\phi\eta_{0}\Big(\frac{x}{A_{0}^{3}\kappa_{[\alpha]}2^{k+s+5}}\Big)\\
&=:J^{1}_{\alpha,k,s}+J^{2}_{\alpha,k,s}.
\end{align*}

Note that if $\rho(x)\geq A_{0}^{3}\kappa_{[\alpha]}2^{k+s+5}$ and $\rho(y)\leq 2^{k}$, then $\rho(xy^{-1})\geq A_{0}^{2}\kappa_{[\alpha]}2^{k+s+4}$.
Hence, following the proof of inequalities \eqref{ererer1000} and \eqref{omit2}, we deduce that
there exists a constant $C_{\mathbb{Q},\alpha}>0$ such that for any $j\geq 0$ and $s\geq1$,
\begin{align}\label{firstpart33}
\int_{\rho(x)\geq 2A_{0}\rho(y)}\sum_{k\in\mathbb{Z}}\bigg|\Delta[2^{k+s-j}]\phi\ast J^{1}_{\alpha,k,s}(y^{-1}x)-\Delta[2^{k+s-j}]\phi\ast J^{1}_{\alpha,k,s}(x)\bigg|dx\leq C_{\mathbb{Q},\alpha}\|\Omega\|_{L^{1}(\Sigma)},
\end{align}
and that
\begin{align}\label{secondpart4444}
\int_{\rho(x)\geq 2A_{0}\rho(y)}\sum_{k\in\mathbb{Z}}\bigg|\Delta[2^{k+s-j}]\phi\ast J^{2}_{\alpha,k,s}(y^{-1}x)-\Delta[2^{k+s-j}]\phi\ast J^{2}_{\alpha,k,s}(x)\bigg|dx\leq C_{\mathbb{Q},\alpha}(1+j)\|\Omega\|_{L^{1}(\Sigma)}.
\end{align}
Combining the inequalities \eqref{firstpart33} and \eqref{secondpart4444}, we obtain \eqref{seone}.

Finally, by combining the estimates of ${\rm I_{1}}$, ${\rm I_{2}}$ and ${\rm I_{3}}$, the proof of Theorem \ref{main1} is complete.
\hfill $\square$
\bigskip
\section{weighted $L^{p}$ estimate}\label{weightsection}
\setcounter{equation}{0}

Throughout this section, unless we mention the contrary, we suppose that $\Omega$ satisfies the following assumption.

\textbf{Assumption: } Let $0\leq\alpha<\mathbb{Q}$. Suppose that $\Omega\in L^{q}(\Sigma)$  for some $q>\mathbb{Q}/\alpha$, and satisfies the cancellation condition of order $[\alpha]$.
\subsection{Kernel truncation and frequency localization}\label{kf}
To begin with, we borrow the ideas in \cite{HRT} to modify the decomposition in Section \ref{sssss2}. Since $S_{j}f\rightarrow f$ as $j\rightarrow -\infty$, for any sequence of integers $\{N(j)\}_{j=0}^{\infty}$ with $0=N(0)<N(1)<\cdots<N(j)\rightarrow +\infty$, we have
\begin{align*}
T_{k}^{\alpha}(-\Delta_{\HH})^{-\alpha/2}=T_{k}^{\alpha}(-\Delta_{\HH})^{-\alpha/2}S_{k}+\sum_{j=1}^{\infty}T_{k}^{\alpha}(-\Delta_{\HH})^{-\alpha/2}(S_{k-N(j)}-S_{k-N(j-1)}).
\end{align*}
In this way, $T_{\Omega,\alpha}(-\Delta_{\HH})^{-\alpha/2}=\sum_{j=0}^{\infty}\tilde{T}_{j}^{\alpha,N}$, where $\tilde{T}_{0}^{\alpha,N}:=\sum_{k\in\mathbb{Z}}T_{k}^{\alpha}(-\Delta_{\HH})^{-\alpha/2}S_{k}$,
and for $j\geq 1$,
\begin{align}\label{b222}
\tilde{T}_{j}^{\alpha,N}:=\sum_{k\in\mathbb{Z}}T_{k}^{\alpha}(-\Delta_{\HH})^{-\alpha/2}(S_{k-N(j)}-S_{k-N(j-1)}).
\end{align}
\subsection{Calder\'{o}n--Zygmund theory of $\tilde{T}_{j}^{\alpha,N}$}
Let $K_{j}^{\alpha,N}$ be the kernel of $\tilde{T}_{j}^{\alpha,N}$, then we have the following lemma.
\begin{lemma}\label{cal}
The operator $\tilde{T}_{j}^{\alpha,N}$ is a Calder\'{o}n--Zygmund operator satisfying: there exists a constant $C_{\mathbb{Q},\alpha,q}>0$ such that for any $j\geq 0$,
\begin{align}\label{sizec}
|K_{j}^{\alpha,N}(x)|\leq C_{\mathbb{Q},\alpha,q}\frac{\|\Omega\|_{L^{q}(\Sigma)}}{\rho(x)^{\mathbb{Q}}},
\end{align}
and if $\rho(x)\geq 2A_{0}\rho(y)$,
\begin{align}\label{jkl}
|K_{j}^{\alpha,N}(y^{-1}x)-K_{j}^{\alpha,N}(x)|\leq C_{\mathbb{Q},\alpha,q} \frac{\omega_{j}(\rho(y)/\rho(x))}{\rho(x)^{\mathbb{Q}}},
\end{align}
where $
\omega_{j}(t)\leq \|\Omega\|_{L^{q}(\Sigma)}\min\{1,2^{N(j)}t\}$ and $\|\omega_{j}\|_{{\rm Dini}}\lesssim (1+N(j))\|\Omega\|_{L^q(\Sigma)}$.
\end{lemma}
\begin{proof}
We only give the proof of the case $0<\alpha<\mathbb{Q}$ and $j\geq 1$ since the case $\alpha=0$ and the case $j=0$ are much more simpler. The $L^{2}$ boundedness of $\tilde{T}_{j}^{\alpha,N}$ follows from Lemma \ref{ckt} and the equality \eqref{equa} directly. Next we estimate the expression $R^{\alpha}\ast A_{k}^{\alpha}K_{\alpha}^{0}$.

\textbf{Case 1:} If $\rho(x)\leq A_{0}^{3}\kappa_{[\alpha]}2^{k+5}$,
then we first note that
\begin{align}\label{indd}
|A_{k}^{\alpha}K_\alpha^0(x)|
&=2^{-k}\left|\int_{-\infty}^{+\infty}\varphi(2^{-k}t)t^{-\mathbb{Q}-\alpha}K_{\alpha}(t^{-1}\circ x)\chi_{1\leq \rho(t^{-1}\circ x)\leq 2}(x)dt\right|\nonumber\\
&=\left|\int_{-\infty}^{+\infty}t\varphi(t)(2^{k}t)^{-\mathbb{Q}-\alpha}K_{\alpha}\big((2^{k}t)^{-1}\circ x\big)\chi_{2^{k}t\leq\rho(x)\leq 2^{k+1}t}(x)dt\right|\nonumber\\
&\lesssim \rho(x)^{-\mathbb{Q}-\alpha}|\Omega(x)|\chi_{2^{k-1}\leq\rho(x)\leq 2^{k+2}}(x).
\end{align}
This, in combination with \eqref{Rie}, yields that for $q>\mathbb{Q}/\alpha$,
\begin{align}\label{caaa1}
|R^{\alpha}\ast A_{k}^{\alpha}K_{\alpha}^{0}(x)|
&\lesssim\int_{2^{k-1}\leq \rho(z)\leq 2^{k+2}}\frac{|\Omega(z)|}{\rho(z)^{\mathbb{Q}+\alpha}}\frac{1}{d(x,z)^{\mathbb{Q}-\alpha}}dz\nonumber\\
&\lesssim\left(\int_{2^{k-1}\leq \rho(z)\leq 2^{k+2}}\frac{|\Omega(z)|^{q}}{\rho(z)^{(\mathbb{Q}+\alpha)q}}dz\right)^{1/q}\left(\int_{d(x,z)\lesssim 2^{k}}\frac{1}{d(x,z)^{(\mathbb{Q}-\alpha)q^{\prime}}}dz\right)^{1/q^{\prime}}\lesssim 2^{-\mathbb{Q}k}\|\Omega\|_{L^{q}(\Sigma)}.
\end{align}

\textbf{Case 2:} If $\rho(x)\geq A_{0}^{3}\kappa_{[\alpha]}2^{k+5}$, then  using the decomposition \eqref{rieszde} and Lemma \ref{citetaylor}, we get
\begin{align}\label{caaa2}
|R^{\alpha}\ast A_{k}^{\alpha}K_{\alpha}^{0}(x)|
\lesssim \frac{2^{([\alpha]+1-\alpha)k}}{\rho(x)^{\mathbb{Q}+1+[\alpha]-\alpha}}\|\Omega\|_{L^{1}(\Sigma)}.
\end{align}
Next, we combine the estimates \eqref{caaa1} and \eqref{caaa2} to estimate $\Delta[2^{k-N(j)}]\phi\ast R^{\alpha}\ast A_{k}^{\alpha}K_{\alpha}^{0}$.
Since $\supp\phi\subset\{x\in\HH:\rho(x)\leq\frac{1}{100}\}$, we have
\begin{align}\label{hjhj}
|\Delta[2^{k-N(j)}]\phi\ast R^{\alpha}\ast A_{k}^{\alpha}K_{\alpha}^{0}(x)|
&\lesssim \|\Omega\|_{L^{q}(\Sigma)}\int_{\rho(y)\leq 2^{k}}2^{-\mathbb{Q}k}\chi_{d(x,y)\leq A_{0}^{3}\kappa_{[\alpha]} 2^{k+5}}|\Delta[2^{k-N(j)}]\phi(y)|dy\nonumber\\
&\quad+ \|\Omega\|_{L^{1}(\Sigma)}\int_{\rho(y)\leq 2^{k}}\frac{2^{([\alpha]+1-\alpha)k}}{d(x,y)^{\mathbb{Q}+1+[\alpha]-\alpha}}\chi_{d(x,y)>A_{0}^{3}\kappa_{[\alpha]}2^{k+5}}|\Delta[2^{k-N(j)}]\phi(y)|dy\nonumber\\
&\lesssim \|\Omega\|_{L^{q}(\Sigma)}2^{-\mathbb{Q}k}\chi_{\rho(x)\leq A_{0}^{4}\kappa_{[\alpha]} 2^{k+6}}(x)+\|\Omega\|_{L^{1}(\Sigma)}\frac{2^{([\alpha]+1-\alpha)k}}{\rho(x)^{\mathbb{Q}+1+[\alpha]-\alpha}}\chi_{\rho(x)\geq A_{0}^{2} \kappa_{[\alpha]} 2^{k+4}}(x).
\end{align}
This, together with triangle's inequality, implies
\begin{align}\label{kkkkkl}
|K_{j}^{\alpha,N}(x)|
&\lesssim \sum_{k\in\mathbb{Z}}\left(\|\Omega\|_{L^{q}(\Sigma)}2^{-\mathbb{Q}k}\chi_{\rho(x)\leq A_{0}^{4}\kappa_{[\alpha]} 2^{k+6}}(x)+\|\Omega\|_{L^{1}(\Sigma)}\frac{2^{([\alpha]+1-\alpha)k}}{\rho(x)^{\mathbb{Q}+1+[\alpha]-\alpha}}\chi_{\rho(x)\geq A_{0}^{2} \kappa_{[\alpha]}2^{k+4}}(x)\right)\nonumber\\
&\lesssim \frac{\|\Omega\|_{L^{q}(\Sigma)}}{\rho(x)^{\mathbb{Q}}}.
\end{align}

To estimate \eqref{jkl}, we first estimate $X_{i} \Delta[2^{k-N(j)}]\phi\ast R^{\alpha}\ast A_{k}^{\alpha}K_{\alpha}^{0}(x)$ for any $i=1,2,\cdots,n$. We also consider it into two cases.

\textbf{Case 1:} If $\rho(x)\leq A_{0}^{3}\kappa_{[\alpha]}2^{k+5}$, by the estimates \eqref{caaa1} and \eqref{caaa2} and the fact $\supp X_{i}\phi\subset\{x\in\HH:\rho(x)\leq \frac{1}{100}\}$, we see that
\begin{align}\label{baaa1}
&|X_{i}\Delta[2^{k-N(j)}]\phi\ast R^{\alpha}\ast A_{k}^{\alpha}K_{\alpha}^{0}(x)|\nonumber\\
&\lesssim\|\Omega\|_{L^{q}(\Sigma)}\int_{d(x,y)<2^{k-N(j)}}\bigg(2^{-\mathbb{Q}k}\chi_{\rho(y)\leq A_{0}^{3}\kappa_{[\alpha]} 2^{k+5}}+\frac{2^{([\alpha]+1-\alpha)k}}{\rho(y)^{\mathbb{Q}+1+[\alpha]-\alpha}}\chi_{\rho(y)\geq A_{0}^{3}\kappa_{[\alpha]}2^{k+5}}\bigg)2^{-(\mathbb{Q}+1)(k-N(j))}dy\nonumber\\
&\lesssim\|\Omega\|_{L^{q}(\Sigma)}2^{N(j)}2^{-(\mathbb{Q}+1)k}.
\end{align}

\textbf{Case 2:} If $\rho(x)\geq A_{0}^{3}\kappa_{[\alpha]}2^{k+5}$, we claim that
\begin{align*}
|X_{i} \Delta[2^{k-N(j)}]\phi\ast R^{\alpha}\ast A_{k}^{\alpha}K_{\alpha}^{0}(x)|
\lesssim \frac{2^{([\alpha]+1-\alpha)k}}{\rho(x)^{\mathbb{Q}+2+[\alpha]-\alpha}}\|\Omega\|_{L^{1}(\Sigma)}.
\end{align*}
To show this inequality, we use the decomposition \eqref{rieszde} and Lemma \ref{citetaylor} to see that if $\rho(x)\geq A_{0}^{2}\kappa_{[\alpha]}2^{k+4}$, then
\begin{align}\label{caaa3}
|X_{i} R^{\alpha}\ast A_{k}^{\alpha}K_{\alpha}^{0}(x)|\lesssim \frac{2^{([\alpha]+1-\alpha)k}}{\rho(x)^{\mathbb{Q}+2+[\alpha]-\alpha}}\|\Omega\|_{L^{1}(\Sigma)}.
\end{align}
Next, we combine the estimates \eqref{baaa1} and \eqref{caaa3} to estimate $X_{i}\Delta[2^{k-N(j)}]\phi\ast R^{\alpha}\ast A_{k}^{\alpha}K_{\alpha}^{0}$. Since $\supp\phi\subset\{x\in\HH:\rho(x)\leq\frac{1}{100}\}$ and $\rho(x)\geq A_{0}^{3}\kappa_{[\alpha]}2^{k+5}$, we have
\begin{align*}
|X_{i} \Delta[2^{k-N(j)}]\phi\ast R^{\alpha}\ast A_{k}^{\alpha}K_{\alpha}^{0}(x)|
&=\int_{d(x,z)\geq A_{0}^{2}\kappa_{[\alpha]} 2^{k+4}}|\Delta[2^{k-N(j)}]\phi(z)||X_{i} R^{\alpha}\ast A_{k}^{\alpha}K_{\alpha}^{0}(z^{-1}x)|dz\\
&\lesssim \|\Omega\|_{L^{1}(\Sigma)}\int_{d(x,z)\geq A_{0}^{2}\kappa_{[\alpha]} 2^{k+4}}|\Delta[2^{k-N(j)}]\phi(z)|\frac{2^{([\alpha]+1-\alpha)k}}{d(x,z)^{\mathbb{Q}+2+[\alpha]-\alpha}}dz\nonumber\\
&\lesssim \frac{2^{([\alpha]+1-\alpha)k}}{\rho(x)^{\mathbb{Q}+2+[\alpha]-\alpha}}\|\Omega\|_{L^{1}(\Sigma)}.
\end{align*}
This, together with triangle's inequality and the fact that $N(j-1)\leq N(j)$, implies that for any $i=1,2,\cdots,n$,
\begin{align}\label{gra}
|X_{i} K_{j}^{\alpha,N}(x)|&\lesssim \sum_{k\in\mathbb{Z}}\|\Omega\|_{L^{q}(\Sigma)}\frac{2^{N(j)}}{2^{(\mathbb{Q}+1)k}}\chi_{\rho(x)\leq A_{0}^{3}\kappa_{[\alpha]} 2^{k+5}}+\sum_{k\in\mathbb{Z}}\|\Omega\|_{L^{q}(\Sigma)}\frac{2^{([\alpha]+1-\alpha)k}}{\rho(x)^{\mathbb{Q}+2+[\alpha]-\alpha}}\chi_{\rho(x)>A_{0}^{3}\kappa_{[\alpha]}2^{k+5}}\nonumber\\
&\lesssim 2^{N(j)}\frac{\|\Omega\|_{L^{q}(\Sigma)}}{\rho(x)^{\mathbb{Q}+1}}.
\end{align}
This, in combination with the mean value theorem on homogeneous groups (see for example \cite{FoSt}), implies that if $\rho(x)\geq 2A_{0}\rho(y)$, then
\begin{align*}
|K_{j}^{\alpha,N}(y^{-1}x)-K_{j}^{\alpha,N}(x)|\lesssim 2^{N(j)}\frac{\|\Omega\|_{L^{q}(\Sigma)}}{\rho(x)^{\mathbb{Q}+1}}\rho(y).
\end{align*}
This, combined with \eqref{kkkkkl}, yields
\begin{align*}
|K_{j}^{\alpha,N}(y^{-1}x)-K_{j}^{\alpha,N}(x)|\lesssim \frac{\omega_{j}(\rho(y)/\rho(x))}{\rho(x)^{\mathbb{Q}}},
\end{align*}
where $\omega_{j}(t)\leq \|\Omega\|_{L^{q}(\Sigma)}\min\{1,2^{N(j)}t\}$. Then a direct calculation shows that
\begin{align*}
\int_{0}^{1}\omega_{j}(t)\frac{dt}{t}\lesssim(1+N(j))\|\Omega\|_{q}.
\end{align*}
This ends the proof of Lemma \ref{cal}.
\end{proof}

\subsection{Quantitative weighted bounds for $ \sup\limits_{k\in\mathbb{Z}}|T_{\Omega,\alpha}f\ast\Delta[2^{k}]\phi |$}\label{seee}
Let $\phi$ be a cut-off function defined in Section \ref{preliminariessec}. Then the following $(L_{\alpha}^{p}(w), L^{p}(w))$ boundedness of discrete maximal function holds.
\begin{proposition}\label{discre}
For any $1<p<\infty$ and $w\in A_{p}$, there exists a constant $C_{\mathbb{Q},\alpha,p,q}>0$ such that
\begin{align*}
\Big\|\sup\limits_{k\in\mathbb{Z}}|T_{\Omega,\alpha}f\ast\Delta[2^{k}]\phi |\Big\|_{L^{p}(w)}\leq C_{\mathbb{Q},\alpha,p,q}\|\Omega\|_{L^{q}(\Sigma)}\{w\}_{A_p}(w)_{A_p}\|f\|_{L_{\alpha}^{p}(w)}.
\end{align*}
\end{proposition}
\begin{proof}
It suffices to show
\begin{align}\label{goalmax}
\Big\|\sup\limits_{k\in\mathbb{Z}}|T_{\Omega,\alpha}(-\Delta_{\HH})^{-\alpha/2}f\ast\Delta[2^{k}]\phi |\Big\|_{L^{p}(w)}\leq C_{\mathbb{Q},\alpha,p,q}\|\Omega\|_{L^{q}(\Sigma)}\{w\}_{A_p}(w)_{A_p}\|f\|_{L^{p}(w)}.
\end{align}
To this end, we apply the decomposition in Section \ref{kf} to see that
\begin{align*}
T_{\Omega,\alpha}(-\Delta_{\HH})^{-\alpha/2}f\ast \Delta[2^{k}]\phi =\sum_{j=0}^{\infty}\tilde{T}_{j}^{\alpha,N}f\ast\Delta[2^{k}]\phi.
\end{align*}
Thus, for $1<p<\infty$ and $w\in A_{p}$,
\begin{align*}
\Big\|\sup\limits_{k\in\mathbb{Z}}|T_{\Omega,\alpha}(-\Delta_{\HH})^{-\alpha/2}f\ast\Delta[2^{k}]\phi |\Big\|_{L^{p}(w)}
\leq \sum_{j=0}^{\infty}\Big\|\sup\limits_{k\in\mathbb{Z}}|\tilde{T}_{j}^{\alpha,N}f\ast\Delta[2^{k}]\phi|\Big\|_{L^{p}(w)}.
\end{align*}
It follows easily from Lemma \ref{ckt} and the equality \eqref{equa} that for $j\geq 1$,
\begin{align}\label{08}
&\Big\|\sup\limits_{k\in\mathbb{Z}}|\tilde{T}_{0}^{\alpha,N}f\ast\Delta[2^{k}]\phi|\Big\|_{2}\lesssim \|\tilde{T}_{0}^{\alpha,N}f\|_{2}\lesssim \|\Omega\|_{L^{q}(\Sigma)}\|f\|_{2},\nonumber\\
&\Big\|\sup\limits_{k\in\mathbb{Z}}|\tilde{T}_{j}^{\alpha,N}f\ast\Delta[2^{k}]\phi|\Big\|_{2}\lesssim \|\tilde{T}_{j}^{\alpha,N}f\|_{2}\lesssim 2^{-\tau N(j-1)}\|\Omega\|_{L^{q}(\Sigma)}\|f\|_{2}.
\end{align}
To continue, we claim that for any $j\geq 0$,
\begin{align}\label{phiy1}
\Big\|\sup\limits_{k\in\mathbb{Z}}|\tilde{T}_{j}^{\alpha,N}f\ast\Delta[2^{k}]\phi|\Big\|_{L^{p}(w)}\lesssim (1+N(j))\|\Omega\|_{L^{q}(\Sigma)}\{w\}_{A_{p}}\|f\|_{L^{p}(w)}.
\end{align}
We assume \eqref{phiy1} for the moment, whose proof will be given later. Taking $w=1$ in \eqref{phiy1} and then applying interpolation between \eqref{08} and \eqref{phiy1}, we see that for any $j\geq 1$,
\begin{align*}
\Big\|\sup\limits_{k\in\mathbb{Z}}|\tilde{T}_{j}^{\alpha,N}f\ast\Delta[2^{k}]\phi|\Big\|_{p}\lesssim \|\tilde{T}_{j}^{\alpha,N}f\|_{p}\lesssim 2^{-\tau_{p} N(j-1)}(1+N(j))\|\Omega\|_{L^{q}(\Sigma)}\|f\|_{p}
\end{align*}
for some constant $\tau_{p}>0$.

Next by choosing $\epsilon=\frac{1}{2}c_{\mathbb{Q}}/(w)_{A_{p}}$, we see that the estimate \eqref{phiy1} gives
\begin{align}
\Big\|\sup\limits_{k\in\mathbb{Z}}|\tilde{T}_{j}^{\alpha,N}f\ast\Delta[2^{k}]\phi|\Big\|_{L^{p}(w^{1+\epsilon})}\lesssim (1+N(j))\|\Omega\|_{L^{q}(\Sigma)}\{w\}_{A_{p}}^{1+\epsilon}\|f\|_{L^{p}(w^{1+\epsilon})}.
\end{align}
Then applying interpolation with change of measure, we obtain that
\begin{align}
\Big\|\sup\limits_{k\in\mathbb{Z}}|\tilde{T}_{j}^{\alpha,N}f\ast\Delta[2^{k}]\phi|\Big\|_{L^{p}(w)}\lesssim \|\Omega\|_{L^{q}(\Sigma)}(1+N(j))2^{-\tau_{p,\mathbb{Q}}N(j-1)/(w)_{A_{p}}}\{w\}_{A_{p}}\|f\|_{L^{p}(w)}
\end{align}
for some constant $\tau_{p,\mathbb{Q}}>0$.

If we choose $N(j)=2^{j}$, then
\begin{align*}
\Big\|\sup\limits_{k\in\mathbb{Z}}|T_{\Omega,\alpha}(-\Delta_{\HH})^{-\alpha/2}f\ast\Delta[2^{k}]\phi |\Big\|_{L^{p}(w)}
&\leq \sum_{j=0}^{\infty}\Big\|\sup\limits_{k\in\mathbb{Z}}|\tilde{T}_{j}^{\alpha,N}f\ast\Delta[2^{k}]\phi|\Big\|_{L^{p}(w)}\\
&\lesssim \sum_{j=0}^{\infty}\|\Omega\|_{L^{q}(\Sigma)}(1+N(j))2^{-\tau_{p,\mathbb{Q}}N(j-1)/(w)_{A_{p}}}\{w\}_{A_{p}}\|f\|_{L^{p}(w)}\\
&\lesssim \|\Omega\|_{L^{q}(\Sigma)}\{w\}_{A_{p}}(w)_{A_{p}}\|f\|_{L^{p}(w)}.
\end{align*}

Now we return to give the proof of \eqref{phiy1}. To this end, we first give the kernel estimates of $\{K_{j}^{\alpha,N}\ast\Delta[2^{k}]\phi \}$. We would like to establish the following inequalities.
\end{proof}
\begin{lemma}\label{caal}
There exists a constant $C_{\mathbb{Q},\alpha,q}>0$ such that for any $j\geq 0$,
\begin{align}\label{siz}
\sup\limits_{k\in\mathbb{Z}}|K_{j}^{\alpha,N}\ast\Delta[2^{k}]\phi (x)|\leq C_{\mathbb{Q},\alpha,q}\frac{\|\Omega\|_{L^{q}(\Sigma)}}{\rho(x)^{\mathbb{Q}}},
\end{align}
and if $\rho(x)\geq 2A_{0}\rho(y)$,
\begin{align}\label{smo}
\sup\limits_{k\in\mathbb{Z}}|K_{j}^{\alpha,N}\ast \Delta[2^{k}]\phi(y^{-1}x)-K_{j}^{\alpha,N}\ast \Delta[2^{k}]\phi (x)|\leq C_{\mathbb{Q},\alpha,q} \frac{\omega_{j}(\rho(y)/\rho(x))}{\rho(x)^{\mathbb{Q}}},
\end{align}
and
\begin{align}\label{smo2}
\sup\limits_{k\in\mathbb{Z}}|K_{j}^{\alpha,N}\ast \Delta[2^{k}]\phi(x y^{-1})-K_{j}^{\alpha,N}\ast \Delta[2^{k}]\phi (x)|\leq C_{\mathbb{Q},\alpha,q} \frac{\omega_{j}(\rho(y)/\rho(x))}{\rho(x)^{\mathbb{Q}}},
\end{align}
where $
\omega_{j}(t)\leq \|\Omega\|_{L^{q}(\Sigma)}\min\{1,2^{N(j)}t\}.$
\end{lemma}
\begin{proof}
We first verify \eqref{siz}.

\textbf{Case 1:} If $\rho(x)\leq A_{0}\kappa_{0}2^{k+4}$, then since $\int_{\HH}K_{j}^{\alpha,N}(z)dz=0$,
\begin{align*}
K_{j}^{\alpha,N}\ast \Delta[2^{k}]\phi (x)=\int_{\HH}(\Delta[2^{k}]\phi(z^{-1}x)-\Delta[2^{k}]\phi(x))K_{j}^{\alpha,N}(z)dz.
\end{align*}
By the support of $\Delta[2^{k}]\phi$, we see that $\rho(z)\leq A_{0}^{2}\kappa_{0}2^{k+5}$. This, combined with \eqref{sizec} and the mean value theorem on homogeneous groups (see \cite{FoSt}) yields
\begin{align}\label{nmnm}
|K_{j}^{\alpha,N}\ast\Delta[2^{k}]\phi (x)|
&\leq \int_{\rho(z)\leq A_{0}^{2}\kappa_{0}2^{k+5}}|\Delta[2^{k}]\phi(z^{-1}x)-\Delta[2^{k}]\phi(x))||K_{j}^{\alpha,N}(z)|dz\nonumber\\
&\lesssim\|\Omega\|_{L^{q}(\Sigma)}\int_{\rho(z)\leq A_{0}^{2}\kappa_{0}2^{k+5}}|\Delta[2^{k}]\phi(z^{-1}x)-\Delta[2^{k}]\phi(x)|\frac{1}{\rho(z)^{\mathbb{Q}}}dz\nonumber\\
&\lesssim \frac{\|\Omega\|_{L^{q}(\Sigma)}}{2^{(\mathbb{Q}+1)k}}\int_{\rho(z)\leq A_{0}^{2}\kappa_{0}2^{k+5}}\frac{1}{\rho(z)^{\mathbb{Q}-1}}dz\lesssim \frac{\|\Omega\|_{L^{q}(\Sigma)}}{2^{\mathbb{Q}k}}\lesssim \frac{\|\Omega\|_{L^{q}(\Sigma)}}{\rho(x)^{\mathbb{Q}}}.
\end{align}

\textbf{Case 2:} If $\rho(x)\geq A_{0}\kappa_{0}2^{k+4}$, then by the support of $\Delta[2^{k}]\phi$, we see that $\rho(z)\geq \frac{\rho(x)}{2}$. This, in combination with \eqref{sizec}, implies
\begin{align*}
|K_{j}^{\alpha,N}\ast \Delta[2^{k}]\phi (x)|\lesssim\|\Omega\|_{L^{q}(\Sigma)}\int_{\rho(z)\geq\frac{\rho(x)}{2}}|\Delta[2^{k}]\phi(z^{-1}x)|\frac{1}{\rho(z)^{\mathbb{Q}}}dz\lesssim\frac{\|\Omega\|_{L^{q}(\Sigma)}}{\rho(x)^{\mathbb{Q}}}.
\end{align*}
Combining the estimates of two cases, we get \eqref{siz}.

Next, we verify \eqref{smo}. By the mean value theorem on homogeneous groups, it suffices to estimate $|X_{i} K_{j}^{\alpha,N}\ast \Delta[2^{k}]\phi(x)|$ for any $i=1,2,\cdots,n$.

\textbf{Case 1:} If $\rho(x)\leq A_{0}\kappa_{0}2^{k+4}$,
then \eqref{nmnm} holds with $\phi$ replaced by $X_{i}\phi$. Thus,
\begin{align*}
|X_{i} K_{j}^{\alpha,N}\ast\Delta[2^{k}]\phi (x)|=\frac{1}{2^{k}}|K_{j}^{\alpha,N}\ast\Delta[2^{k}](X_{i} \phi) (x)|
\lesssim \frac{\|\Omega\|_{L^{q}(\Sigma)}}{2^{(\mathbb{Q}+1)k}}\lesssim \frac{\|\Omega\|_{L^{q}(\Sigma)}}{\rho(x)^{\mathbb{Q}+1}}.
\end{align*}

\textbf{Case 2:} If $\rho(x)\geq A_{0}\kappa_{0}2^{k+4}$, then it follows from the support property of $\phi$ and \eqref{gra} that
\begin{align*}
|X_{i} K_{j}^{\alpha,N}\ast\Delta[2^{k}]\phi (x)|&\lesssim\|\Omega\|_{L^{q}(\Sigma)}\int_{\rho(z)\geq\frac{\rho(x)}{2}}|\Delta[2^{k}]\phi(z^{-1}x)|\frac{2^{N(j)}}{\rho(z)^{\mathbb{Q}+1}}dz\lesssim\frac{2^{N(j)}\|\Omega\|_{L^{q}(\Sigma)}}{\rho(x)^{\mathbb{Q}+1}}.
\end{align*}

Combining the estimates of two cases, we see that for if $\rho(x)\geq 2A_{0}\rho(y)$,
\begin{align}
\sup\limits_{k\in\mathbb{Z}}|K_{j}^{\alpha,N}\ast \Delta[2^{k}]\phi(y^{-1}x)-K_{j}^{\alpha,N}\ast\Delta[2^{k}]\phi (x)|\lesssim \|\Omega\|_{L^{q}(\Sigma)}2^{N(j)}\frac{\rho(y)}{\rho(x)^{\mathbb{Q}+1}}.
\end{align}
This, in junction with \eqref{siz}, yields \eqref{smo}. Following a similar calculation, we obtain \eqref{smo2} as well. This ends the proof of Lemma \ref{caal}.
\end{proof}

In the following, we recall the grand maximal truncated operator $\mathcal{M}_{T}$ on homogeneous groups defined as follows.
\begin{align*}
\mathcal{M}_{T}f(x):=\sup\limits_{B\ni x}\esssup\limits_{\xi\in B}|T(f\chi_{\HH\backslash C_{A_{0}}B})(\xi)|,
\end{align*}
where $C_{A_{0}}$ is a fixed constant depending only on $A_{0}$ (for the precise definition of $C_{A_0}$, we refer the readers to the notation $C_{\tilde{j}_{0}}$ in \cite{DGKLWY}), and the first supremum is taken over all balls $B\subset\HH$ containing~$x$.
\begin{lemma}\label{lem 3} Let  T be a sublinear operator. Assume that the operators $T$ and $\mathcal{M}_T$ are of weak type $(1,1)$, then for every compactly supported $f\in L^1(\HH),$ we have
\begin{align*}
|Tf(x)|\leq C_{\mathbb{Q}}(\|T\|_{L^{1}\rightarrow L^{1,\infty}}+\|\mathcal{M}_{T}\|_{L^{1}\rightarrow L^{1,\infty}})\mathcal{A}_{\mathcal{S}}(|f|)(x),
\end{align*}
where $\mathcal{A}_{\mathcal{S}}$ is the sparse operator (see for example \cite{DGKLWY}).
\end{lemma}
\begin{proof}
The result in the Euclidean setting was shown in \cite{Ler1}. In our setting, by noting that the homogeneous group is a special case of space of homogeneous type in the sense of Coifman--Weiss, we can obtain this result by following the approach in \cite{Ler1} and using the dyadic grid and sparse operator on space of homogeneous type as developed in \cite{DGKLWY}. We leave the details to readers.
\end{proof}

For simplicity, we denote $\mathcal{T}_{j}^{\alpha,N}f:=\sup\limits_{k\in\mathbb{Z}}|\tilde{T}_{j}^{\alpha,N}f\ast\Delta[2^{k}]\phi|$ and $\mathcal{M}_{j}^{\alpha,N}f:=\mathcal{M}_{\mathcal{T}_{j}^{\alpha,N}}f$.  It follows from  Lemma \ref{lem 3} and the quantitative weighted $L^{p}$ boundedness of sparse operators that the estimate \eqref{phiy1} is a direct consequence of the following lemma.
\begin{lemma}\label{twoweak}
There exists a constant $C_{\mathbb{Q},\alpha,q}>0$ such that for any $j\geq 0$,
\begin{align}\label{TT}
\|\mathcal{T}_{j}^{\alpha,N}f\|_{L^{1,\infty}}\leq C_{\mathbb{Q},\alpha,q}\|\Omega\|_{L^{q}(\Sigma)}(1+N(j))\|f\|_{1},
\end{align}
and
\begin{align}\label{MM}
\|\mathcal{M}_{j}^{\alpha,N}f\|_{L^{1,\infty}}\leq C_{\mathbb{Q},\alpha,q}\|\Omega\|_{L^{q}(\Sigma)}(1+N(j))\|f\|_{1}.
\end{align}
\end{lemma}
\begin{proof}

To begin with, it follows from  \eqref{smo} that the following uniform H\"{o}rmander inequality holds:
\begin{align*}
\int_{\rho(x)>2A_0\rho(y)}\sup\limits_{k\in\mathbb{Z}}|K_{j}^{\alpha,N}\ast \Delta[2^{k}]\phi (y^{-1}x)-K_{j}^{\alpha,N}\ast\Delta[2^{k}]\phi (x)|dx
\lesssim \|\omega_j\|_{{\rm Dini}}\lesssim(1+N(j))\|\omega\|_{L^q(\Sigma)}.
\end{align*}
Then the estimate \eqref{TT} is a standard argument of Calder\'{o}n-Zygmund decomposition (see for example \cite{JDbook,s93}).

Next, we verify \eqref{MM}. Let $x$, $\xi\in B:=B(x_{0},r)$. Let $B_{x}$ be the closed ball centered at $x$ with radius $4(A_{0}^{2}+C_{A_{0}})r$. Then $C_{A_{0}}B\subset B_{x}$, and we obtain
\begin{align*}
&\big|\sup\limits_{k\in\mathbb{Z}}|\tilde{T}_{j}^{\alpha,N}(f\chi_{\HH\backslash C_{A_{0}}B})\ast\Delta[2^{k}]\phi(\xi)|\big|\\
&\leq \big|\sup\limits_{k\in\mathbb{Z}}|\tilde{T}_{j}^{\alpha,N}(f\chi_{\HH\backslash B_{x}})\ast\Delta[2^{k}]\phi(\xi)|-\sup\limits_{k\in\mathbb{Z}}|\tilde{T}_{j}^{\alpha,N}(f\chi_{\HH\backslash B_{x}})\ast\Delta[2^{k}]\phi(x)|\big|\\
&\quad+\sup\limits_{k\in\mathbb{Z}}|\tilde{T}_{j}^{\alpha,N}(f\chi_{B_{x}\backslash C_{A_{0}}B})\ast\Delta[2^{k}]\phi(\xi)|+\sup\limits_{k\in\mathbb{Z}}|\tilde{T}_{j}^{\alpha,N}(f\chi_{\HH\backslash B_{x}})\ast\Delta[2^{k}]\phi(x)|=:{\rm I+II+III}.
\end{align*}

\textbf{Estimate of ${\rm I}$.} By \eqref{smo2},
\begin{align}\label{001}
{\rm I}&\leq \int_{\HH\backslash B_{x}}\sup\limits_{k\in\mathbb{Z}}|K_{j}^{\alpha,N}\ast \Delta[2^{k}]\phi (y^{-1}\xi)-K_{j}^{\alpha,N}\ast \Delta[2^{k}]\phi(y^{-1}x)||f(y)|dy\nonumber\\
&\lesssim\int_{d(x,y)>4A_{0}^{2}r}\frac{\omega_{j}(d(x,\xi)/d(x,y))}{d(x,y)^{\mathbb{Q}}}|f(y)|dy\nonumber\\
&\lesssim \sum_{\ell=0}^{\infty}\omega_{j}(2^{-\ell})(2^{\ell}r)^{-\mathbb{Q}}\int_{2^{\ell+2}A_{0}^{2}r\leq d(x,y)\leq 2^{\ell+3}A_{0}^{2}r}|f(y)|dy\lesssim \|\omega_{j}\|_{{\rm Dini}}Mf(x).
\end{align}

\textbf{Estimate of ${\rm II}$.} It can be verified directly that $d(x,y)\sim d(\xi,y)\sim r$ whenever $y\in B_{x}\backslash C_{A_{0}}B$. Then it follows from \eqref{siz} that
\begin{align}\label{002}
{\rm II}&\leq \int_{B_{x}\backslash C_{A_{0}}B}\sup\limits_{k\in\mathbb{Z}}|K_{j}^{\alpha,N}\ast \Delta[2^{k}]\phi (y^{-1}\xi)||f(y)|dy\nonumber\\
&\lesssim \|\Omega\|_{L^{q}(\Sigma)}\int_{ d(x,y)\sim r}\frac{1}{d(x,y)^{\mathbb{Q}}}|f(y)|dy\lesssim \|\Omega\|_{L^{q}(\Sigma)}Mf(x).
\end{align}

\textbf{Estimate of ${\rm III}$.} Note that
\begin{align*}
{\rm III}
&\leq \sup\limits_{k\in\mathbb{Z}}\left|\int_{d(x,y)\leq A_{0}\kappa_{0}2^{k+4}}K_{j}^{\alpha,N}\ast \Delta[2^{k}]\phi(y^{-1}x)(f\chi_{\HH\backslash B_{x}})(y)dy\right|\\
&\quad+ \sup\limits_{k\in\mathbb{Z}}\left|\int_{d(x,y)\geq A_{0}\kappa_{0}2^{k+4}}K_{j}^{\alpha,N}\ast \Delta[2^{k}]\phi(y^{-1}x)(f\chi_{\HH\backslash B_{x}})(y)dy\right|=:{\rm III_{1}+III_{2}}.
\end{align*}
For the term ${\rm III_{1}}$, it follows from \eqref{nmnm} that $|K_{j}^{\alpha,N}\ast\Delta[2^{k}]\phi (x)|\lesssim 2^{-\mathbb{Q}k}\|\Omega\|_{L^{q}(\Sigma)}$ whenever $\rho(x)\leq A_{0}\kappa_{0}2^{k+4}$. Hence,
\begin{align}\label{31}
{\rm III_{1}}\lesssim \|\Omega\|_{L^{q}(\Sigma)}\sup\limits_{k\in\mathbb{Z}}2^{-\mathbb{Q}k}\int_{d(x,y)\leq A_{0}\kappa_{0}2^{k+4}}|f(y)|dy\lesssim \|\Omega\|_{L^{q}(\Sigma)}Mf(x).
\end{align}
For the term ${\rm III_{2}}$,
\begin{align*}
{\rm III_{2}}&\leq
\sup\limits_{k\in\mathbb{Z}}\left|\int_{d(x,y)\geq A_{0}\kappa_{0}2^{k+4}}K_{j}^{\alpha,N}\ast (\Delta[2^{k}]\phi-\delta_{0})(y^{-1}x)(f\chi_{\HH\backslash B_{x}})(y)dy\right|\\
&\quad+\sup\limits_{k\in\mathbb{Z}}\left|\int_{d(x,y)\geq A_{0}\kappa_{0}2^{k+4}}K_{j}^{\alpha,N}(y^{-1}x)(f\chi_{\HH\backslash B_{x}})(y)dy\right|,
\end{align*}
where $\delta_{0}$ is the Dirac measure at 0. To continue, we first estimate $K_{j}^{\alpha,N}\ast(\Delta[2^{k}]\phi-\delta_{0})(x)$ for $\rho(x)\geq A_{0}\kappa_{0}2^{k+4}$. By the fact that $\int_{\HH}\Delta[2^{k}]\phi dx=1$  and \eqref{jkl},
\begin{align*}
|K_{j}^{\alpha,N}\ast(\Delta[2^{k}]\phi-\delta_{0})(x)|
&=\left|\int_{\HH}\Delta[2^{k}]\phi(y^{-1}x)K_{j}^{\alpha,N}(y)dy-K_{j}^{\alpha,N}(x)\right|\\
&=\left|\int_{\HH}\Delta[2^{k}]\phi(y^{-1}x)(K_{j}^{\alpha,N}(y)-K_{j}^{\alpha,N}(x))dy\right|\\
&\lesssim \int_{\HH}|\Delta[2^{k}]\phi(y^{-1}x)|\frac{\omega_{j}(2^{k}/\rho(x))}{\rho(x)^{\mathbb{Q}}}dy\lesssim \frac{\omega_{j}(2^{k}/\rho(x))}{\rho(x)^{\mathbb{Q}}}.
\end{align*}
Therefore,
\begin{align*}
{\rm III_{2}}&\lesssim \sup\limits_{k\in\mathbb{Z}}\int_{d(x,y)\geq A_{0}\kappa_{0}2^{k+4}}\frac{\omega_{j}(2^{k}/d(x,y))}{d(x,y)^{\mathbb{Q}}}|(f\chi_{\HH\backslash B_{x}})(y)|dy\\
&\quad+\sup\limits_{k\in\mathbb{Z}}\bigg|\int_{d(x,y)\geq \max\{A_{0}\kappa_{0}2^{k+4}, 4(A_{0}^{2}+C_{A_{0}})r\}}K_{j}^{\alpha,N}(y^{-1}x)f(y)dy\bigg|\lesssim\|\omega_{j}\|_{{\rm Dini}}Mf(x)+\tilde{T}_{j}^{\alpha,N,\#}f(x),
\end{align*}
where $\tilde{T}_{j}^{\alpha,N,\#}$ is the truncated maximal operator of $\tilde{T}_{j}^{\alpha,N}$ defined by
\begin{align}\label{maxdefi}
\tilde{T}_{j}^{\alpha,N,\#}f(x)=:\sup\limits_{\varepsilon>0}\left|\int_{d(x,y)>\varepsilon}K_j^{\alpha,N}(x,y)f(y)dy\right|.
\end{align}

This, along with \eqref{31}, implies
\begin{align}\label{003}
{\rm III}\lesssim(\|\Omega\|_{L^{q}(\Sigma)}+\|\omega_{j}\|_{{\rm Dini}})Mf(x)+\tilde{T}_{j}^{\alpha,N,*}f(x).
\end{align}

Combining the estimates \eqref{001}, \eqref{002} and \eqref{003}, we see that for any $x,\xi\in B$,
\begin{align*}
\sup\limits_{k\in\mathbb{Z}}|\tilde{T}_{j}^{\alpha,N}(f\chi_{\HH\backslash B_{x}})\ast\Delta[2^{k}]\phi(\xi)|\lesssim(\|\Omega\|_{L^{q}(\Sigma)}+\|\omega_{j}\|_{{\rm Dini}})Mf(x)+\tilde{T}_{j}^{\alpha,N,\#}f(x),
\end{align*}
which implies that
\begin{align*}
\mathcal{M}_{j}^{\alpha,N}f(x)\lesssim(\|\Omega\|_{L^{q}(\Sigma)}+\|\omega_{j}\|_{{\rm Dini}})Mf(x)+\tilde{T}_{j}^{\alpha,N,\#}f(x).
\end{align*}
By sparse domination theorem (see the arxiv version for more details),
\begin{align*}
|\tilde{T}_{j}^{\alpha,N,\#}f(x)|\lesssim (\|\Omega\|_{L^{q}(\Sigma)}+\|\omega_{j}\|_{{\rm Dini}}+\|\tilde{T}_{j}^{\alpha,N}\|_{2\rightarrow 2})\mathcal{A}_{\mathcal{S}}(|f|)(x).
\end{align*}
Moreover, it follows from Lemma \ref{ckt} that $\|\tilde{T}_{j}^{\alpha,N}\|_{2\rightarrow 2}\lesssim 2^{-\alpha N(j-1)}\|\Omega\|_{L^{q}(\Sigma)}\|f\|_{2}.$ Combining the above inequalities, the weak type (1,1) boundedness of the Hardy-Littlewood maximal operators and of sparse operators (see for example \cite{LN}) together, we see that
\begin{align*}
\|\mathcal{M}_{j}^{\alpha,N}f\|_{L^{1,\infty}}\lesssim \|\Omega\|_{L^{q}(\Sigma)}(1+N(j))\|f\|_{1},
\end{align*}
which verifies \eqref{MM}. This finishes the proof of Lemma \ref{twoweak}.
\end{proof}
\subsection{Proof of Theorem \ref{main2}}
In this subsection, we modify the ideas in Section \ref{Basic reduction} to give the proof of Theorem \ref{main2}. To begin with, recall from \eqref{maximalcontrol} that
\begin{align*}
T_{\Omega,\alpha}^{\#}f(x)&\leq M_{\Omega,\alpha}f(x)+\sup\limits_{k\in\mathbb{Z}}|T_{\Omega,\alpha}^{k}f(x)|.
\end{align*}
Denote $v_{\alpha,\varepsilon}(x):=\frac{\Omega(x)}{\rho(x)^{\mathbb{Q}+\alpha}}\chi_{\varepsilon< \rho(x)\leq 2^{[\log \varepsilon]+1}}$.  If $\alpha=0$, then it is direct that
\begin{align}\label{xixixi}
|v_{\alpha,\varepsilon}(x)|\lesssim \|\Omega\|_{L^{\infty}(\Sigma)}\varepsilon^{-\mathbb{Q}}\chi_{\rho(x)\leq 2^{[\log \varepsilon]+1}}\lesssim \|\Omega\|_{L^{\infty}(\Sigma)}2^{-\mathbb{Q}[\log \varepsilon]}\chi_{\rho(x)\leq 2^{[\log \varepsilon]+1}}.
\end{align}
If $0<\alpha<\mathbb{Q}$, then similar to the proofs of \eqref{caaa1} and \eqref{caaa2}, we get that
\begin{align}\label{yiyiyi}
|R^{\alpha}\ast v_{\alpha,\varepsilon}(x)|\lesssim \|\Omega\|_{L^{q}(\Sigma)}2^{-\mathbb{Q}[\log \varepsilon]}\chi_{\rho(x)\leq A_{0}^{3}\kappa_{[\alpha]}2^{[\log \varepsilon]+5}}+ \frac{2^{([\alpha]+1-\alpha)[\log \varepsilon]}}{\rho(x)^{\mathbb{Q}+1+[\alpha]-\alpha}}\|\Omega\|_{L^{1}(\Sigma)}\chi_{\rho(x)\geq A_{0}^{3}\kappa_{[\alpha]}2^{[\log \varepsilon]+5}}.
\end{align}
In both cases, the above inequalities imply that for $0\leq\alpha<\mathbb{Q}$,
\begin{align*}
|M_{\Omega,\alpha}(-\Delta_{\HH})^{-\alpha/2}f(x)|=\sup\limits_{\varepsilon>0}|f\ast R^{\alpha}\ast v_{\alpha,\varepsilon}(x)|\lesssim\|\Omega\|_{L^{q}(\Sigma)}Mf(x).
\end{align*}
Then by the sharp weighted boundedness of the Hardy--Littlewood maximal operator $M$ (see for example \cite[Corollary 1.10]{HP}), for any $1<p<\infty$,
 \begin{align}\label{HLM}
 \|Mf\|_{L^p(w)}\lesssim \{w\}_{A_{p}}\|f\|_{L^{p}(w)},
 \end{align}
and therefore, \begin{align}\label{a0}\|M_{\Omega,\alpha}f\|_{L^p(w)}\lesssim  \|\Omega\|_{L^{q}(\Sigma)}\{w\}_{A_p}\| f\|_{L_\alpha^p(w)}.\end{align}
Hence, to prove Theorem \ref{main2}, it suffices to show the $(L_{\alpha}^p(w), L^p(w))$ boundedness of $\sup\limits_{k\in\mathbb{Z}}|T_{\Omega,\alpha}^{k}f(x)|$.
To this end, we define the smooth truncated kernel by
\begin{align*}
K_{\alpha,k}(x):=K_{\alpha}(x)\int_{-\infty}^{+\infty}t\varphi(t)\chi_{\rho(x)\geq 2^{k+1} t}(x)dt,
\end{align*}
and the corresponding smooth truncated singular integral operator by
\begin{align}\label{tildeT}
\tilde{T}_{\Omega,\alpha}^{k}f(x):=f\ast K_{\alpha,k}(x).
\end{align}

We next show that the weighted estimate of $\sup\limits_{k\in\mathbb{Z}}|T_{\Omega,\alpha}^{k}f(x)|$ is equivalent to that of $\sup\limits_{k\in\mathbb{Z}}|\tilde{T}_{\Omega,\alpha}^{k}f(x)|$ with the same bound. To this end, we note that $K_{\alpha,k}(x)=cK_{\alpha}(x)\chi_{\rho(x)\geq 2^{k+1}}(x)$ when $\rho(x)\geq 2^{k+2}$ or $\rho(x)\leq 2^{k}$, where $c:=\int_{-\infty}^{+\infty}t\varphi(t)dt$ and that $K_{\alpha,k}(x)-cK_{\alpha}(x)\chi_{\rho(x)\geq 2^{k+1}}(x)$ satisfies the cancellation condition of order $[\alpha]$. Hence, if $\alpha=0$, then it is direct that
\begin{align}\label{xixixi}
|K_{\alpha,k}(x)-cK_{\alpha}(x)\chi_{\rho(x)\geq 2^{k+1}}(x)|\lesssim \|\Omega\|_{L^{\infty}(\Sigma)}2^{-\mathbb{Q}k}\chi_{\rho(x)\leq 2^{k+2}}.
\end{align}
If $0<\alpha<\mathbb{Q}$, then similar to the proofs of \eqref{caaa1} and \eqref{caaa2}, we have
\begin{align*}
&|R^{\alpha}\ast\big(K_{\alpha,k}(x)-cK_{\alpha}(x)\chi_{\rho(x)\geq 2^{k+1}}(x)\big)|\\
&\lesssim \|\Omega\|_{L^{q}(\Sigma)}2^{-\mathbb{Q}k}\chi_{\rho(x)\leq A_{0}^{3}\kappa_{[\alpha]}2^{k+5}}+\|\Omega\|_{L^{1}(\Sigma)}\frac{2^{([\alpha]+1-\alpha)k}}{\rho(x)^{\mathbb{Q}+1+[\alpha]-\alpha}}\chi_{\rho(x)\geq A_{0}^{3}\kappa_{[\alpha]}2^{k+5}}.
\end{align*}
In both cases,
\begin{align*}
|\tilde{T}_{\Omega,\alpha}^{k}(-\Delta_{\HH})^{-\alpha/2}f(x)-cT_{\Omega,\alpha}^{k}(-\Delta_{\HH})^{-\alpha/2}f(x)|
\lesssim\|\Omega\|_{L^{q}(\Sigma)}Mf(x).
\end{align*}
This, in combination with the inequality \eqref{HLM}, implies that to prove Theorem \ref{main2}, it suffices to obtain the $(L_{\alpha}^p(w), L^p(w))$ boundedness of $\sup\limits_{k\in\mathbb{Z}}|\tilde{T}_{\Omega,\alpha}^{k}f(x)|$. Observe that
\begin{align*}
\tilde{T}_{\Omega,\alpha}^{k}f(x)= T_{\Omega,\alpha}f\ast\Delta[2^{k}]\phi-\sum_{s=-\infty}^{0}f\ast A_{k+s}^{\alpha}K_{\alpha}^{0}\ast\Delta[2^{k}]\phi+\sum_{s=1}^{\infty}f\ast A_{k+s}^{\alpha}K_{\alpha}^{0}\ast(\delta_{0}-\Delta[2^{k}]\phi).
\end{align*}
From this equality we see that
\begin{align}\label{iiiiii}
&\Big\|\sup\limits_{k\in\mathbb{Z}}|\tilde{T}_{\Omega,\alpha}^{k}(-\Delta_{\HH})^{-\alpha/2}f(x)|\Big\|_{L^{p}(w)}\nonumber\\
&\leq \Big\|\sup\limits_{k\in\mathbb{Z}}|T_{\Omega,\alpha}(-\Delta_{\HH})^{-\alpha/2}f\ast\Delta[2^{k}]\phi|\Big\|_{L^{p}(w)}
+\Big\|\sup\limits_{k\in\mathbb{Z}}\Big|\sum_{s=-\infty}^{0}f\ast R^{\alpha}\ast A_{k+s}^{\alpha}K_{\alpha}^{0}\ast\Delta[2^{k}]\phi\Big|\Big\|_{L^{p}(w)}\nonumber\\
&+\Big\|\sup\limits_{k\in\mathbb{Z}}\Big|\sum_{s=1}^{\infty}f\ast R^{\alpha}\ast A_{k+s}^{\alpha}K_{\alpha}^{0}\ast(\delta_{0}-\Delta[2^{k}]\phi)\Big|\Big\|_{L^{p}(w)}=:{\rm I+II+III}.
\end{align}

\textbf{Estimate of ${\rm I}$.} By Proposition \ref{discre}, we see that for $q>\mathbb{Q}/\alpha$, $1<p<\infty$ and $w\in A_{p}$,
\begin{align*}
{\rm I}\lesssim\|\Omega\|_{L^{q}(\Sigma)}\{w\}_{A_p}(w)_{A_p}\|f\|_{L^{p}(w)}.
\end{align*}

\textbf{Estimate of ${\rm II}$.} Observe that
\begin{align*}
\sum_{s=-\infty}^{0}A_{k+s}^{\alpha}K_{\alpha}^{0}=K_{\alpha}(x)\int_{-\infty}^{+\infty}t\varphi(t)\chi_{\rho(x)\leq 2^{k+1} t}(x)dt.
\end{align*}
Note that \eqref{fjfjfjfj}-\eqref{estimateI222} hold with $B_{k+s}^{\alpha}\Omega$ replaced by $A_{k+s}^{\alpha}K_{\alpha}^{0}$. In particular, if $0<\alpha<\mathbb{Q}$,
\begin{align}\label{mite}
\sup\limits_{k\in\mathbb{Z}}\Big|\sum_{s=-\infty}^{0}f\ast R^{\alpha}\ast A_{k+s}^{\alpha}K_{\alpha}^{0}\ast\Delta[2^{k}]\phi(x)\Big|\lesssim \|\Omega\|_{L^{1}(\Sigma)}Mf(x).
\end{align}
For the case $\alpha=0$, by the 0-order cancellation property of $\Omega$ and the mean value theorem on homogeneous groups,
\begin{align*}
\Big|\sum_{s=-\infty}^{0}A_{k+s}^{0}K_{0}^{0}\ast\Delta[2^{k}]\phi(x)\Big|
&=2^{-\mathbb{Q}k}\bigg|\int_{\HH}\sum_{s=-\infty}^{0}A_{k+s}^{0}K_{0}^{0}(y)\big(\phi(2^{-k}\circ(y^{-1}x))-\phi(2^{-k}\circ x)\big)dy\bigg|\\
&\lesssim 2^{-(\mathbb{Q}+1)k}\int_{\rho(y)\leq 2^{k+2}}\frac{\Omega(y)}{\rho(y)^{\mathbb{Q}-1}}dy\chi_{\rho(x)\leq A_{0}2^{k+3}}\lesssim 2^{-\mathbb{Q}k}\|\Omega\|_{L^{1}(\Sigma)}\chi_{\rho(x)\leq A_{0}2^{k+3}},
\end{align*}
and therefore, \eqref{mite} also holds in this case. This, together with \eqref{HLM}, yields
\begin{align*}
{\rm II}\lesssim \|\Omega\|_{L^{1}(\Sigma)}\{w\}_{A_{p}}\|f\|_{L^{p}(w)}.
\end{align*}

\textbf{Estimate of ${\rm III}$.} We first claim that there exists a constant $\tau>0$, such that for any $s\geq 1$,
\begin{align}\label{makey}
\Big\|\sup\limits_{k\in\mathbb{Z}}|f\ast R^{\alpha}\ast A_{k+s}^{\alpha}K_{\alpha}^{0}\ast (\delta_{0}-\Delta[2^{k}]\phi)|\Big\|_{2}\lesssim 2^{-\tau s}\|\Omega\|_{L^{q}(\Sigma)}\|f\|_{2}.
\end{align}
To this end, by Khinchin's inequality and Lemma \ref{ckt},
\begin{align*}
&\Big\|\sup\limits_{k\in\mathbb{Z}}|f\ast R^{\alpha}\ast A_{k+s}^{\alpha}K_{\alpha}^{0}\ast (\delta_{0}-\Delta[2^{k}]\phi)|\Big\|_{2}
\leq \Big\|\Big(\sum_{k\in\mathbb{Z}}|f\ast R^{\alpha}\ast A_{k+s}^{\alpha}K_{\alpha}^{0}\ast (\delta_{0}-\Delta[2^{k}]\phi)|^{2}\Big)^{1/2}\Big\|_{2}\\
&\lesssim \bigg\|\Big\|\sum_{k\in\mathbb{Z}}r_{k}(t)f\ast R^{\alpha}\ast A_{k+s}^{\alpha}K_{\alpha}^{0}\ast (\delta_{0}-\Delta[2^{k}]\phi)\Big\|_{L^{2}([0,1])}\bigg\|_{2}\\
&\lesssim \sum_{\ell=0}^{\infty}\sup\limits_{t\in[0,1]}\Big\|\sum_{k\in\mathbb{Z}}r_{k}(t)f\ast R^{\alpha}\ast A_{k+s}^{\alpha}K_{\alpha}^{0}\ast \Psi_{k-\ell}\Big\|_{2}
= \sum_{\ell=0}^{\infty}\sup\limits_{t\in[0,1]}\Big\|\sum_{k\in\mathbb{Z}}r_{k-s}(t)f\ast R^{\alpha}\ast A_{k}^{\alpha}K_{\alpha}^{0}\ast \Psi_{k-s-\ell}\Big\|_{2}\\
&\lesssim\sum_{\ell=0}^{\infty} 2^{-\tau(\ell+s)}\|\Omega\|_{L^{q}(\Sigma)}\|f\|_{2}\lesssim 2^{-\tau s}\|\Omega\|_{L^{q}(\Sigma)}\|f\|_{2}.
\end{align*}

Next we show that
\begin{align}\label{4546}
\Big\|\sup\limits_{k\in\mathbb{Z}}|f\ast R^{\alpha}\ast A_{k+s}^{\alpha}K_{\alpha}^{0}\ast (\delta_{0}-\Delta[2^{k}]\phi)|\Big\|_{L^{p}(w)}\lesssim \|\Omega\|_{L^{q}(\Sigma)}\{w\}_{A_{p}}\|f\|_{L^{p}(w)}.
\end{align}
To begin with, we note that
\begin{align}\label{supde}
\sup\limits_{k\in\mathbb{Z}}|f\ast R^{\alpha}\ast A_{k+s}^{\alpha}K_{\alpha}^{0}\ast (\delta_{0}-\Delta[2^{k}]\phi)(x)|
&\leq \sup\limits_{k\in\mathbb{Z}}|f\ast R^{\alpha}\ast A_{k+s}^{\alpha}K_{\alpha}^{0}(x)|\nonumber\\
&\quad+\sup\limits_{k\in\mathbb{Z}}|f\ast R^{\alpha}\ast A_{k+s}^{\alpha}K_{\alpha}^{0}\ast \Delta[2^{k}]\phi(x)|.
\end{align}
To continue, for the case $\alpha=0$, by \eqref{indd} with $k$ replaced by $k+s$,
\begin{align}\label{mnbj}
|A_{k+s}^{0}K_{0}^{0}(x)|\lesssim 2^{-\mathbb{Q}(k+s)}\|\Omega\|_{L^{\infty}(\Sigma)}\chi_{\rho(x)\leq 2^{k+s+2}},
\end{align}
which implies that
\begin{align}\label{ewr2}
|A_{k+s}^{0}K_{0}^{0}\ast \Delta[2^{k}]\phi(x)|\lesssim \|\Omega\|_{L^{\infty}(\Sigma)}\int_{\rho(y)\leq 2^{k}}2^{-\mathbb{Q}(k+s)}\chi_{d(x,y)\leq 2^{k+s+2}}|\Delta[2^{k}]\phi(y)|dy.
\end{align}
Moreover, if $0<\alpha<\mathbb{Q}$, then by \eqref{caaa1} and \eqref{caaa2} with $k$ replaced by $k+s$, we have
\begin{align}\label{mnbj2}
|R^{\alpha}\ast A_{k+s}^{\alpha}K_{\alpha}^{0}(x)|\lesssim 2^{-\mathbb{Q}(k+s)}\|\Omega\|_{L^{q}(\Sigma)}\chi_{\rho(x)\leq A_{0}^{3}\kappa_{[\alpha]}2^{k+s+5}}+\frac{2^{([\alpha]+1-\alpha)(k+s)}}{\rho(x)^{\mathbb{Q}+1+[\alpha]-\alpha}}\|\Omega\|_{L^{1}(\Sigma)}\chi_{\rho(x)\geq A_{0}^{3}\kappa_{[\alpha]}2^{k+s+5}},
\end{align}
which implies that
\begin{align}\label{jkl1}
&| R^{\alpha}\ast A_{k+s}^{\alpha}K_{\alpha}^{0}\ast \Delta[2^{k}]\phi(x)|\nonumber\\
&\lesssim \|\Omega\|_{L^{q}(\Sigma)}\int_{\rho(y)\leq 2^{k}}2^{-\mathbb{Q}(k+s)}\chi_{d(x,y)\leq A_{0}^{3}\kappa_{[\alpha]} 2^{k+s+5}}|\Delta[2^{k}]\phi(y)|dy\nonumber\\
&\qquad+ \|\Omega\|_{L^{1}(\Sigma)}\int_{\rho(y)\leq 2^{k}}\frac{2^{([\alpha]+1-\alpha)(k+s)}}{d(x,y)^{\mathbb{Q}+1+[\alpha]-\alpha}}\chi_{d(x,y)>A_{0}^{3}\kappa_{[\alpha]} 2^{k+s+5}}|\Delta[2^{k}]\phi(y)|dy\nonumber\\
&\lesssim \|\Omega\|_{L^{q}(\Sigma)}2^{-\mathbb{Q}(k+s)}\chi_{\rho(x)\leq A_{0}^{4}\kappa_{[\alpha]}  2^{k+s+6}}(x)+\|\Omega\|_{L^{q}(\Sigma)}\frac{2^{([\alpha]+1-\alpha)(k+s)}}{\rho(x)^{\mathbb{Q}+1+[\alpha]-\alpha}}\chi_{\rho(x)\geq A_{0}^{2}\kappa_{[\alpha]} 2^{k+s+4}}(x).
\end{align}
In both cases, combining \eqref{mnbj} and \eqref{mnbj2}, we conclude that
\begin{align}\label{rr11}
\sup\limits_{k\in\mathbb{Z}}|f\ast R^{\alpha}\ast A_{k+s}^{\alpha}K_{\alpha}^{0}(x)|\lesssim \|\Omega\|_{L^{q}(\Sigma)}Mf(x).
\end{align}
Moreover, combining \eqref{ewr2} and \eqref{jkl1}, we conclude that
\begin{align}\label{rr22}
\sup\limits_{k\in\mathbb{Z}}|f\ast R^{\alpha}\ast A_{k+s}^{\alpha}K_{\alpha}^{0}\ast \Delta[2^{k}]\phi(x)|\lesssim \|\Omega\|_{L^{q}(\Sigma)}Mf(x).
\end{align}
Hence, it follows from \eqref{HLM}, \eqref{supde}, \eqref{rr11} and \eqref{rr22} that \eqref{4546} holds.

Now we back to the estimate of ${\rm I_{3}}$.
It follows from \eqref{4546} with $w$ replaced by $w^{1+\epsilon}$, where $\epsilon=\frac{1}{2}c_{\mathbb{Q}}/(w)_{A_{p}}$, that
\begin{align}\label{4547}
\Big\|\sup\limits_{k\in\mathbb{Z}}|f\ast R^{\alpha}\ast A_{k+s}^{\alpha}K_{\alpha}^{0}\ast (\delta_{0}-\Delta[2^{k}]\phi)|\Big\|_{L^{p}(w^{1+\epsilon})}\lesssim \|\Omega\|_{L^{q}(\Sigma)}\{w\}_{A_{p}}^{1+\epsilon}\|f\|_{L^{p}(w^{1+\epsilon})}.
\end{align}
Now interpolating between \eqref{makey} and \eqref{4547} with change of measures (\cite[Theorem 2.11]{SW}), we obtain that there exists a constant $\tau>0$ such that
\begin{align*}
\Big\|\sup\limits_{k\in\mathbb{Z}}|f\ast R^{\alpha}\ast A_{k+s}^{\alpha}K_{\alpha}^{0}\ast (\delta_{0}-\Delta[2^{k}]\phi)|\Big\|_{L^{p}(w)}\lesssim \|\Omega\|_{L^{q}(\Sigma)}2^{-\tau s/(w)_{A_{p}}}\{w\}_{A_{p}}\|f\|_{L^{p}(w)}.
\end{align*}
Therefore,
\begin{align*}
{\rm III}
&\leq \sum_{s=1}^{\infty}\Big\|\sup\limits_{k\in\mathbb{Z}}\Big|f\ast R^{\alpha}\ast A_{k+s}^{\alpha}K_{\alpha}^{0}\ast(\delta_{0}-\Delta[2^{k}]\phi)\Big|\Big\|_{L^{p}(w)}\\
&\leq \|\Omega\|_{L^{q}(\Sigma)}\sum_{s=1}^{\infty}2^{-\tau s/(w)_{A_{p}}}\{w\}_{A_{p}}\|f\|_{L^{p}(w)}\leq \|\Omega\|_{L^{q}(\Sigma)}\{w\}_{A_{p}}(w)_{A_{p}}\|f\|_{L^{p}(w)}.
\end{align*}

Finally, by combining the estimates of ${\rm I}$, ${\rm II}$ and ${\rm III}$, the proof of Theorem \ref{main2} is complete.
\hfill $\square$
\bigskip

\bigskip

 \noindent
 {\bf Acknowledgements:}

The authors would like to express their sincere gratitude to the referees for their careful reading, patient reviewing, valuable corrections and constructive comments, which greatly improved the exposition of manuscript. The authors would like to thank Kangwei Li for helpful discussions on the sparse domination for maximal operators. Z. Fan would also like to thank Prof. Lixin Yan for useful discussions.  Y. Chen is supported by the National Natural Science Foundation of China, Grant No.  ~11871096 and ~11471033.
J. Li is supported by the Australian Research Council (ARC) through the
research grant DP220100285.


\begin{thebibliography}{99999}

%
%
%
%
%


\bibitem{Liebook} A. Bonfiglioli, E. Lanconelli and F. Uguzzoni, Stratified Lie Group and Potential Theory for their Sub-Laplacian, {\it Springer}. (2007).




\bibitem{CZ0} A.P. Calder\'{o}n and A. Zygmund,  On the existence of certain singular integrals,
       {\it Acta Math}. {\bf 88} (1952),  85--139.

\bibitem{CZ} A.P. Calder\'{o}n and A. Zygmund, On singular integrals,
    {\it Amer. J. Math}. {\bf 78} (1956),  289--309.

\bibitem{CD} Y. Chen and Y. Ding, $L^{p}$ bounds for the commutators of singular integrals and maximal singular integrals with rough kernels,  {\it Trans. Amer. Math. Soc}. {\bf 367} (2015), 1585--1608.

\bibitem{CFY}  J. Chen, D. Fan and Y. Ying, Certain operators with rough singular kernels, {\it Canad. J. Math}. {\bf 55} (2003),  504--532.



\bibitem{CZ2} Q. Chen and Z. Zhang, Boundedness of a class of super singular integral operators and the associated commutators, {\it Sci. China Ser}. {\bf 47} (2004), 842--853.

\bibitem{Christ1} M. Christ, Weak type (1,1) bounds for rough operators,
    {\it Ann. of Math}. {\bf 128} (1988),  19--42.



\bibitem{Christ2} M. Christ and J.L. Rubio de Francia, Weak type (1,1) bounds for rough operators, II,
    {\it Invent. Math}. {\bf 93} (1988),  225--237.






\bibitem{CCDO} J.M. Conde-Alonso, A. Culiuc, F. Di Plinio and Y. Ou, A sparse domination principle for rough singular integrals, {\it Anal. PDE}. {\bf 10} (2017),  1255--1284.

\bibitem{CSpde} L. Caffarelli and L. Silvestre, An Extension Problem Related to the Fractional
Laplacian. {\it Comm. Partial Differential Equations}. {\bf 32} (2007), 1245--1260.

\bibitem{CJpde} D. Chamorro and O. Jarr\'{i}n, Fractional Laplacians, Fractional Laplacians, extension problems and Lie groups. {\it C. R. Math. Acad. Sci. Paris}. {\bf 353} (2015), 517--522.



\bibitem{DHL} F. Di Plinio, T.P. Hyt\"onen and K. Li, Sparse bounds for maximal rough singular integrals via the Fourier transform, {\it Annales de l'Institut Fourier}. To appear.

\bibitem{DL} Y. Ding and X. Lai, Weak type (1,1) bound criterion for singular integrals with rough kernel and its applications,  {\it Trans. Amer. Math. Soc}. {\bf 371} (2019),  1649--1675.

%



\bibitem{JDbook} J. Duoandikoetxea, Fourier Analysis, Graduate studies in Mathematics, vol 29,
{\it American Mathematical Society, Providence, RI}, (2001).



\bibitem{DR} J. Duoandikoetxea and J.L. Rubio de Francia, Maximal and singular integral operators via Fourier transform estimates,
    {\it Invent. Math}. {\bf 84} (1986),  541--561.


\bibitem{DGKLWY} X.T. Duong, R. Gong, M.S. Kuffner, J. Li, B.D. Wick and D. Yang, Two weight commutators on spaces of homogeneous type and applications,  J. Geom. Anal., {\bf31} (2021), no. 1, 980--1038.


\bibitem{FP} D. Fan and Y. Pan, Singular integrals operators with rough kernels supported by subvarieties,  {\it Amer. J. Math}. {\bf 119} (1997),  799--839.



\bibitem {FoSt} G.B. Folland and E.M. Stein, Hardy Spaces on Homogeneous Groups, {\it Princetion University Press, Princeton, N. J}. (1982).

\bibitem{GS} L. Grafakos and A. Stefanov, $L^{p}$ bounds for singular integrals and maximal singular integrals with rough kernels, {\it Indiana Univ. Math. J}. {\bf 47} (1998),  455--469.

\bibitem{Hof} S. Hofmann, Weak (1,1) boundedness of singular integrals with nonsmooth kernel, {\it Proc. Amer. Math. Soc}. {\bf 103} (1988),  260--264.

\bibitem{H} T.P. Hyt\"onen, The sharp weighted bound for general Calder\'{o}n-Zygmund operators, {\it Ann. of Math}. {\bf 175} (2012),  1473--1506.


\bibitem{HP} T.P. Hyt\"onen and C. P\'{e}rez, Sharp weighted bounds involving $A_1$, {\it Anal. PDE} {\bf 6} (2013), 777-818.


\bibitem{HRT} T.P. Hyt\"onen, L. Roncal and O. Tapiola, Quantitative weighted estimates for rough homogeneous singular integrals, {\it Israel J. Math}. {\bf 218} (2017),  133--164.






\bibitem{La}
     M.T. Lacey, An elementary proof of the $A_{2}$ bound, {\it Israel J. Math}. {\bf 217} (2017).



\bibitem{Lerner1}
A.K. Lerner, On an estimate of Calder\'{o}n-Zygmund operators by dyadic positive operators, {\it J. Anal. Math}. {\bf 121} (2013), 141--161.



\bibitem{Lerner2}
A.K. Lerner, A simple proof of the $A_{2}$ conjecture, {\it Int. Math. Res. Not. IMRN}. (2013), 3159--3170.


\bibitem{Ler1}
A.K. Lerner, On pointwise estimates involving sparse operators, {\it New York J. Math}. {\bf 22} (2016), 341--349.


\bibitem{Lerner4}
A.K. Lerner, A note on weighted bounds for rough singular integrals, {\it C. R. Acad. Sci. Paris, Ser. I}. {\bf 356} (2018), 77--80.


\bibitem{LN}
A.K. Lerner and F. Nazarov, Intuitive dyadic calculus:  the basics, {\it Expo. Math}. {\bf 37} (2019), 225-265.



\bibitem{LMW}
F. Liu, S. Mao and H. Wu, On rough singular integrals related to homogeneous mappings, {\it Collect. Math}. {\bf 67} (2016), 113--132.






\bibitem{Petermichl1} S. Petermichl, The sharp weighted bound for the Riesz transforms, {\it Proc. Amer. Math. Soc}. {\bf 136} (2008), 1237-1249.

\bibitem{Petermichl2} S. Petermichl, The sharp bound for the Hilbert transform on weighted Lebesgue spaces in terms of the classical $A_{p}$ characteristic, {\it Amer. J. Math}. {\bf 129} (2007), 1355-1375.

\bibitem{PV} S. Petermichl and A. Volberg, Heating of the Ahlfors-Beurling operator:
weakly quasiregular maps on the plane are quasiregular, {\it Duke Math. J}. {\bf 112} (2002), 281-305.


\bibitem{RW} F. Ricci and G. Weiss, A characterization of $H^{1}(\Sigma_{n-1})$, {\it Proc. Sympos. Pure Math}. {\bf 35} (1978), 289-294.

\bibitem{sato} S. Sato, Estimates for singular integrals on homogeneous groups, {\it J. Math. Anal. Appl}. {\bf 400} (2013), 311--330.


  \bibitem{Se} A. Seeger, Singular integral operators with rough convolution kernels, {\it J. Amer. Math. Soc}. {\bf 9} (1996),
    95--105.

\bibitem{SW} E.M. Stein and G. Weiss, Interpolation of operators with change of measures, {\it Trans. Amer. Math. Soc}. {\bf 87} (1958),
    159--172.

\bibitem{s93} E.M. Stein, Harmonic Analysis:
Real-variable Methods, Orthogonality, and Oscillatory Integrals,
{\it Princeton, NJ: Princeton University Press}, (1993).

\bibitem{Tao} T. Tao, The Weak-type (1,1) of $L{\rm log}L$ Homogeneous Convolution Operator, {\it Indiana Univ. Math. J}. {\bf 48} (1999), 1547-1584.

\bibitem{VSC} {N.T. Varopoulos, Saloff-Coste and T. Coulhon, Analysis and geometry on groups, {\it Cambridge Tracts in mathematics, 100. Cambridge University, Cambridge}, (1992).}


\end{thebibliography}
\end{document}